\documentclass{amsart}

\usepackage{amsmath}

\usepackage{euscript}
\usepackage{mathrsfs}
\usepackage{bbold}
\usepackage{amssymb}
\usepackage{mathabx}
\usepackage{tikz}
\usetikzlibrary{arrows}

\usepackage{float}

\RequirePackage{color}
\definecolor{myred}{rgb}{0.75,0,0}
\definecolor{mygreen}{rgb}{0,0.5,0}
\definecolor{myblue}{rgb}{0,0,0.65}

\RequirePackage{ifpdf}
\ifpdf
  \IfFileExists{pdfsync.sty}{\RequirePackage{pdfsync}}{}
  \RequirePackage[pdftex,
   colorlinks = true,
   urlcolor = myblue, 
   citecolor = mygreen, 
   linkcolor = myred, 
   pagebackref,
   bookmarksopen=true]{hyperref}
\else
  \RequirePackage[hypertex]{hyperref}
  \fi

  \RequirePackage{hyperref}

\RequirePackage{ae, aecompl, aeguill} 

\newcommand{\nc}{\newcommand} \newcommand{\renc}{\renewcommand}

    \def\CM{{\mathbb{C}}}
    \def\DM{{\mathbb{D}}}

  \def\hg{{\mathfrak h}}

  \def\mg{{\mathfrak m}}

    \def\PM{{\mathbb{P}}}
    \def\QM{{\mathbb{Q}}}

    \def\ZM{{\mathbb{Z}}}

    \def\OC{{\mathcal{O}}}

    \def\RC{{\mathcal{R}}}

\def\FS{{\EuScript F}}

\def\e{\varepsilon}

\def\XSS{{\mathscr{X}}}

\nc{\todo}[1]{ {\color{red}XXX #1 XXX}}

\def\to{\rightarrow}

\def\longto{\longrightarrow}

\def\onto{\twoheadrightarrow}
\nc{\triright}{\stackrel{[1]}{\to}}
\nc{\longtriright}{\stackrel{[1]}{\longto}}

\nc{\Hb}{H^\bullet}

\nc{\Br}{\mathcal{B}}
\nc{\HotRR}{{}_R\mathcal{K}_R}
\nc{\HotR}{\mathcal{K}_R}
\nc{\excise}[1]{}
\nc{\defect}{\text{df}}
\nc{\h}[1]{\underline{H}_{#1}}

\nc{\Ga}{\mathbb{G}_a} 
\nc{\Gm}{\mathbb{G}_m} 

\nc{\Perv}{{\mathbf{P}}}

\nc{\IH}{{\mathrm{IH}}}

\nc{\ic}{\mathbf{IC}}

\nc{\gl}{{\mathfrak{gl}}}
\renc{\sl}{{\mathfrak{sl}}}
\renc{\sp}{{\mathfrak{sp}}}

\renc{\Im}{\textrm{Im}}

\nc{\HBM}{H^{BM}}

 \DeclareMathOperator{\Hom}{Hom}

\DeclareMathOperator{\id}{id}

\DeclareMathOperator{\rank}{rank}

\DeclareMathOperator{\GL}{GL}

\nc{\St}{\mathrm{St}}
\nc{\rot}{\mathrm{rot}}
\nc{\ext}{\mathrm{ext}}
\nc{\Tilt}{\mathrm{Tilt}}
\nc{\gen}{\mathrm{gen}}
\nc{\Graph}{\mathrm{Graph}}

\newcommand{\into}{\hookrightarrow}

\nc{\simto}{\stackrel{\sim}{\to}}
\nc{\simfrom}{\stackrel{\sim}{\leftarrow}}

\nc{\gbmod}{\mathrm{-gmod-}}

\nc{\gmod}{\mathrm{-gmod}}

\nc{\Parity}{\mathrm{Parity}}

\nc{\mult}{\mathrm{mult}}

\nc{\Hecke}{\textrm{H}}

\nc{\geom}{\mathrm{geom}}
\nc{\decr}{\mathrm{decr}}
\nc{\ind}{\mathrm{ind}}
\nc{\hyp}{\mathrm{hyp}}

\newcommand{\Poly}{\mathrm{Poly}}

\newcommand{\Plucker}{\mathrm{\text{Pl\"ucker}}}

\nc{\source}{\mathrm{source}}
\nc{\expl}{\mathrm{expl}}
\nc{\Soe}{\mathrm{Soe}}
\nc{\Abe}{\mathrm{Abe}}
\nc{\diag}{\mathrm{diag}}

\nc{\fin}{\textrm{finite}}

\nc{\reflect}{\RC}
\nc{\Chi}{\XSS}
\nc{\pt}{\mathrm{pt}}
\nc{\odd}{\textrm{odd}}
\nc{\even}{\textrm{even}}

\nc{\weights}{\textrm{weights}}

\newtheorem{thm}{Theorem}[section]
\newtheorem{lem}[thm]{Lemma}

\newtheorem{prop}[thm]{Proposition}
\newtheorem{cor}[thm]{Corollary}
  
\newtheorem{conj}[thm]{Conjecture}

\theoremstyle{definition}

\newtheorem{ex}[thm]{Example}

\theoremstyle{remark}
\newtheorem{remark}[thm]{Remark}

\newcommand{\KL}{P}
\newcommand{\qKL}{\partial P}

\title[]{Towards combinatorial invariance for Kazhdan-Lusztig polynomials}
\author[]{Charles Blundell}
\address{DeepMind, London, UK}
\email{cblundell@google.com}
\author[]{Lars Buesing}
\address{DeepMind, London, UK}
\email{lbuesing@google.com}
\author[]{Alex Davies}
\address{DeepMind, London, UK}
\email{adavies@google.com}
\author[]{Petar Veli\v{c}kovi\'c}
\address{DeepMind, London, UK}
\email{petarv@google.com}
\author[]{Geordie Williamson}
\address{University of Sydney, Sydney, Australia.}
\email{g.williamson@sydney.edu.au}

\begin{document}

\begin{abstract}
  Kazhdan-Lusztig polynomials are important and mysterious objects in
representation theory. Here we present a new formula for their
computation for symmetric groups based on the Bruhat graph. Our
approach suggests a solution to the combinatorial invariance
conjecture for symmetric groups, a well-known conjecture formulated by Lusztig and Dyer in
the 1980s.
\end{abstract}

\maketitle


\section{Introduction}

Kazhdan-Lusztig polynomials are important polynomials associated to
pairs of elements $x,y$ in Coxeter groups. They appear throughout
representation theory and related fields. Lusztig (ca. 1983) and
Dyer (1987) independently conjectured that Kazhdan-Lusztig polynomials depend
only on the poset of elements between $x$ and $y$ in Bruhat
order. This is a fascinating conjecture. For example, it suggests that
Kazhdan-Lusztig polynomials are providing subtle
invariants of  Bruhat order.
This conjecture is known in several special cases
(\cite{DyerThesis,BrentiIH,Incitti,BCM,PatimoKL,BLN}, see
\cite{BrentiHistory} for an overview).
 
In this work we prove a new formula 
Kazhdan-Lusztig polynomials for symmetric groups. Our formula was
discovered whilst trying to understand certain
machine learning models trained to predict Kazhdan-Lusztig polynomials
from Bruhat graphs (see \cite{Nature}). The new formula suggests an
approach to the combinatorial invariance
conjecture for symmetric groups.

Traditionally, Kazhdan-Lusztig polynomials are computed
inductively based on the length. This means that in order to compute
$P_{x,y}$ one might need to know $P_{u,v}$ where $u, v$ lies
anywhere in $[\id, x]$. Our new formula allows
inductive calculation using only polynomials $P_{u,v}$ where $u, v$
belong to the Bruhat interval $[x,y]$---thus it ``stays in the Bruhat
interval''. Roughly speaking, we compute Kazhdan-Lusztig
polynomials via induction over the \emph{rank} of the symmetric group, rather
than the \emph{length} of a permutation.

As our new formula uses only information present in the Bruhat
interval $[x,y]$ and thus might be useful for
approaching the combinatorial invariance conjecture. However, the
formula requires the knowledge of the intersection of $[x,y]$ with 
a coset $S_{n-1}x$, which is not combinatorially invariant
information. It is not difficult to see that the intersection
\[
[x,y] \cap S_{n-1}x
\]
forms a sub-interval $[x,c] \subset [x,y]$. We observe that the embedding $[x,c] \subset [x,y]$ has
remarkable properties: the edges connecting any node $u \in
[x,c]$ to a node $v \notin [x,c]$ possess a lattice-like quality which we
call a \emph{hypercube cluster}.

Abstracting the notion of hypercube cluster leads to the general
notion of \emph{hypercube decomposition} of which our sub-interval $[x,c]
\subset [x,y]$ is an example. Hypercube decompositions appear to
provide a nice general way of understanding Bruhat intervals in
symmetric groups. We expect them to play a role in the solutions of
other problems.  In general, there may be many more
hypercube decompositions than come from intersections with cosets of
$S_{n-1}$. Remarkably, our formula appears to give the right answer for
any hypercube decomposition. We cojecture that this is always the
case. We have checked our conjecture on all Bruhat intervals up to $S_9$ (over a
million Bruhat intervals, and many more hypercube
decompositions!). Our conjecture implies the combinatorial invariance
conjecture in more precise form.

\subsection{Structure of this paper} We begin in \S\ref{sec:KL} with
an overview of Kazhdan-Lusztig polynomials and the combinatorial
invariance conjecture. This is introductory, and can be skipped over
by an experienced reader.
In
\S\ref{sec:formula} we introduce hypercube clusters, and state our 
formula and conjecture. This section is purely combinatorial. In
\S\ref{sec:proof} and \S\ref{sec:hypercube combinatorics} we prove our formula; here geometric machinery
(Schubert varieties, intersection cohomology, torus actions, weights \dots)
is needed.  In \S\ref{sec:conjecture} we outline a route to a proof
of our conjecture in general, and explain that our conjecture would
follow from the purity of a certain cohomology group, or the
surjectivity of restriction map map in cohomology.

\subsection{Acknowledgements} We are grateful to Erez Lapid,
Matthew Dyer, Peter Fiebig, Leonardo Patimo and Joel Kamnitzer for useful 
discussions and/or correspondence. We are particularly grateful to
Anthony Henderson for patiently listening to our incomplete explanations, and
numerous helpful suggestions.


\section{Kazhdan-Lusztig polynomials and combinatorial invariance} \label{sec:KL}

In this background section, we give some background on Coxeter groups,
Bruhat graphs, Kazhdan-Lusztig polynomials and the combinatorial
invariance conjecture. Excellent references for the following include
\cite{Humphreys, BrentiHistory,BjBr,SoeKL,soergelbook}.

\subsection{Coxeter groups and Kazhdan-Lusztig polynomials} 
\emph{Coxeter groups} are an important class of groups, which arose out of H.~S.~M.~
Coxeter's study of finite reflection groups in the 1930s. They are
characterised by a presentation via generators and relations. In this
presentation the generating set are called the \emph{simple
reflections}. An important example of a Coxeter group is the symmetric group $S_{n+1}$
consisting of all permutations of $0, 1, \dots, n$, with simple
reflections consisting of the set $S = \{ (i, i+1) \}$  of
adjacent transpositions.

In a seminal paper \cite{KL}, Kazhdan and Lusztig associated to any
pair of elements $x, y$ in a Coxeter group a polynomial with integer coefficients
\[
x,y \in W \mapsto \KL_{x,y} \in \ZM[q]
\]
known as the \emph{Kazhdan-Lusztig polynomial}. All that we say here
concerning their definition is that it is highly inductive; one works
one's way ``out'' in the group, starting at the identity and applying
generators from $S$. At each step in the calculation one might need
any of the previously computed polynomials. Thus, they are rather
cumbersome to calculate by hand, but it is not difficult to compute
billions of them on a computer with enough memory. (This is useful for
machine learning, as one often needs access to large data
sets.)\footnote{The reader who wants to get a feeling for
  Kazhdan-Lusztig polynomials is encouraged to experiment with Joel
  Gibson's LieVis software. For example, Kazhdan-Lusztig
  for an affine Weyl group of type $A_2$ are computed live 
  \href{https://www.jgibson.id.au/lievis/affine_weyl/\#(controls:!f,fullscreen:!t,labels:klpoly)}{here}.}

\subsection{The Bruhat graph}
To any Coxeter group one may associate its \emph{Bruhat graph}. For
the symmetric group, this is the graph with vertices corresponding to
all elements of $S_n$, and an edge joining $x$ and $y$ if and only if
they differ by multiplication by a transposition. (In other words, $x$
and $y$ are connected in the Bruhat graph if they agree on all but two
elements of $0, 1, \dots, n$.) Below, we will
denote the transposition that exchanges $i$ and $j$ by
$t_{(i,j)}$. The transpositions are precisely the conjugates of the
simple reflections in $S_n$.

The symmetric group has a natural
\emph{length function} given by the number of inversions:
\[
\ell(x) = \# \{ i < j \; | \; x(i) > x(j) \}.
 \]
We regard the length function as giving us a notion
of ``height'' on the Bruhat graph. This allows us to orient the edges of the Bruhat graph via decreasing length.
Figures \ref{fig:BruhatGraph23} and
\ref{fig:BruhatGraph4} give pictures of the Bruhat graph for $n =
2,3$ and $4$.\footnote{Throughout this paper we use \emph{string
    notation} for permutations. Thus $(2,0,3,1)$ (or often simply
  $2031$) denotes the permutation of $0, 1, 2$ and $3$ that sends $0 \mapsto 2$, $1\mapsto
  0$, $2\mapsto 3$ and $3 \mapsto  1$.}
 

\begin{figure}
\[\begin{array}{c}
    \includegraphics[scale=0.5]{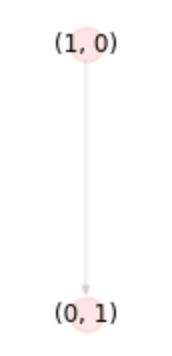}
\end{array} 
\qquad \text{and} \qquad
  \begin{array}{c}
    \includegraphics[scale=0.5]{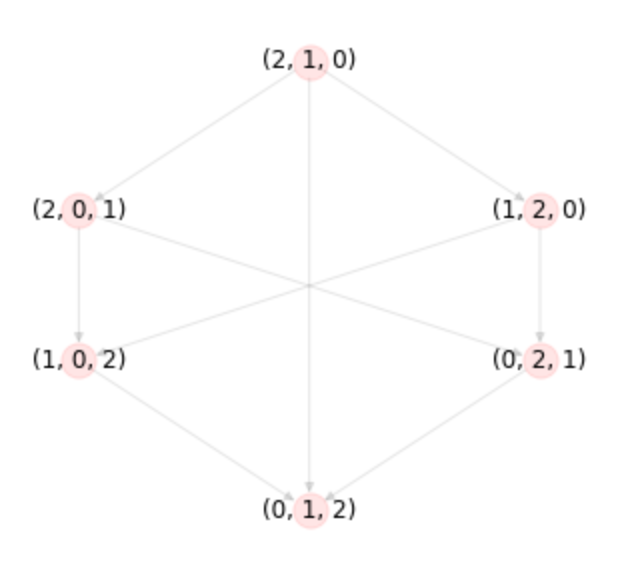}
   \end{array} \]
  \caption{Ordered Bruhat graphs for $n = 2, 3$}
  \label{fig:BruhatGraph23}
\end{figure}

\begin{figure}
    \includegraphics[scale=0.4]{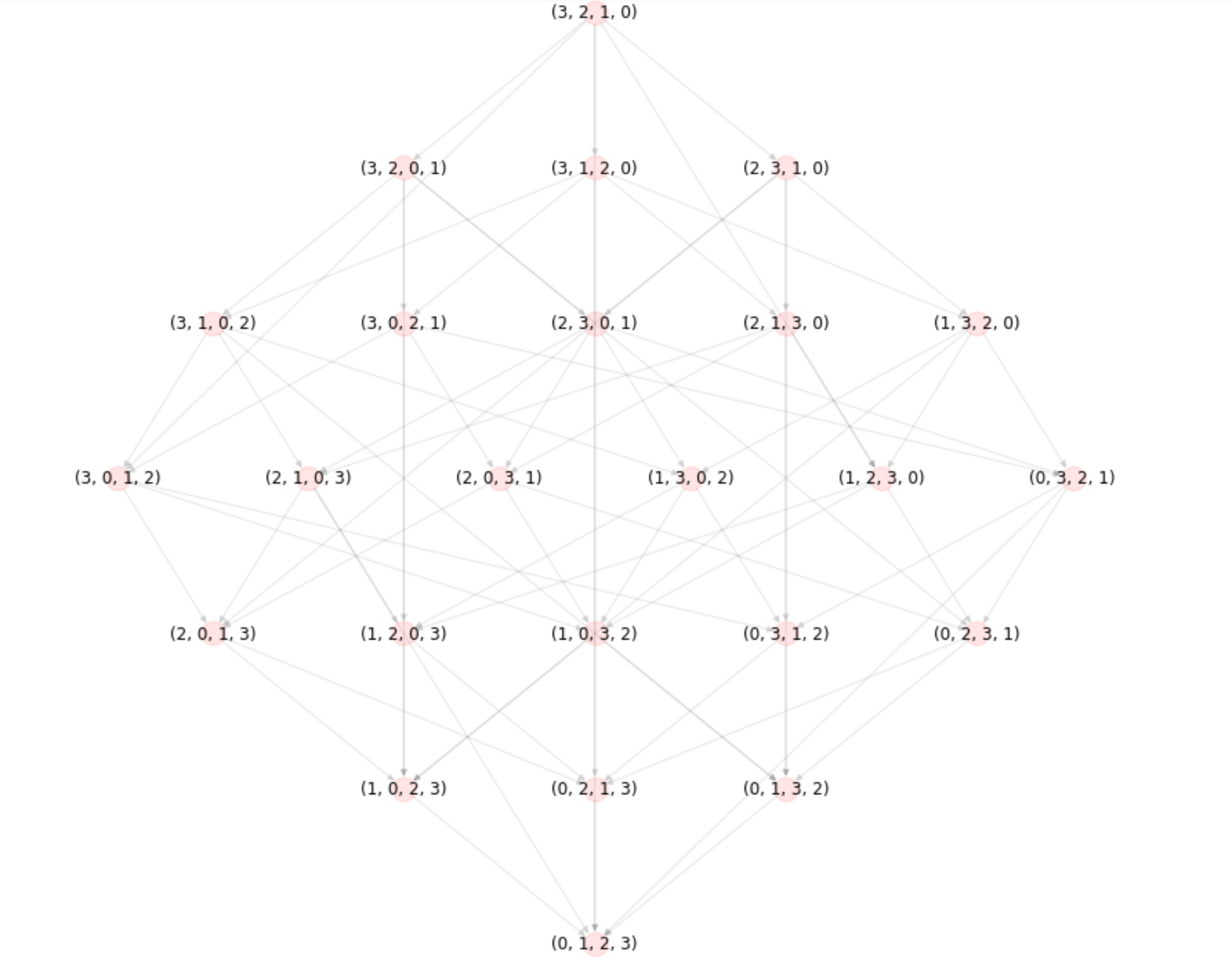}
  \caption{Ordered Bruhat graph for $n = 4$}
  \label{fig:BruhatGraph4}
\end{figure}


\subsection{Bruhat order}
The Bruhat graph allows us to define the \emph{Bruhat
  order}, which is a partial order on $W$. It is defined as follows:
\begin{equation}
  \label{eq:1}
x \le y \quad \Leftrightarrow \quad \begin{array}{c} \text{there exists a downward
  path}\\
\text{from $y$ to $x$ in the
Bruhat graph.} \end{array}
\end{equation}
(We include paths of length zero, so $x \le x$ always holds.)
For example, in the symmetric group $S_{n+1}$ the minimal element
is always the identity permutation, and the maximal element is the
permutation $w_0$ which interchanges $0$ and $n$, $1$ and $n-1$ etc.

The Bruhat order is remarkably complex, and it has long been suspected
that Kazhdan-Lusztig polynomials reflect subtle properties of the
Bruhat order. An elementary manifestation of this phenomenon (easy to prove) is that:
\[
\KL_{x,y} \ne 0 \Leftrightarrow x \le y.
\]
Less elementary connections tend to involve the interval $[x,y]$ consisting
of the full subgraph of the Bruhat graph between $x$ and $y$. (That
is, this consists of all edges and vertices which may be reached from
$y$ on the way to $x$, whilst progressing downwards.) For example, a
much less obvious fact \cite{Carrell, DyerDeodhar} is that
\[
  \KL_{x,y} = 1 \Leftrightarrow
  \begin{array}{c}\text{the graph obtained from $[x,y]$} \\
    \text{by forgetting edge orientations is regular} \\
    \text{(i.e. all vertices have the same degree).}
  \end{array}
\]
(It is easy to see that the full Bruhat graph is regular, and in
particular $\KL_{\id, w_0} = 1$.)

\subsection{First examples} 
In $S_3$ all proper intervals are isomorphic to the
following posets\footnote{poset = partially ordered set} (as the reader may check easily, using Figure 
\ref{fig:BruhatGraph23}):
\begin{gather*}
\begin{array}{c}
 \includegraphics[scale=0.2]{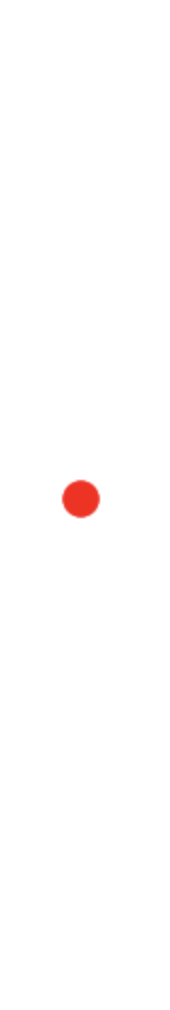}
\end{array}
\quad
\begin{array}{c}
 \includegraphics[scale=0.2]{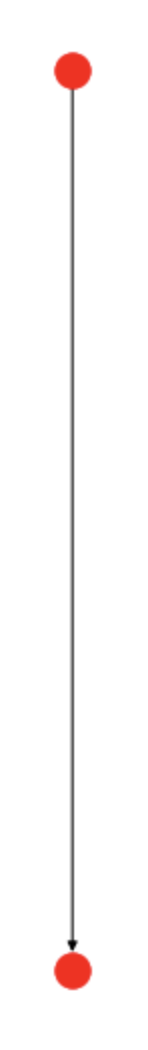}
\end{array}
\quad
\begin{array}{c}
 \includegraphics[scale=0.2]{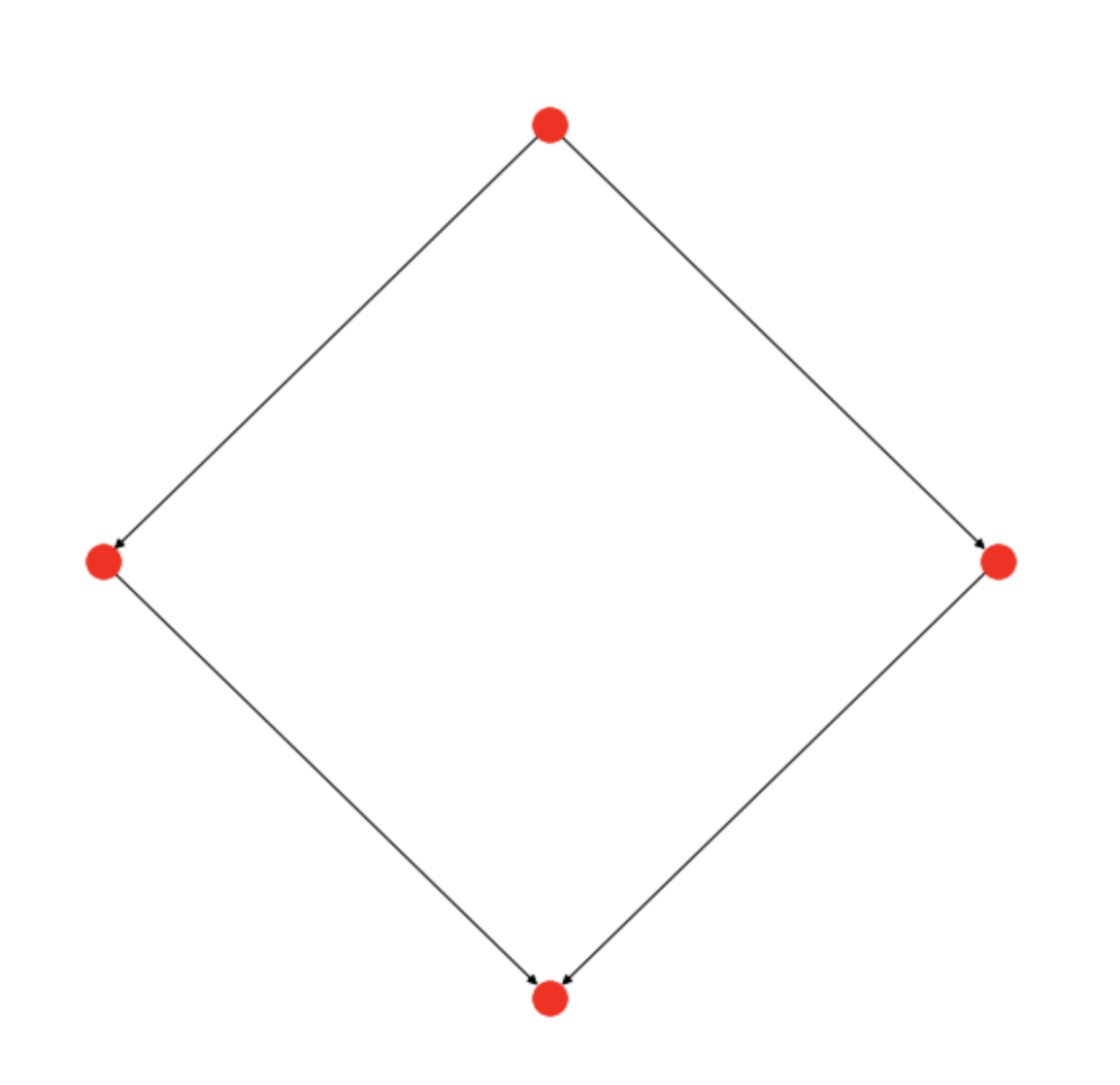}
\end{array}
\begin{array}{c}
 \includegraphics[scale=0.2]{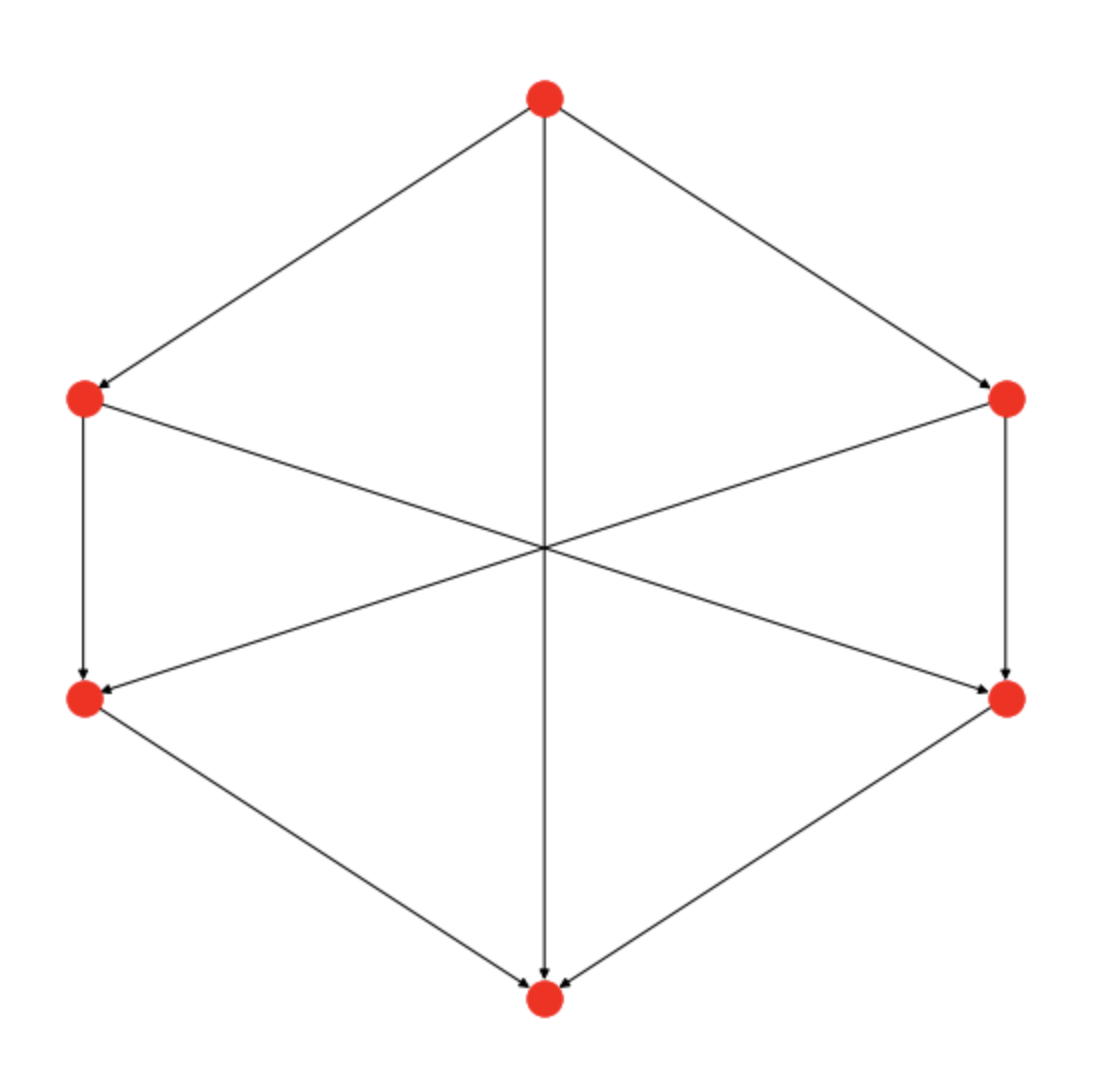}
\end{array}
\end{gather*}
All these graphs are regular (after forgetting edge orientations), and thus all
Kazhdan-Lusztig polynomials are $1$.

In $S_4$, almost all intervals $[x,y]$ are regular, and hence almost
all 
Kazhdan-Lusztig polynomials $\KL_{x,y}$ are 1. There are
four intervals which are not regular. Here we depict two intervals which are not: those
between $0213$ and $2301$, and between $1032$ and $3120$:
\begin{gather*}
\begin{array}{c}
 \includegraphics[scale=0.25]{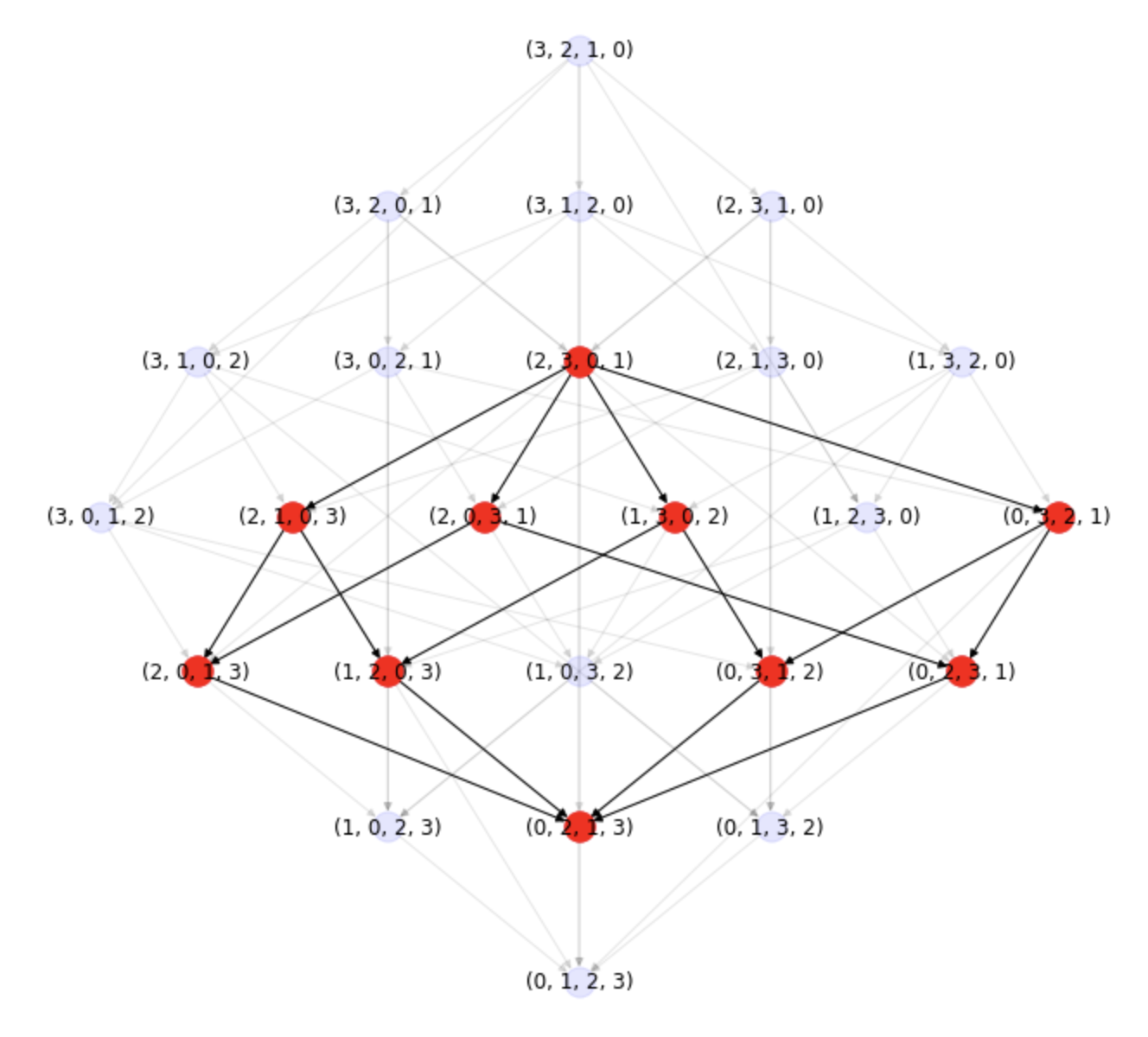}
\end{array}
\quad
\begin{array}{c}
 \includegraphics[scale=0.25]{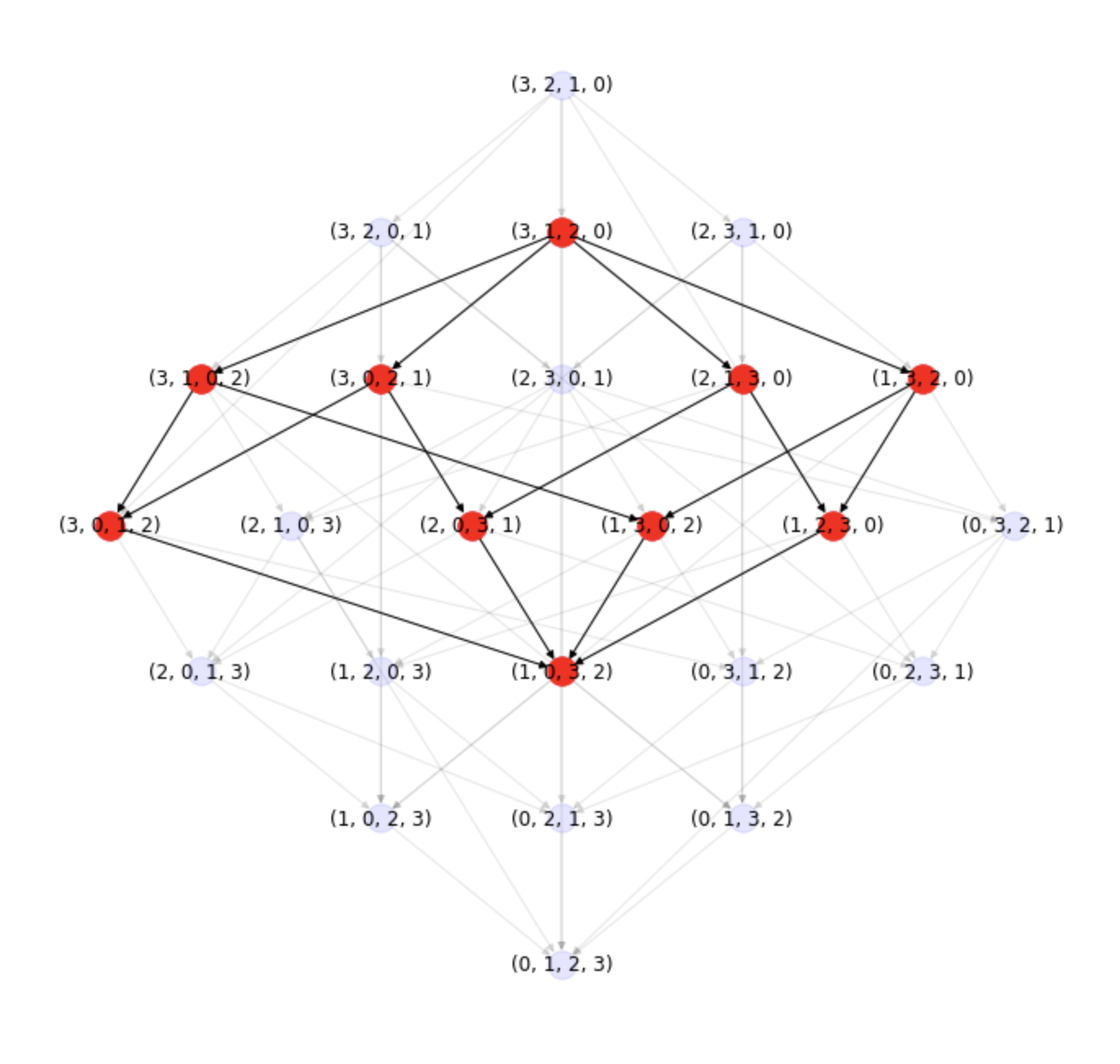}
\end{array}
\end{gather*}
In both cases, the interval is isomorphic to the following directed
graph, known as the ``4 crown'':
\begin{gather} \label{eq:4crown}
\begin{array}{c}
 \includegraphics[scale=0.2]{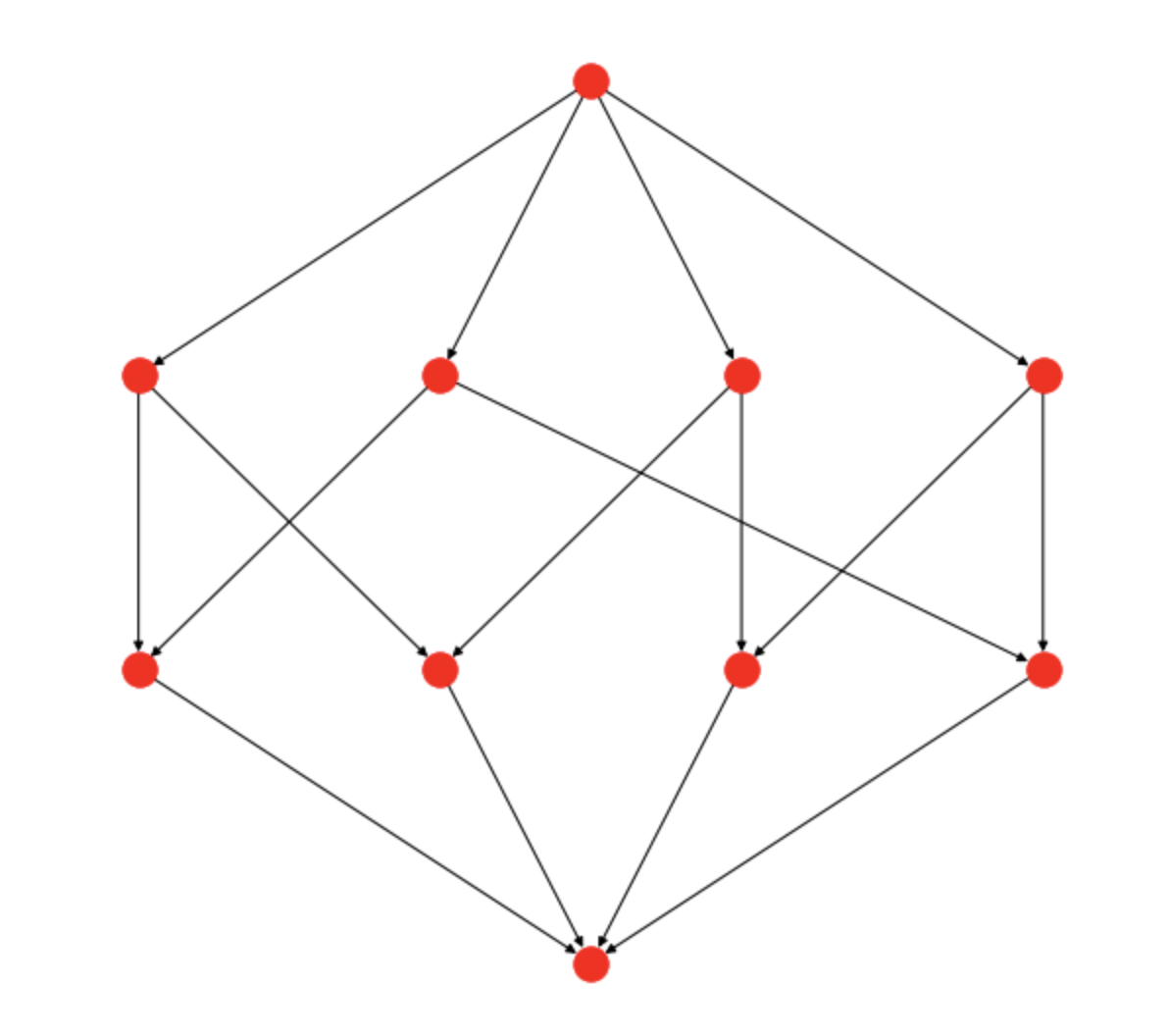}
\end{array}
\end{gather}
In both these cases the Kazhdan-Lusztig polynomials are equal:
\begin{equation}
  \label{eq:S4KL}
\KL_{0213,2301} = \KL_{1032,3120} = 1 + q.  
\end{equation}

\subsection{The combinatorial invariance conjecture} The following
conjecture, formulated independently by Lusztig and Dyer in
the 1980s, was a major motivation for the current work:

\begin{conj}
  The Kazhdan-Lusztig polynomial $\KL_{x,y}$ depends only on the
  isomorphism type of Bruhat graph of the interval $[x,y]$.
\end{conj}

For example, given only the ordered graph \eqref{eq:4crown} (and
not the labelling of its vertices) 
we should be able to predict the Kazhdan-Lusztig polynomial $1+
q$. The fact that \eqref{eq:4crown} occurs in two different ways in the 
Bruhat graph of $S_4$ with equal Kazhdan-Lusztig polynomials, can be seen as an instance of this conjecture. Figures \ref{fig:KL101} and  \ref{fig:KL131}  shows two more
examples of the assignment of  Kazhdan-Lusztig polynomial to Bruhat
intervals.

\begin{remark}
  The combinatorial invariance conjecture is a central conjecture in
  the study of Bruhat intervals. The reader is referred to
  \cite{BrentiHistory} for more detail on known cases (see also \cite{DyerThesis,BrentiIH,Incitti,PatimoKL,BLN}). We do not
  discuss the various partial results towards the conjecture here, except to mention that it
  is known to hold for intervals starting at the identity \cite{BCM}.
\end{remark}

\begin{figure}
\[
\begin{array}{c}
    \includegraphics[scale=0.3]{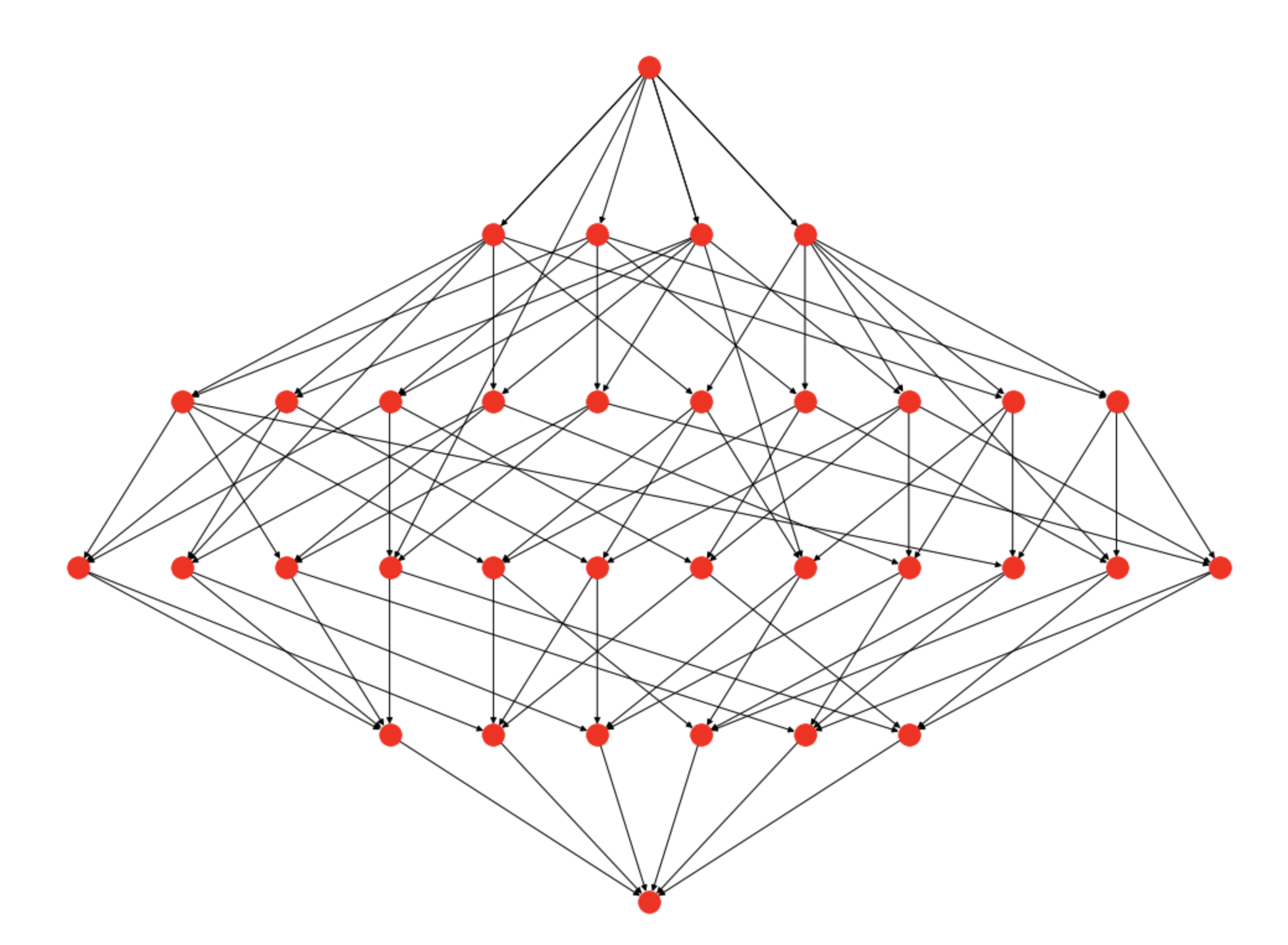}
\end{array}
\mapsto 1 + q^2
\]
  \caption{Interval and Kazhdan-Lusztig polynomial for $x = 03214$ and $34201$}
  \label{fig:KL101}
\end{figure}

\begin{figure}
\[
\begin{array}{c}
    \includegraphics[scale=0.3]{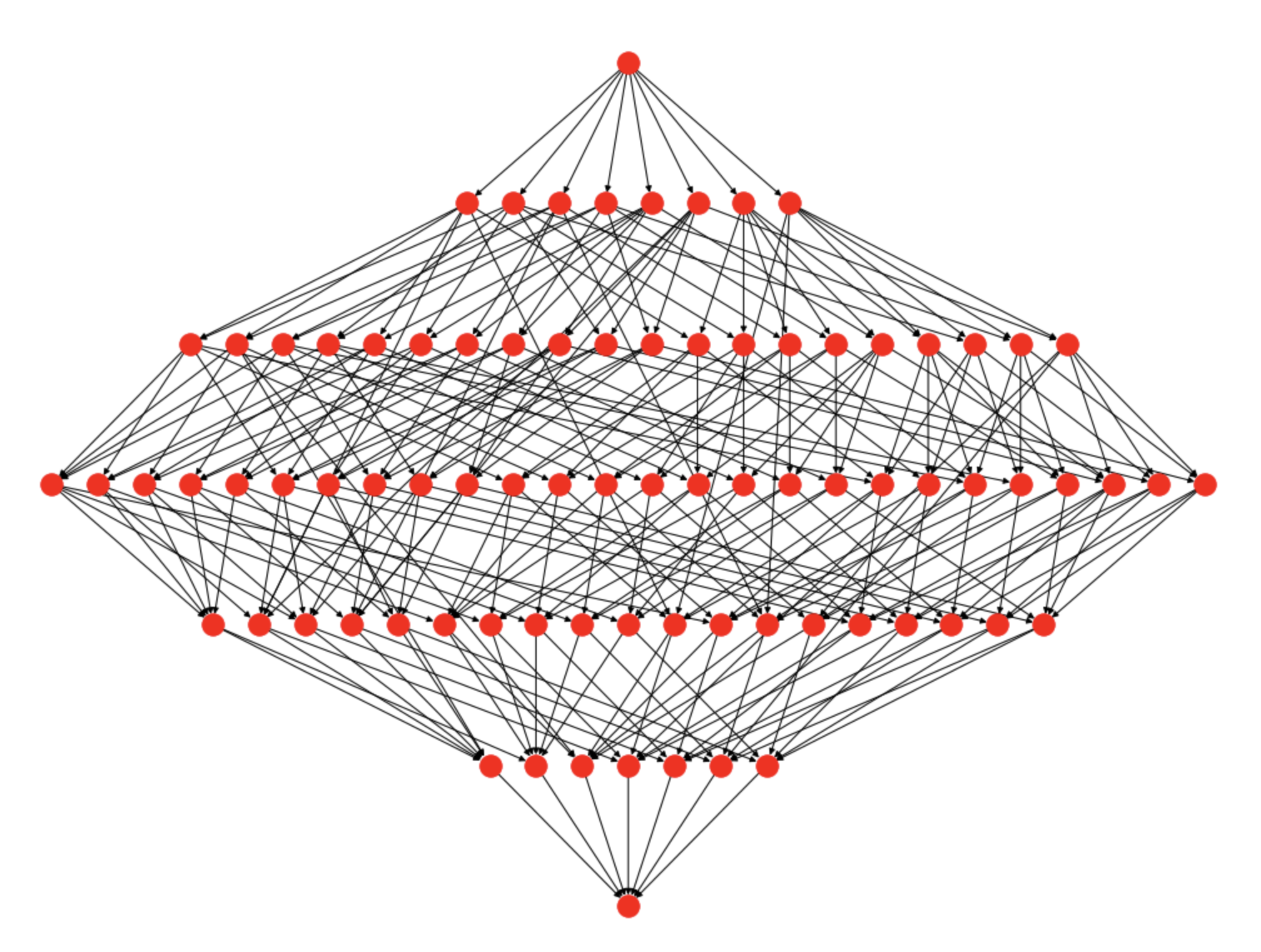}
\end{array}
\mapsto 1 + 3q + q^2
\]
  \caption{Interval and Kazhdan-Lusztig polynomial for $x = 021435$ and $y=234501$}
  \label{fig:KL131}
\end{figure}



\section{The new formula} \label{sec:formula}

In this section we describe our new formula. 
Before going into detail, let us give a rough idea of what the formula
looks like. Recall that our goal is to
compute the Kazhdan-Lusztig polynomial starting from the Bruhat
graph. By induction, we can assume that we can do this for any smaller
graph. In particular, we can assume that all intermediate polynomials
$\KL_{u,v}$ are known, for all $u,v \in [x,y]$ with $(u,v) \ne (x,y)$.

Our formula depends on the choice of an auxiliary structure on our
graph, called a \emph{hypercube decomposition}. Such a decomposition
amounts to the choice of a subinterval $J\subset [x,y]$ satisfying certain concrete
combinatorial conditions.  (The reader is
encouraged to skip ahead a few pages to Figure \ref{fig:hypercubes5},
where a typical hypercube decomposition is illustrated.) There always
exists at least one hypercube decomposition, but in general there will
be many. Any choice of hypercube decomposition determines two polynomials in $q$, the  
\emph{inductive piece} and \emph{hypercube piece}. Our formula is:
\begin{equation}
  \qKL_{x,y} = \text{\emph{inductive piece}} +\text{\emph{hypercube piece}}.
\end{equation}
The left
hand side is the $q$-derivative of the Kazhdan-Lusztig polynomial,
from which the Kazhdan-Lusztig polynomial can be recovered.
The calculation of the  \emph{inductive piece} (resp. \emph{hypercube piece}) uses only
the part of the graph which lies (resp. does not lie) in $J$. (Again,
the reader is encouraged to glance at
Figure \ref{fig:hypercubes5}. The nodes necessary for the computation
of the hypercube and inductive piece are in blue (resp. red).)

\subsection{The $q$-derivative of Kazhdan-Lusztig polynomials}
We now introduce a new polynomial, whose knowledge is equivalent to
knowledge of the Kazhdan-Lusztig
polynomial, but which is easier to handle. Define
\[
\qKL_{x,y}(q) = \frac{\KL_{x,y}(q) - q^{\ell(y)-\ell(x)}\KL_{x,y}(q^{-1})}{1- q}.
\]
(One checks easily that the denominator always divides the numerator,
so $\qKL_{x,y}$ is always an integer valued polynomial.)
We refer to $\qKL_{x,y}$ as the \emph{$q$-derivative of the
  Kazhdan-Lusztig polynomial}.\footnote{In the related theory of
  Kazhdan-Lusztig-Stanley polynomials, this polynomial is often called
  the ``$H$-polynomial''.}
Defining properties of
Kazhdan-Luztig polynomials\footnote{more precisely, the fact that their degree
is bounded above by $(\ell(y)-\ell(x)-1)/2$} ensure that $\qKL_{x,y}$
determines $\KL_{x,y}$.

\begin{ex}
  We give three examples of Kazhdan-Lusztig polynomials, and their
  corresponding $q$-derivatives:
  \begin{enumerate}
  \item $\KL_{x,y} = 1$, $\ell(y) - \ell(x) = 3$, $\qKL_{x,y} = 1 + q
    + q^2$,
    $\begin{array}{c}\tikz[scale=.3]{ \draw (0,0) rectangle (1,1);
      \draw (1,0) rectangle (2,1);
      \draw (2,0) rectangle (3,1);}\end{array}
  $.
    \item $\KL_{x,y} = 1 + q$, $\ell(y) - \ell(x) = 3$, $\qKL_{x,y} = 1 + 2q
    + q^2$,
    $\begin{array}{c}\tikz[scale=.3]{ \draw (0,0) rectangle (1,1);
       \draw (1,0) rectangle (2,1);
       \draw (1,1) rectangle (2,2);
      \draw (2,0) rectangle (3,1);}\end{array}
  $.
      \item $\KL_{x,y} = 1 + q^2$, $\ell(y) - \ell(x) = 6$, $\qKL_{x,y} = 1 + q
    + 2q^2 + 2q^3 + q^4+ q^5$,
    $\begin{array}{c}\tikz[scale=.3]{ \draw (0,0) rectangle (1,1);
       \draw (1,0) rectangle (2,1);
       \draw (2,0) rectangle (3,1);
       \draw (2,1) rectangle (3,2);
       \draw (3,0) rectangle (4,1);
       \draw (3,1) rectangle (4,2);
       \draw (4,0) rectangle (5,1);
       \draw (5,0) rectangle (6,1);}
     \end{array}
    $.
  \end{enumerate}
  (We leave it up to the reader to determine the meaning of the box
  diagrams, as well as how to recover the Kazhdan-Lusztig polynomial from them.)
\end{ex}

\subsection{Hypercube clusters} We work in the setting of
directed acyclic graphs. Any such graph is a poset in natural way,
where we declare $x \le y$ if there exists a directed path from $y$ to
$x$.

For any finite set $E$, the \emph{$E$-hypercube} $H_E$ is the directed acyclic graph with:
\begin{enumerate}
\item vertices consisting of subsets of $E$;
\item an edge $I \to J$ if $J$ is obtained from $I$ by removing one element.
\end{enumerate}

\begin{ex}
  $E$-hypercubes $H_E$ for $E = \{0 \}$, $\{ 0, 1\}$ and $\{0,1,2\}$:
  \[
    \begin{array}{c}
      \begin{tikzpicture}
        \node (0) at (0,1) {$\{ 0 \}$};
        \node (empty) at (0,0) {$\emptyset$};
        \draw[->] (0) to (empty);
        \end{tikzpicture} 
    \end{array}
\quad
       \begin{array}{c}
      \begin{tikzpicture}[scale=0.8]
        \node (01) at (0,2) {$\{ 0,1 \}$};
        \node (1) at (1,1) {$\{ 1 \}$};
        \node (0) at (-1,1) {$\{ 0 \}$};
        \node (empty) at (0,0) {$\emptyset$};
        \draw[->] (0) to (empty);
        \draw[->] (1) to (empty);
        \draw[->] (01) to (0); 
        \draw[->] (01) to (1);
        \end{tikzpicture} 
       \end{array}
       \quad
       \begin{array}{c}
      \begin{tikzpicture}[scale=0.8]
        \node (012) at (0,3) {$\{ 0,1,2 \}$}; 
        \node (01) at (-2,2) {$\{ 0,1 \}$};
        \node (02) at (0,2) {$\{ 0,2 \}$};
        \node (12) at (2,2) {$\{ 1,2 \}$};
        \node (2) at (2,1) {$\{ 2 \}$};
        \node (1) at (0,1) {$\{ 1 \}$};
        \node (0) at (-2,1) {$\{ 0 \}$};
        \node (empty) at (0,0) {$\emptyset$};
        \draw[->] (0) to (empty); \draw[->] (1) to (empty); \draw[->] (2) to (empty);
        \draw[->] (01) to (0); \draw[->] (01) to (1); 
        \draw[->] (02) to (0); \draw[->] (02) to (2); 
        \draw[->] (12) to (1); \draw[->] (12) to (2);
        \draw[->] (012) to (01); \draw[->] (012) to (02); \draw[->] (012) to (12);
        \end{tikzpicture} 
      \end{array}
    \]
  \end{ex}

Now suppose given a directed acyclic graph $X$ (in our applications
$X$ will be a Bruhat graph) and a node $x \in X$. We say that a subset
$E$ of edges with target $x$ 
\emph{spans a hypercube} if there exists a unique embedding (i.e. injection) of
directed graphs
\[
\vartheta : H_E \to X
\]
sending the edge $(\{ \alpha \} \to
\emptyset)$ in $H_E$ to $\alpha$, for all edges $\alpha$ in $E$. If $E$
spans a hypercube, its \emph{crown} is $\vartheta(E)$.

\begin{ex} \label{ex:S3}
  Consider the following directed graph (isomorphic to the
  Bruhat graph of $S_3$):
  \[
    \begin{array}{c}
      \begin{tikzpicture}[scale=0.6]
        \node (id) at (-90:2) {$\bullet$};
        \node (t) at (-30:2) {$\bullet$};
        \node (s) at (-150:2) {$\bullet$};
        \node (ts) at (30:2) {$\bullet$};
        \node (st) at (150:2) {$\bullet$};
        \node (sts) at (90:2) {$\bullet$};
        \draw[->] (sts) to (st); \draw[->] (sts) to (ts); \draw[->]
        (sts) to (id);
        \draw[->] (st) to (s); \draw[->] (st) to (t); 
        \draw[->] (ts) to (s); \draw[->] (ts) to (t);
        \draw[->] (s) to (id); \draw[->] (t) to (id); 
      \end{tikzpicture}
        \end{array}
    \]
    In the following we show some pairs of edges with common target
    with a colour. The pairs of red marked edges do not span hypercubes, whereas the
    blue edges do:
     \[
    \begin{array}{c}
      \begin{tikzpicture}[scale=0.6]
        \node (id) at (-90:2) {$\bullet$};
        \node (t) at (-30:2) {$\bullet$};
        \node (s) at (-150:2) {$\bullet$};
        \node (ts) at (30:2) {$\bullet$};
        \node (st) at (150:2) {$\bullet$};
        \node (sts) at (90:2) {$\bullet$};
        \draw[->] (sts) to (st); \draw[->] (sts) to (ts); \draw[red,->]
        (sts) to (id);
        \draw[->] (st) to (s); \draw[->] (st) to (t); 
        \draw[->] (ts) to (s); \draw[->] (ts) to (t);
        \draw[red,->] (s) to (id); \draw[->] (t) to (id); 
      \end{tikzpicture}
    \end{array}, 
        \begin{array}{c}
      \begin{tikzpicture}[scale=0.6]
        \node (id) at (-90:2) {$\bullet$};
        \node (t) at (-30:2) {$\bullet$};
        \node (s) at (-150:2) {$\bullet$};
        \node (ts) at (30:2) {$\bullet$};
        \node (st) at (150:2) {$\bullet$};
        \node (sts) at (90:2) {$\bullet$};
        \draw[->] (sts) to (st); \draw[->] (sts) to (ts); \draw[->]
        (sts) to (id);
        \draw[->] (st) to (s); \draw[->] (st) to (t); 
        \draw[->] (ts) to (s); \draw[->] (ts) to (t);
        \draw[red,->] (s) to (id); \draw[red,->] (t) to (id); 
      \end{tikzpicture}
        \end{array},
            \begin{array}{c}
      \begin{tikzpicture}[scale=0.6]
        \node (id) at (-90:2) {$\bullet$};
        \node (t) at (-30:2) {$\bullet$};
        \node (s) at (-150:2) {$\bullet$};
        \node (ts) at (30:2) {$\bullet$};
        \node (st) at (150:2) {$\bullet$};
        \node (sts) at (90:2) {$\bullet$};
        \draw[->] (sts) to (st); \draw[->] (sts) to (ts); \draw[->]
        (sts) to (id);
        \draw[blue,->] (st) to (s); \draw[->] (st) to (t); 
        \draw[blue,->] (ts) to (s); \draw[->] (ts) to (t);
        \draw[->] (s) to (id); \draw[->] (t) to (id); 
      \end{tikzpicture}
      \end{array}
    \]
    It should be clear that the first pair of red edges do not span a
    hypercube. In the second example, the issue is that 
    there is not a
    \emph{unique} way to map in a hypercube, with these specified base
    vertices.
  \end{ex}

  We now come to a key definition. Suppose that $E$ is a set of arrows
  with target $x$ as above. We say that $E$ \emph{spans a hypercube
    cluster} if every subset $E' \subseteq E$ consisting of arrows with
  pairwise incomparable sources spans a hypercube.\footnote{Elements
    $a$ and $b$ in a poset are \emph{incomparable} if neither $a \le b$
    nor $b \le a$ holds.}

  \begin{ex} Continuing the previous example, the pairs of blue arrows
    span hypercube clusters, whereas the red arrows do not:
      \[
    \begin{array}{c}
      \begin{tikzpicture}[scale=0.6]
        \node (id) at (-90:2) {$\bullet$};
        \node (t) at (-30:2) {$\bullet$};
        \node (s) at (-150:2) {$\bullet$};
        \node (ts) at (30:2) {$\bullet$};
        \node (st) at (150:2) {$\bullet$};
        \node (sts) at (90:2) {$\bullet$};
        \draw[->] (sts) to (st); \draw[->] (sts) to (ts); \draw[blue,->]
        (sts) to (id);
        \draw[->] (st) to (s); \draw[->] (st) to (t); 
        \draw[->] (ts) to (s); \draw[->] (ts) to (t);
        \draw[blue,->] (s) to (id); \draw[->] (t) to (id); 
      \end{tikzpicture}
    \end{array}, 
        \begin{array}{c}
      \begin{tikzpicture}[scale=0.6]
        \node (id) at (-90:2) {$\bullet$};
        \node (t) at (-30:2) {$\bullet$};
        \node (s) at (-150:2) {$\bullet$};
        \node (ts) at (30:2) {$\bullet$};
        \node (st) at (150:2) {$\bullet$};
        \node (sts) at (90:2) {$\bullet$};
        \draw[->] (sts) to (st); \draw[->] (sts) to (ts); \draw[->]
        (sts) to (id);
        \draw[->] (st) to (s); \draw[->] (st) to (t); 
        \draw[->] (ts) to (s); \draw[->] (ts) to (t);
        \draw[red,->] (s) to (id); \draw[red,->] (t) to (id); 
      \end{tikzpicture}
        \end{array},
            \begin{array}{c}
      \begin{tikzpicture}[scale=0.6]
        \node (id) at (-90:2) {$\bullet$};
        \node (t) at (-30:2) {$\bullet$};
        \node (s) at (-150:2) {$\bullet$};
        \node (ts) at (30:2) {$\bullet$};
        \node (st) at (150:2) {$\bullet$};
        \node (sts) at (90:2) {$\bullet$};
        \draw[->] (sts) to (st); \draw[->] (sts) to (ts); \draw[->]
        (sts) to (id);
        \draw[blue,->] (st) to (s); \draw[->] (st) to (t); 
        \draw[blue,->] (ts) to (s); \draw[->] (ts) to (t);
        \draw[->] (s) to (id); \draw[->] (t) to (id); 
      \end{tikzpicture}
      \end{array}
    \]
    In the first diagram, the sources of the blue arrows are
    comparable in $X$, so the condition to span a hypercube cluster reduces
    to each singleton spanning a hypercube, which is trivially the case.
  \end{ex}

  Suppose that $E$ spans a hypercube cluster. Then the edges in $E$
  form a partially ordered set by declaring that $\alpha \le \beta$ if
  there exists a directed path from the source of $\beta$ to that of
  $\alpha$. Given a subset $F \subset E$, we define $F_{\max}$ to be
  the subset of maximal elements with respect to this poset
  structure. Thus $F_{\max} \subset F$ consists of incomparable
  elements. Define the \emph{hypercube map}
\[
\theta : F \mapsto \text{crown of the hypercube spanned by
  $F_{\max}$.}
  \]
(Later we will see another rather different looking map arising from
geometry, and finally see that both are equal. Later we will sometimes
refer to $\theta$ as defined above as the \emph{combinatorial
  hypercube map}.)

  \subsection{Diamonds}
Given a directed
  graph $X$, a \emph{diamond} in $X$ is a subgraph isomorphic to
  \[
    \begin{array}{c}
      \tikz[scale=0.8]{
      \node (t) at (0,1) {$\bullet$};
\node (l) at (-1,0) {$\bullet$};
      \node (r) at (1,0) {$\bullet$};
      \node (b) at (0,-1) {$\bullet$};
      \draw[->] (t) to (l); \draw[->] (t) to (r);
      \draw[->] (l) to (b); \draw[->] (r) to (b);      }
    \end{array}
  \]
  A full subgraph $J \subset X$ is \emph{diamond
      complete}\footnote{This notion is due to Patimo \cite{PatimoKL}.}
    if whenever
it contains two edges sharing a node, it contains the entire
diamond. In other words, for all diamonds in $X$, if $J$ contains the red edges in any of the
diagrams below, it necessarily contains the black edges as well:
  \[
    \begin{array}{c}
      \tikz[scale=0.8]{
      \node (t) at (0,1) {$\bullet$};
\node (l) at (-1,0) {$\bullet$};
      \node (r) at (1,0) {$\bullet$};
      \node (b) at (0,-1) {$\bullet$};
      \draw[red,->] (t) to (l); \draw[red,->] (t) to (r);
      \draw[->] (l) to (b); \draw[->] (r) to (b);      }
    \end{array},
    \begin{array}{c}
      \tikz[scale=0.8]{
      \node (t) at (0,1) {$\bullet$};
\node (l) at (-1,0) {$\bullet$};
      \node (r) at (1,0) {$\bullet$};
      \node (b) at (0,-1) {$\bullet$};
      \draw[red,->] (t) to (l); \draw[->] (t) to (r);
      \draw[red,->] (l) to (b); \draw[->] (r) to (b);      }
    \end{array}, 
    \begin{array}{c}
      \tikz[scale=0.8]{
      \node (t) at (0,1) {$\bullet$};
\node (l) at (-1,0) {$\bullet$};
      \node (r) at (1,0) {$\bullet$};
      \node (b) at (0,-1) {$\bullet$};
      \draw[->] (t) to (l); \draw[->] (t) to (r);
      \draw[red,->] (l) to (b); \draw[red,->] (r) to (b);      }
    \end{array}
    \]

      \subsection{Hypercube decompositions} Recall that $X = [x,y]$
      denotes the Bruhat graph of the interval between $x$ and $y$. The following is the most
      important definition of this work. 
    We say that a full subgraph $J \subset X$ is a \emph{hypercube decomposition} if
    \begin{enumerate}
    \item $J = \{ v \in X \; | \; v \le z\}$ for some $y \ne z \in X$, and
      $J$ is diamond complete;
     \item for all $v \in J$, the set $E = \{ \alpha : u \to v \; | \;
       u \notin J \}$ spans a hypercube cluster.
    \end{enumerate}

    \begin{ex} Continuing Example \ref{ex:S3}, here are some possible
      choices of $J$ (indicated by red):
  \[
    \begin{array}{c}
      \begin{tikzpicture}[scale=0.6]
        \node[red] (id) at (-90:2) {$\bullet$};
        \node (t) at (-30:2) {$\bullet$};
        \node (s) at (-150:2) {$\bullet$};
        \node (ts) at (30:2) {$\bullet$};
        \node (st) at (150:2) {$\bullet$};
        \node (sts) at (90:2) {$\bullet$};
        \draw[->] (sts) to (st); \draw[->] (sts) to (ts); \draw[->]
        (sts) to (id);
        \draw[->] (st) to (s); \draw[->] (st) to (t); 
        \draw[->] (ts) to (s); \draw[->] (ts) to (t);
        \draw[->] (s) to (id); \draw[->] (t) to (id); 
      \end{tikzpicture}
    \end{array},
        \begin{array}{c}
      \begin{tikzpicture}[scale=0.6]
        \node[red] (id) at (-90:2) {$\bullet$};
        \node[red] (t) at (-30:2) {$\bullet$};
        \node (s) at (-150:2) {$\bullet$};
        \node (ts) at (30:2) {$\bullet$};
        \node (st) at (150:2) {$\bullet$};
        \node (sts) at (90:2) {$\bullet$};
        \draw[->] (sts) to (st); \draw[->] (sts) to (ts); \draw[->]
        (sts) to (id);
        \draw[->] (st) to (s); \draw[->] (st) to (t); 
        \draw[->] (ts) to (s); \draw[->] (ts) to (t);
        \draw[->] (s) to (id); \draw[red,->] (t) to (id); 
      \end{tikzpicture}
        \end{array},
            \begin{array}{c}
      \begin{tikzpicture}[scale=0.6]
        \node[red] (id) at (-90:2) {$\bullet$};
        \node[red] (t) at (-30:2) {$\bullet$};
        \node[red] (s) at (-150:2) {$\bullet$};
        \node[red] (ts) at (30:2) {$\bullet$};
        \node (st) at (150:2) {$\bullet$};
        \node (sts) at (90:2) {$\bullet$};
        \draw[->] (sts) to (st); \draw[->] (sts) to (ts); \draw[->]
        (sts) to (id);
        \draw[->] (st) to (s); \draw[->] (st) to (t); 
        \draw[red,->] (ts) to (s); \draw[red,->] (ts) to (t);
        \draw[red,->] (s) to (id); \draw[red,->] (t) to (id); 
      \end{tikzpicture}
        \end{array}
      \]
      Only the middle choice of $J$ constitutes a hypercube
      decomposition. In the example on the left, the edges arriving at
      the base vertex do not span a hypercube cluster. In the example
      on the right, $J$ is not diamond complete. Here is the incomplete diamond:
      \[
      \begin{array}{c}
           \begin{tikzpicture}[scale=0.6]
        \node[red] (id) at (-90:2) {$\bullet$};
        \node[red] (t) at (-30:2) {$\bullet$};
        \node[red] (s) at (-150:2) {$\bullet$};
        \node[red] (ts) at (30:2) {$\bullet$};
        \node (st) at (150:2) {$\bullet$};
        \node (sts) at (90:2) {$\bullet$};
        \draw[gray!30!white,->] (sts) to (st); \draw[gray!30!white,->] (sts) to (ts); \draw[gray!30!white,->]
        (sts) to (id);
        \draw[->] (st) to (s); \draw[->] (st) to (t); 
        \draw[red!20!white,->] (ts) to (s); \draw[red!20!white,->] (ts) to (t);
        \draw[red,->] (s) to (id); \draw[red,->] (t) to (id); 
      \end{tikzpicture}
    \end{array}
    \]
    \end{ex}
    
    \subsection{The hypercube piece} \label{sec:hypercube piece}
    We are now ready to define the
hypercube piece and inductive piece in our formula. We assume that $X$
is the Bruhat graph corresponding to the interval $[x,y]$ and that we
have fixed a hypercube decomposition $J \subset X$. The set
\[
  E = \{ \alpha: v \to x \; | \;  v \notin J \}
\]
spans a hypercube cluster by definition. In particular we have a
hypercube map:
\[
\theta : \text{subsets of $E$} \to X
  \]
We consider the polynomial:
\[
\widetilde{Q}_{x,y,J} = \sum_{\emptyset \ne I \subset E} (q-1)^{|I|-1} \KL_{\theta(I),y}
\in \ZM[q].
\]
(Note that we may assume that all terms on the right hand side are
known by induction.) We define the hypercube piece as follows:
\[
Q_{x,y,J} = q^{\ell(y)-\ell(x)-1} \widetilde{Q}_{x,y,J} (q^{-1}).
  \]

  \subsection{The inductive piece} \label{sec:inductive_combinatorial}
  Let $J \subset X$ and consider the
free $\ZM[q]$-module:
\[
M_J = \bigoplus_{ x \ne v \in J} \ZM[q] \delta_v.
  \]
This has a \emph{standard basis} $\{ \delta_v \; | \; v \in J \}$. If
we define
\[
b_v = \sum_{x \ne w \in J} \KL_{w,v} \cdot \delta_w
\]
then $\{b_v \; | \;  x \ne v \in J \}$ is also a basis for $M$, which
we call the \emph{Kazhdan-Lusztig basis}. This basis is
known by induction.

We may now define the inductive piece. Recall that $X$ is a Bruhat interval, with top node $y$. Define
\[
r_{x,y,J} = \sum_{x \ne v \in J} \KL_{v,y} \cdot  \delta_v  \in M.
\]
(In other words we consider all inductively computed Kazhdan-Lusztig
and ``restrict'' to $J$.) Now expand $r$ in the Kazhdan-Lusztig
basis:
\[
r_{x,y,J} = \sum_{x \ne v \in J} \gamma_v \cdot b_v.
  \]
  The inductive piece is defined as follows:
  \[
I_{x,y,J} = \sum_{x \ne v \in J}  \gamma_v \cdot \qKL_{x,v}.
    \]

    \subsection{A theorem} As above, $X = [x,y]$ denotes the Bruhat
    graph of the interval between $x$ and $y$.

    Suppose for a moment that we know the labelling of the nodes of
    $X$ by permutations. (Note that this is forbidden information in
    the combinatorial invariance conjecture.) In this case, consider
    the full subgraph
    \[
      L= \{ v \in X \; | \; v^{-1}(0) = x^{-1}(0) \} \subset X.
    \]

    \begin{remark}
  For any node $v \in L$, the ``hypercube edges'' (i.e. those edges
  with targed $v$ and source $\notin L$) are those edges
  corresponding to swapping $0$ and $i$ in $v$. These are precisely
  the edges which saliency analysis tell us are most important in our
  machine learning models (see Figure 3(a) in \cite{Nature}). This was
  our initial motivation for considering $L$.
\end{remark}
    
We have:

\begin{thm} \label{thm:main}
 $L \subset X$ is a hypercube decomposition, and we have:
 \[
\qKL_{x,y} = I_{x,y,L} + Q_{x,y,L}.
   \]
 \end{thm}

We will prove this formula in \S\ref{sec:proof} and \S\ref{sec:hypercube combinatorics}. This is a powerful new formula for
Kazhdan-Lusztig polynomials for symmetric groups, and should have
other applications. However, it does not solve the combinatorial
invariance conjecture, as the node labellings are needed to define
$L$. 

\subsection{A conjecture} 

Theorem \ref{thm:main} motivated us to consider more general
hypercube decompositions. Remarkably, it seems that this combinatorial
notion is exactly what is needed to make the above 
theorem hold:

\begin{conj} \label{conj:main} For any hypercube decomposition $J
  \subset X$ we have
  \[
\qKL_{x,y} = I_{x,y,J} + Q_{x,y,J}.
   \]
\end{conj}

Some remarks on this conjecture:
\begin{enumerate}
\item We have just seen that any interval admits a hypercube
  decomposition, and thus the
  conjecture implies the combinatorial invariance conjecture for
  symmetric groups.
  \item The conjecture is equivalent to the
statement that $I_{x,y,J} + Q_{x,y,J}$ is independent
of the choice of hypercube decomposition.
\item We have considerable computational evidence for this
  conjecture. It has been checked for all hypercube decompositions of
  all Bruhat intervals up to $S_7$, and over a million non-isomorphic
  intervals in $S_8$ and $S_9$.
\end{enumerate}

Positivity plays an important role in Kazhdan-Lusztig
theory. Remarkably, both pieces  $I_{x,y,J}$ and $Q_{x,y,J}$
in our formula should have positive coefficients:
\begin{enumerate}
\item The polynomials $Q_{x,y,J}$ have positive coefficients. (This
  can be shown directly, using the ``unimodality of Kazhdan-Lusztig
  polynomials'' \cite{Irving, BMP}.)
\item Conjecturally, the polynomials $\gamma_v$ involved in the computation
  of the inductive piece have positive coefficients. This has also
  been checked in all the cases mentioned above, and is true for the
  hypercube decomposition $L$ discussed above.
\end{enumerate}

\begin{remark}
It is interesting to ask whether our formula might solve the
combinatorial invariance conjecture for Coxeter groups other than
symmetric goups. It does not, as one can
see by inspecting the 5-crown:
\begin{equation*}
  \begin{tikzpicture}[yscale=.4,xscale=.8]
    \node (t) at (0,3) {$\bullet$};
    \node (t1) at (-2,1) {$\bullet$};     \node (t2) at (-1,1) {$\bullet$};     \node (t3) at
    (0,1) {$\bullet$};     \node (t4) at (1,1) {$\bullet$};     \node (t5) at (2,1) {$\bullet$};
        \node (b1) at (-2,-1) {$\bullet$};     \node (b2) at (-1,-1) {$\bullet$};     \node
        (b3) at (0,-1) {$\bullet$};     \node (b4) at (1,-1) {$\bullet$};     \node (b5) at
        (2,-1) {$\bullet$};
        \node (b) at (0,-3) {$\bullet$};
        \draw (t) -- (t1); \draw (t) -- (t2); \draw (t) -- (t3); \draw
        (t) -- (t4); \draw (t) -- (t5);
   \draw (t1) -- (b2) -- (t3) -- (b4) -- (t5) -- (b5) -- (t4) -- (b3)
   -- (t2) -- (b1) -- (t1);
           \draw (b) -- (b1); \draw (b) -- (b2); \draw (b) -- (b3); \draw
        (b) -- (b4); \draw (b) -- (b5);
  \end{tikzpicture}
\end{equation*}
This occurs as a Bruhat graph of an interval in the group $H_3$
of symmetries of the icosahedron (see
\cite[\S 2.8]{BjBr}). One can check
directly that it does not admit a hypercube decomposition.  
\end{remark}

    \subsection{Two worked examples}
We give the reader two examples of
our formula in action. These examples are
illustrated in Figures \ref{fig:hypercubes4} and
\ref{fig:hypercubes5}.

\begin{figure}[h]
 \includegraphics[scale=.4]{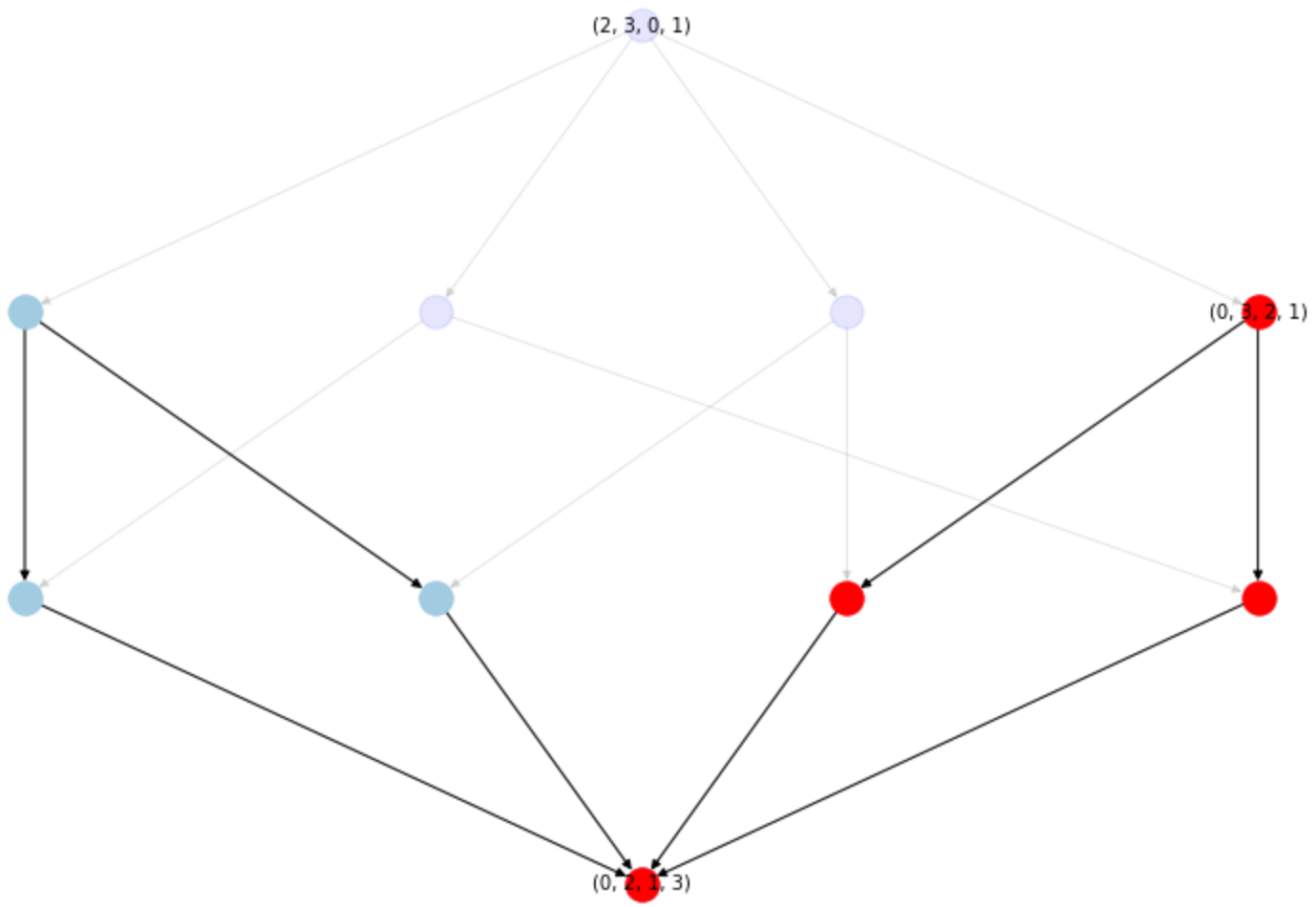}
  \caption{The Bruhat graph for the interval between $x =(0,2,1,3)$
    and $y = (2,3,0,1)$. The image of the hypercube map
    at $x$ is shaded blue, and the inductive piece is shaded red. All
    Kazhdan-Lusztig polynomials $P_{z,y}$ for $ z \ne x$ are 1. All hypercube
    decompositions of this interval are isomorphic to this one.}
  \label{fig:hypercubes4}
\end{figure}

\subsubsection{The interval between $x =(0,2,1,3)$
    and $y = (2,3,0,1)$} This is the first non-trivial
  example of a Kazhdan-Lusztig polynomial. In several respects this
  example is ``too simple'', but we discuss it anyway. The interval together with
  a choice of hypercube decomposition is illustrated in Figure
  \ref{fig:hypercubes4}. (The reader may check that in this
  example all hypercube decompositions are isomorphic.)

  We have
  \[
\KL_{x,y} = 1 + q \quad \text{and} \quad \qKL_{x,y} = 1 + 2q + q^2.
\]
In this case the polynomial $\widetilde{Q}_{x,y,J}$ is
\[
1 + 1 + (q-1) = 1 + q
\]
and hence the hypercube piece is
\[
q^2 ( 1+q^{-1}) = q + q^2.
\]
The inductive piece is
\[
1 + q
\]
and we indeed we have
\[
\qKL_{x,y}  = (1+ q) + (q + q^2).
\]

\begin{figure}[h]
 \includegraphics[scale=.4]{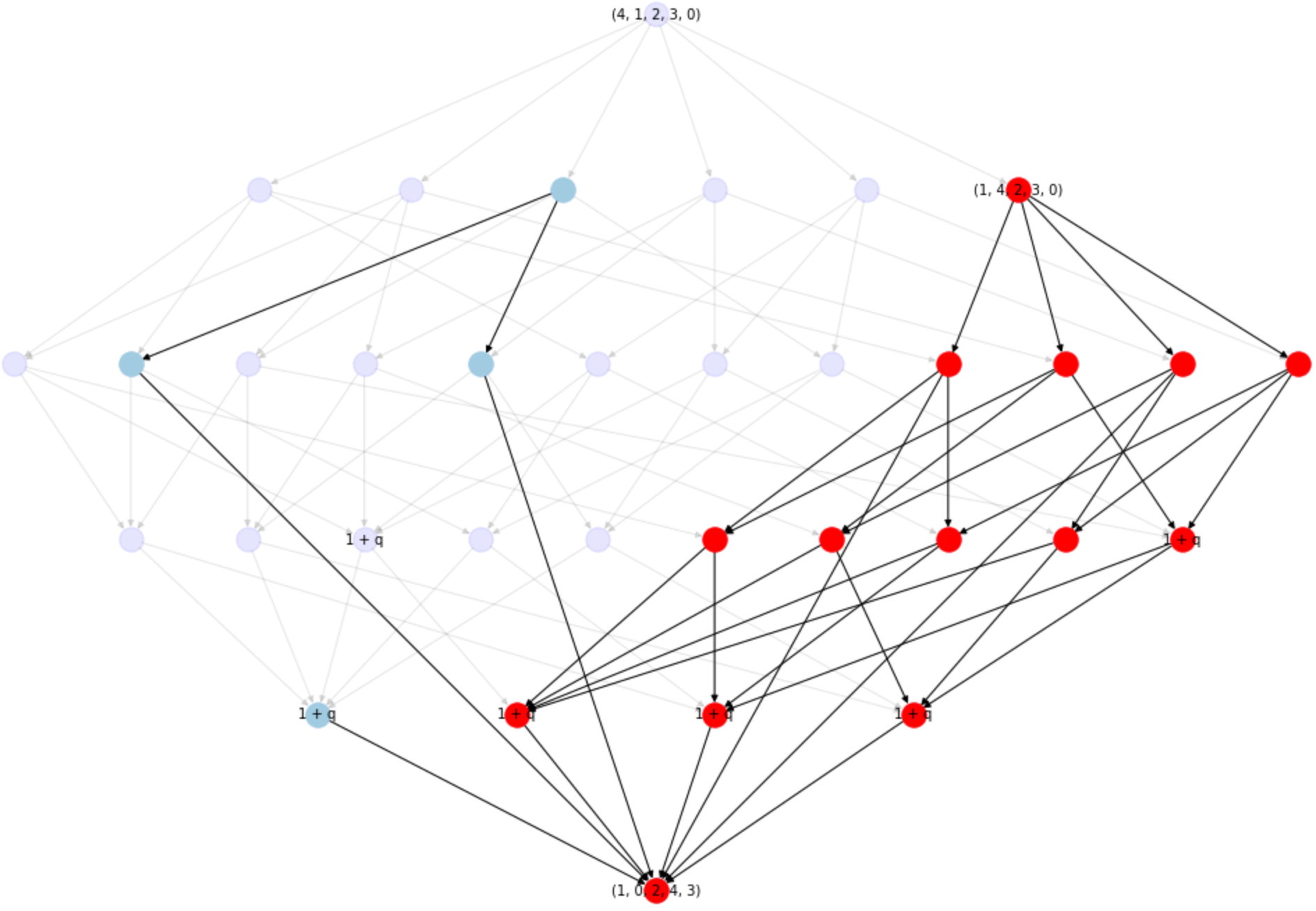}
  \caption{The Bruhat graph for the interval between $x =
    (1,0,2,4,3)$ and $y = (4, 1, 2, 3, 0)$ with a choice of hypercube
    decomposition. The image of the hypercube map 
    at $x$ is shaded blue, and the inductive piece is shaded red. All
    Kazhdan-Lusztig polynomials $P_{z,y}$ for $z \ne x$ are $1$ unless indicated.}
  \label{fig:hypercubes5}
\end{figure}

\subsubsection{The interval between $x =
  (1,0,2,4,3)$ and $y = (4, 1, 2, 3, 0)$.} This is a more interesting
case, and illustrates several features of the general case.  The interval together with
  a choice of hypercube decomposition $J$ is illustrated in Figure
  \ref{fig:hypercubes5}.

  The reader may check with a little work that we have
  \[
\widetilde{Q}_{x,y,J} = 1 + 2q + q^2
\]
and hence
\[
Q_{x,y,J} = q^{4}(1 + 2q^{-1} + q^{-2}) = q^2 + 2q^3 + q^4.
  \]
  We now turn to the inductive piece. We have
  \[
r_{x,y,J} = b_{14230} + q \cdot b_{10432}.
\]
(In Figure \ref{fig:hypercubes5}, 10432 is the only node of length
$\ell(x) + 2$ in the inductive piece with non-trivial Kazhdan-Lusztig
polynomial.) In particular,
\begin{align*}
I_{x,y,J} = \qKL_{x, 14230} + q \cdot \qKL_{x, 10432} &= 1 + 2q + 2q^2 + q^3
+ q(1 + q) \\ & = 1 + 3q + 3q^2 + q^3
\end{align*}
and we deduce correctly that
\[
\qKL_{x,y} = (1 + 3q + 3q^2 + q^3) + (q^2 + 2q^3 + q^4) = 1 + 3q +
4q^2 + 3q^2 + q^4.
\]
Or in other words
\[
P_{x,y} = 1 + 2q + q^2.
\]


  \section{Geometry of the new formula}   \label{sec:proof}

  In this section we explain the proof of Theorem \ref{thm:main}. This
  is a consequence of geometric techniques. We first recall
  fundamental results of Kazhdan and Lusztig and Bernstein and Lunts which
  connect Kazhdan-Lusztig polynomials to the geometry of Schubert
  varieties, and then proceed to the proof.

  \subsection{Geometric background} \label{sec:geo background}

  Given a complex $d$-dimensional complex algebraic variety $X$, we denote by
  $\ic(X,\QM)$ the intersection cohomology complex on $X$, with
  coefficients in $\QM$, normalized so that its restriction to $X$ is
  isomorphic to the constant sheaf \emph{without shift}.
  Given a sheaf $\FS$ of $\QM$ vector spaces on $X$ we denote by
 $H^\bullet(X,\FS)$ its hypercohomology (a graded $\QM$-vector
 space). The intersection cohomology  of $X$ is by definition
 \[
IH^\bullet(X,\QM) = H^\bullet(X, \ic(X,\QM)).
\]
Because of our conventions, the intersection
 cohomology of a complex $d$-dimensional variety is concentrated in
 degrees between $0$ and $2d$, as opposed to the more usual convention
 where it is concentrated in degrees between $-d$ and $d$.
  
 Given a graded $\QM$-vector space $H$ we define its \emph{Poincar\'e
 polynomial} to be
 \[
   \sum_{i \in \ZM} \dim H^i q^{i/2} \in \ZM[q^{\pm 1/2}].
   \]
Given a complex of $\QM$-vector spaces with finite-dimensional total
cohomology we define its Poincar\'e polynomial to be that of its
cohomology groups. In this paper, all Poincar\'e polynomnials we
consider are of vector spaces concentrated in positive even
degrees. In this case the Poincar\'e polynomail is a polynomial in
$q$.

Throughout this section, we will work in the equivariant derived
category (see \cite{BLu}). Throughout $T$ will denote an algebraic
torus acting algebraically on a variety $X$. By an equivariant sheaf
we mean an object of $D^b_T(X,\QM)$. All sheaves considered will be
constructible. On (equivariant) constructible categories we have the usual menagerie of
functors $f^*, f_*, f^!, f_!$ for equivariant maps $f : X \to Y$ which
all commute with the forgetful functor.

We will also need some arguments with weights. We will
exploit the fact that $\ic(X,\QM)$ is a constructible complex
underlying a mixed Hodge module, and hence so are all complexes
obtained from $\ic(X,\QM)$ via standard functors (see e.g.
\cite{Saito,Schnell}). In particular, each
cohomology group $IH^i(X,\QM)$ underlies a mixed Hoge structure, and
we can talk about its weights. Our normalizations are such that
$\ic(X,\QM)$ is pure of weight zero. When we say that a cohomology group $H^\bullet$ is
pure of weight $w$, we mean that $H^i$ is pure of weight $w + i$ for
all $i$. We can also talk about weights in
equivariant cohomology, by taking appropriate resolutions in the
category of algebraic varieities. Our demands on the theory of weights
are modest, and the reader who prefers to work with \'etale cohomology
and Frobenius weights will have no difficulty adapting the arguments.

\subsection{Tools of the trade} \label{sec:tools}
The following techniques will be used repeatedly:
\begin{enumerate}
\item (``Attractive Proposition'') Suppose that there exists a   $\CM^*$-action on $X$ 
  which \emph{retracts $X$ equivariantly onto $X^{\CM^*}$}. In other words, there
  exists a diagram
  \[
a : \CM \times X \to X
\]
such that the restriction of $a$ to $\CM^* \times X \to X$ agrees with
the action map, and such that $\{ 0 \} \times X$ maps into
$X^{\CM^*}$. (Roughly speaking, this says that the $\CM^*$-action
extends over $0$, attracting $X$ onto $X^{\CM^*}$.) We can consider
the inclusion and projection maps
\[
  \begin{tikzpicture}[xscale=1.2,yscale=.7]
 \node (l)  at (0,0) {$X^{\CM^*}$};   
 \node (r)  at (2,0) {$X$};
 \draw[->] (l) to[out=45,in=135] node[above] {$i$} (r);
 \draw[->] (r) to[out=-135,in=-45] node[below] {$p$} (l);
\end{tikzpicture}
\]
where $p : X = \{ 0 \} \times X \to X^{\CM^*}$ is induced by
$a$. Given any $\CM^*$-equvariant constructible sheaf $\FS$ on $X$, we
have canonical 
isomorphisms
\[
p_*\FS \simto i^*\FS \quad \text{and} \quad i^!\FS \simto p_!\FS. 
\]
We have similar isomorphisms of $X$ is a $T$-space, and $\CM^*$
commutes with the $T$-action. This technique is fundamental to
Geometric Representation Theory, for proofs and more detail see \cite{SpringerPurity},
\cite[Proposition 1]{Soergeln} and \cite[\S 2.6]{FW}.
\item (``Attractive Weight Argument'') Let us stay in the setting of
  the previous point. We have seen that we have an isomorphism
  \[  p_*\FS \simto i^*\FS. \]
  Now let us assume that $\FS$ is pure (of weight $0$, say). Then
  $i^*\FS$ is of weights $\le 0$, and $p_*\FS$ is of weight $\ge 0$
  because $i^*$ (resp. $p_*$) can only decrease (resp. increase)
  weights \cite[Proposition 1.7]{Saito}. Thus $i^*\FS$ is pure of
  weight $0$. An identical argument establishes that $i^!\FS$ is also
  pure of weight 0.
\item (``Finite Stabilizer Argument'') Suppose that
  $\CM^*$ acts on $X$ with finite stabilizers with (geometric) quotient
  $X/\CM^*$. Then we have a full embedding
  \[
    D^b(X/\CM^*, \QM) \to D^b_{\CM^*}(X,\QM)
\]
If $\FS$ and $\tilde{\FS}$ correspond under this embedding, we have a
canonical isomorphism
\[
H^\bullet_{\CM^*}(X,\FS) = H^\bullet(X/\CM^*,\tilde{\FS}).
\]
We have analogous statements for equivariant cohomology and
equivariant derived categories (i.e. when $X$ is a $T$-space). If we assume a free action, these facts are true very generally, and
are known as the quotient equivalence \cite[Proposition 2.2.5]{BLu}. When we allow
finite stabilizers the situation is a little more subtle, and coefficients of
characteristic zero are essential \cite[\S 9]{BLu}.
\end{enumerate}
  
  \subsection{Geometric interpretation of Kazhdan-Lusztig
    polynomials} Let $G = GL_{n+1}(\CM)$, $B \subset G$ the Borel subgroup of
  upper triangular matrices and $T$ the maximal torus of diagonal
  matrices. We identify the Weyl group with permutation matrices,
  which is canonically isomorphic to $S_{n+1}$. We have the flag variety $G/B$ and its
  Bruhat decomposition into Schubert cells
  \[
G/B = \bigsqcup_{y \in S_{n+1}} ByB/B.
\]
For any $y \in S_{n+1}$ we can consider the Schubert variety
 $\overline{ByB/B} \subset G/B$. On $\overline{ByB/B}$ we can consider
 the intersection cohomology complex $\ic(\overline{ByB/B}, \QM)$ (see
 \S\ref{sec:geo background} for conventions).  We have the following
 fundamental result of Kazhdan and Lusztig
 \cite{KLSchubert}:

 \begin{thm} \label{thm:kl}
   The cohomology sheaves $H^i(\ic(\overline{ByB/B}, \QM)_{xB/B})$ vanish
   in odd degree, moreover their Poincar\'e polynomial is given by the
   Kazhdan-Lusztig polynomial:
   \[
P_{x,y} = \sum_{i \ge 0} H^{2i}(\ic(\overline{ByB/B}, \QM) _{xB/B}) q^i.
     \]
 \end{thm}

 Let $U^-$ denote lower unitriangular matrices.
 For any $x \in S_{n+1}$ the product
 \[
BxB/B \times U^-xB/B
\]
is isomorphic to a $T$-stable affine neighbourhood of $x$ in
$G/B$. Similarly, 
\[
BxB/B \times (U^-xB/B \cap \overline{ByB/B})
\]
is isomorphic to a $T$-stable affine neighbourhood of $x$ in
$\overline{ByB/B}$. Fundamental to everything below will be the space
\[
S_{x,y} = U^-xB/B \cap \overline{ByB/B}
\]
which is a normal slice to the $B$-orbit through $xB/B$ in
$\overline{ByB/B}$. We denote by point $xB/B$ in $S_{x,y}$ simply by
$x$. Because $S_{x,y}$ is a normal slice, we deduce the
following from Theorem \ref{thm:kl}:

 \begin{prop} \label{prop:kl}
   The cohomology sheaves $H^i(\ic(S_{x,y}, \QM)_x)$ vanish
   in odd degree, moreover their Poincar\'e polynomial is given by the
   Kazhdan-Lusztig polynomial:
   \[
P_{x,y} = \sum_{i \ge 0} H^{2i}(\ic(S_{x,y}, \QM)_x) q^i.
     \]
   \end{prop}

   \begin{remark} There is no shift in the gradings between Theorem
     \ref{thm:kl} and Proposition \ref{prop:kl} because of our normalizations of $\ic$. \end{remark}

  \subsection{Equivariance and the fundamental example} Recall that $T$ denotes
  the maximal torus inside $GL_{n+1}$ consisting of diagonal
  matrices. Every $T$-fixed point $xB/B$ on $G/B$ is \emph{attractive} in
  the sense that there exists a one-parameter subgroup $\lambda: \CM^* \to T$
  in $T$ such that
  \[
\lim_{z \to 0} \lambda(z) \cdot v = xB/B
\]
for all $v$ in some neighbourhood of $xB/B$. (Equivalently and more
algebraically: there exists $\lambda$ as above which acts with
positive weights on functions on a $T$-stable affine neighbourhood of
$xB/B$.) It follows that $x \in S_{x,y}$ is attractive and that (with $\lambda$
as above)
\[
\lim_{z \to 0} \lambda(z) \cdot v = x
\]
for all $v \in S_{x,y}$.

Because $\lambda$ acts with positive weights the geometric
quotient
\[
\PM_\lambda \dot{S}_{x,y} := \dot{S}_{x,y} /_\lambda \CM^*
\]
exists and is a projective variety. Here ``$/_\lambda$'' means that we consider the quotient by the
$\CM^*$-action given by $\lambda$. (One may establish the existence
of this quotient by choosing a linear embedding $S_{x,y}$ inside an
affine space $\CM^n$ with linear $T$-action. Then $(\CM^n \setminus \{
0 \})/_\lambda \CM^*$ exists, and is isomorphic to a weighted
projective space. The chosen embedding $S_{x,y} \subset \CM^n$ then
provides a closed embedding of $\PM_\lambda \dot{S}_{x,y}$ inside a
weighted projective space.)

\begin{prop} \label{prop:qkl}
  The intersection cohomology groups $IH^i( \PM_\lambda \dot{S}_{x,y}, \QM))$ vanish
   in odd degree, moreover their Poincar\'e polynomial is given by the
   $q$-derivative of the Kahzdan-Lusztig polynomial:
   \[
\qKL_{x,y} := \sum_{i \ge 0} IH^{2i}( \PM_\lambda S_{x,y}, \QM))q^i.
     \]
   \end{prop}

   \begin{remark} The following proof is standard in Geometric
     Representation Theory. For a detailed discussion of closely related
     ideas, the reader might enjoy
     \cite[\S 14]{BLu}.\end{remark}

\begin{proof} 
Note that $\ic(S_{x,y},\QM)$ is equivariant for $T$, and hence may be
regarded as an element the equivariant derived category
$D^b_{\CM^*}(S_{x,y},\QM)$. In this proof we always regard $S_{x,y}$ as
a $\CM^*$-variety via the action induced from $\lambda : \CM^* \to T$.

Consider the closed/open decomposition
  \[
\{ x \} \stackrel{i}{\longrightarrow} S_{x,y}
\stackrel{j}{\longleftarrow} \dot{S}_{x,y}
\]
where $\dot{S}_{x,y} = S_{x,y} - \{ x \}$. This gives us a distinguished triangle
\[
i_!i^! \ic(S_{x,y},\QM) \to \ic(S_{x,y},\QM) \to
j_*j^*\ic(S_{x,y},\QM) \triright
\]
Taking global sections (i.e. hypercohomology) we get a long exact
sequence. We claim that all terms are pure of the same weight, and
hence that the connecting homomorphism is zero. Indeed:
\begin{enumerate}
\item Consider the projection $p : S_{x,y} \to \{x \}$. We have\[
    H^\bullet_{\CM^*}(\{x\}, p_* \ic(S_{x,y},\QM))   =
    H^\bullet_{\CM^*}(S_{x,y},  \ic(S_{x,y},\QM)) \]
  by definition, and
  \[H^\bullet_{\CM^*}(S_{x,y},  \ic(S_{x,y},\QM)) = H^\bullet_{\CM^*}(
    \{ x \}, i^*\ic(S_{x,y},\QM))\]by the Attractive Proposition. By
  the Attractive Weight Argument,
\[
  H^\bullet_{\CM^*}(S_{x,y},
  \ic(S_{x,y},\QM))\]is pure of weight zero as claimed.
\item By the Attractive Proposition again we have
  \[
    H^\bullet_{\CM^*}(\{x\}, p_! \ic(S_{x,y},\QM))    = H^\bullet_{\CM^*}(
  \{ x \}, i^!\ic(S_{x,y},\QM))\] and the purity of both sides follows
by the Attractive Weight Argument again.
\item Because $j^*$ is an open inclusion we have $j^*\ic(S_{x,y},\QM)
  = \ic(\dot{S}_{x,y},\QM)$. Because  $\CM^*$ acts with finite stabilizers and our
  coefficients are $\QM$ we have
  \[
    H^\bullet_{\CM^*}(\dot{S}_{x,y}, \ic(\dot{S}_{x,y},\QM)) =
    IH^\bullet( \PM_\lambda \dot{S}_{x,y}, \QM).
  \]
  The latter is pure of weight zero, as the intersection cohomology of a projective
  variety.
\end{enumerate}

Thus we have a short exact sequence
\begin{equation}
  \label{eq:gr_rk}
H^\bullet_{\CM^*}(
\{ x \}, i^!\ic(S_{x,y},\QM)) \into H^\bullet_{\CM^*}(\{x\},
  i^*\ic(S_{x,y},\QM)) \onto IH^\bullet( \PM_\lambda \dot{S}_{x,y}, \QM).  
\end{equation}
A similar argument to the one given above using the attractive
proposition shows that non-equivariant cohomology groups
\[
H^\bullet(
\{ x \}, i^!\ic(S_{x,y},\QM)) \quad \text{(resp.} \quad H^\bullet(
\{ x \}, i^*\ic(S_{x,y},\QM)) ) \]
are pure. Moreover, their Poincar\'e polynomials are given by
\[
q^{\ell(y) - \ell(x)} \KL_{x,y}(q^{-1}) \quad \text{(resp.} \quad \KL_{x,y}).
\]
(The second is simply a restatement of Proposition \ref{prop:kl}, the
first follows from it by Verdier duality.) It follows that the equivariant
cohomology groups are free over $H^*_{\CM^*}(pt,\QM)$ of graded ranks
given by
\[
q^{\ell(y) - \ell(x)} \KL_{x,y}(q^{-1})/(1-q) \quad \text{(resp.} \quad \KL_{x,y}/(1-q)).
\]
We deduce that the Poincar\'e polynomial of the last term in
\eqref{eq:gr_rk} is given by
\begin{gather*}
\KL_{x,y}/(1-q) - q^{\ell(y) - \ell(x)} \KL_{x,y}(q^{-1})/(1-q)  
= \qKL_{x,y}  \end{gather*}
which is what we wanted to show.
\end{proof}

\begin{ex} \label{ex:running} Consider the first singular Schubert variety $y =
  (2,3,0,1)$ and $x = (0,2,1,3)$. We will keep this as a running
  example throughout.
  One can compute directly (or wait until we discuss this example in
  more detail in Example \ref{ex:running3})
  that $S_{x,y}$ may be identified with the following subvariety of
  $4\times 4$-matrices:
  \begin{equation*}
  S_{x,y} = \left \{ \left ( \begin{matrix}
    1 & 0 & 0 & 0 \\ x & 0 & 1 & 0 \\ y & 1 & 0 & 0 \\ 0 & z & w &
    1 \end{matrix} \right ) \; \middle | \; xw-yz = 0 \right \}
\end{equation*}
We can choose $\lambda$ such that the attractive $\CM^*$-action simply
scales $x, y, z$ and $w$. With this choice of $\lambda$ the quotient
$\dot{S}_{x,y}/_\lambda \CM^*$ is simply
\[
\{ [x:y:z:w] \in \PM^3 \; | \; xz-yw = 0 \} \cong \PM^1 \times \PM^1.
  \]
This is smooth, thus its intersection cohomology and cohomology agree,
and its Poincar\'e polynomial is $1 + 2q + q^2 = \qKL_{x,y}$.
\end{ex}

\subsection{Where does the formula come from?} \label{sec:sketch}
We are now in a position to describe our formula in more detail. Here
we give an outline of the argument.

In the previous section we saw that the $q$-derivative of the
Kazhdan-Lusztig polynomial $\qKL_{x,y}$ computes the Poincar\'e
polynomial of the intersection cohomology of the projective variety
$\PM_\lambda \dot{S}_{x,y}$. Now $S_{x,y}$ has a natural
stratification via intersections with Schubert cells:
\[
S_{x,y} = \bigsqcup_{u \in [x,y]} S_{x,y}^u \quad \text{where
  $S_{x,y}^u = S_{x,y} \cap BuB/B$.}
\]
We will refer to this stratification as the \emph{Schubert
  stratification} of $S_{x,y}$. The Schubert stratification is
$T$-invariant and induces a
stratification of the quotient
\begin{equation} \label{eq:strat}
  \PM_\lambda \dot{S}_{x,y} = \dot{S}_{x,y}/_\lambda \CM^*= \bigsqcup_{u \in (x,y]} S_{x,y}^u/_\lambda \CM^*
\end{equation}
(Note that $[x,y]$ has been replaced by $(x,y]$ because the point $\{
x \} = S_{x,y}^x = S_{x,y} \cap BxB/B$ has been removed in
$\dot{S}_{x,y}$.)

The pullback of the IC sheaf on $\PM_\lambda
\dot{S}_{x,y}$ to $\dot{S}_{x,y}$ agrees with the IC sheaf on
$\dot{S}_{x,y}$, which in turn agrees with the restriction of the IC
sheaf on $\overline{ByB/B}$, because $S_{x,y}$ is a normal slice. In particular, all the Poincar\'e
polynomials of the stalks of the IC sheaf on $\PM_\lambda
\dot{S}_{x,y}$ are known. (To spell things out: the Poincar\'e
polynomial of $\ic( \PM_\lambda
\dot{S}_{x,y}, \QM)$ along the stratum $S_{x,y}^u/_\lambda \CM^*$ is
given by $P_{u,y}$.) Thus it is reasonable, perhaps, that we can
compute the Poincar\'e polynomial of the global sections of  $\ic( \PM_\lambda
\dot{S}_{x,y}, \QM)$ via some kind of long exact sequence.

We will define an open/closed decomposition
\begin{equation} \label{eq:openshut}
Z \stackrel{i}{\into} \PM_\lambda \dot{S}_{x,y}  \stackrel{j}{\hookleftarrow} U
\end{equation}
with $U$ open and $Z$ closed. We have an associated long exact
sequence:
\[
\to H^\bullet( Z, i^! \ic) \to H^\bullet(\PM_\lambda \dot{S}_{x,y}, \ic)
\to H^\bullet(U, j^*\ic) \to
\]
(where we have abbreviated $\ic := \ic(\PM_\lambda \dot{S}_{x,y},
\QM)$). The decomposition \eqref{eq:openshut} has several very
favourable properties:
\begin{enumerate}
\item $Z$ is isomorphic to a weighted projective space, the 
  stratification induced by \eqref{eq:strat} coursens its toric
  stratification\footnote{See \S\ref{sec:weightedproj} for the
    definition of ``toric stratification''}, and $i^! \ic$ is pure;
  \item $U$ retracts $T$-equivariantly onto a space $Z'$ isomorphic to $\PM_\lambda \dot{S}_{x,c}$,
where $c$ denotes the maximal element in $[x,y]$ in the $S_n$-coset of
$x$, and the restriction of $\ic$ to $Z'$ is pure.
\end{enumerate}

By the Attractive Proposition our long exact sequence can be
rewritten
\begin{equation} \label{eq:ses-sketch}
\to H^\bullet( Z, i^! \ic) \to H^\bullet(\PM_\lambda \dot{S}_{x,y}, \ic)
\to H^\bullet(U, (i')^*\ic) \to
\end{equation}
where $i' : Z' \into \PM_\lambda \dot{S}_{x,y}$ denotes the
inclusion. Because everything is pure of weight zero, we conclude that our sequence is in fact short
exact. Thus, taking Poincar\'e polynomials (and using that the
$q$-derivative of the KL polynomials is the Poincar\'e polynomial of
the middle term) we deduce
\[
\qKL_{x,y} = Q + I
  \]
where $Q$ (resp. $I$) is the Poincar\'e polynomial of the left-most
(resp. right-most) term in \eqref{eq:ses-sketch}. We can compute $Q$ via a simple
formula (yielding the hypercube piece), and $I$ is computed by
induction (yielding the inductive piece).

\begin{ex} \label{ex:running2}
  We continue our running example. We have seen that
  $\dot{S}_{x,y}/_\lambda \CM^*$ is isomorphic to $\PM^1 \times
  \PM^1$. The torus $T/\lambda(\CM^*) \cong \CM^* \times \CM^*$ still
  acts on this quotient, and one may check that this is the standard action, i.e.
  \[
    (z,z') \cdot ([u:v],[u':v']) = ([u:zv],[u':z'v']).
  \]
  There are $4 + 4 + 1 = 9$ orbits of $(\CM^*)^2$ on $\PM^1 \times
  \PM^1$, and these orbits coincide with the image of the Schubert
  stratification. (It is nice to notice that the closure patterns of these
  orbits nicely match the vertices $\ne x$ in the interval $[x,y]$:
  \begin{gather}
   \begin{array}{c}
 \includegraphics[scale=0.2]{pics/4crown}
\end{array}
\end{gather}
Thus the four vertices are height 1 correspond to the 4-fixed points
on $\PM^1 \times \PM^1$, the four vertices at height 2 correspond to
the $4$ invariant $\CM^*$'s, and the unique vertex at height 3 corresponds
to the open orbit.)

Our decomposition has $Z$ equal to a torus invariant
$\PM^1 \subset \PM^1 \times \PM^1$, and $U$ its
complement, which retracts equivariantly onto the ``opposite'' $\PM^1$
which yields $Z'$. The short exact sequence
\[
0 \to H^*(Z, i^!\QM_{\PM^1\times \PM^1}) \to H^\bullet(\PM^1 \times
\PM^1, \QM) \to H^*(Z', i^*\QM_{\PM^1\times \PM^1})  \to 0
\]
gives the unsurprising identity
\[
1 + 2q + q^2 = (q + q^2 ) + (1 + q)
\]
coming from the Poincar\'e polynomial of the left hand term
(resp. right hand term) in the short exact sequence.
\end{ex}

\subsection{Slices and their subvarieties in the flag variety} \label{sec:slices}
Our
eventual goal is to make the sketch 
provided in the previous section precise. We will need an explicit description of the slice
$S_{x,y}$ in the case of $\GL_{n+1}$. Most of this is rather
standard, for more information, see \cite{Fulton}, \cite[\S 3.2]{WY1} and \cite[\S 2.2]{WY2}. The experienced
reader can probably skim over this subsection and the next.

Recall that $B \subset \GL_{n+1}(\CM)$ denotes the subgroup of upper
triangular matrices, $U_-$ denotes the subgroup of lower
uni-triangular matrices and that we identify permutations in
$S_{n+1}$ with the corresponding permutation matrices in
$\GL_{n+1}(\CM)$.

Because the multiplication $U_- \times B \to \GL_n(\CM)$ is an open
immersion, $U_-$ provides an open neighbourhood of $B/B$ in $G/B$. It
follows that $xU_-$ provides an open neighbourhood of the point $xB/B$
in $G/B$. We have
\[
xU_- = \{ (\gamma_{i,j}) \; | \; \gamma_{x(i),i} = 1, \gamma_{x(i),j}
= 0 \text{ if $j > x(i)$.} \}
\]
That is, ``zero if right of a 1''. Here is a picture for $x = (1,0,3,2) \in
S_4$:
\begin{equation*}
  xU_- = \left ( \begin{matrix}
    * & 1 & 0 & 0 \\ 1 & 0 & 0 & 0 \\ * & * & * & 1 \\ * & * & 1 &
    0 \end{matrix} \right )
\end{equation*}
In this chart, the $B$-orbit through $xB/B$ consists of ``zero if not
above a 1''.

In particular, a normal slice to the $B$-orbit through $xB/B$ is given by 
\begin{align*}
  S_x &=\{ (\gamma_{i,j}) \; | \; \gamma_{x(i),i} = 1, \gamma_{i,j} = 0
  \text{ if $j > x(i)$ or $i < x^{-1}(j)$} \}
\end{align*}
That is, ``zero if right or above a 1''. 
A picture for $x = (1,0,3,2) \in
S_4$:
\begin{equation*}
  S_x = \left ( \begin{matrix}
      0 & 1 & 0 & 0 \\ 1 & 0 & 0 & 0 \\ * & * & 0 & 1 \\ * & * & 1 &
    0 \end{matrix} \right )
\end{equation*}

Throughout an important role will be played by
\[
  m := x^{-1}(0).
\]
This is the column in which the 1 occurs in the top row in $x$. Consider the subvarieties
\begin{gather*}
S^H_x = \{ \gamma \in S_x \; | \; \gamma_{i,j} = 0 \text{ unless $(i,j)
  = (x(j),j)$ or $j = m$} \}, \\
S^I_x = \{ \gamma \in S_x \; | \; \gamma_{i,j} = 0 \text{ unless $(i,j)
  = (x(j),j)$ or $j \ne m$} \}.
\end{gather*}
That is, ``zero if not in column number $m$'' and ``zero if in
column number $m$'' respectively. A picture for $x = 1032 \in
S_4$ (so $m = 1$):
\begin{equation*}
  S^H_x = \left ( \begin{matrix}
      0 & 1 & 0 & 0 \\ 1 & 0 & 0 & 0 \\ 0 & * & 0 & 1 \\ 0 & * & 1 &
      0 \end{matrix} \right ) \quad \text{and} \quad
    S^I_x = \left ( \begin{matrix}
      0 & 1 & 0 & 0 \\ 1 & 0 & 0 & 0 \\ * & 0 & 0 & 1 \\ * & 0 & 1 &
    0 \end{matrix} \right ) 
\end{equation*}



\subsection{Equations defining slices to Schubert varieties} \label{sec:equations}

Let $g \in G = \GL_{n+1}(\CM)$ be a matrix.  By a \emph{lower left corner} we
mean the submatrix:
\[
g^{\le (p,q)} = \{ (g_{ij})_{i \ge p, j \le q} \}
\]
Given a matrix, its \emph{corner rank matrix} is the matrix which
at position $(p,q)$ records the rank of $g^{\le (p,q)}$. These matrices will
be particularly important for permutation matrices. For example, if $y = (1,2,3,0)$ then
\[
y = \left ( \begin{matrix} 0 & 0 & 0 & 1 \\ 1 & 0 & 0 & 0 \\ 0 &
    1 & 0 & 0 \\ 0 & 0 & 1 & 0 \end{matrix} \right ) \quad \text{and}
\quad \text{corner rank matrix} = \left ( \begin{matrix} 1 & 2 & 3 & 4 \\ 1 & 2 & 3 & 3 \\ 0 &
    1 & 2 & 2 \\ 0 & 0 & 1 & 1 \end{matrix} \right ) 
  \]

  Suppose that $g'$ is obtained from $g$ by a scaling rows or columns, upwards row operation,
or a rightwards column operation, then the rank of any lower left
corner remains unchanged. In particular, any matrix $g \in ByB$
satisfies
\[
\rank( g^{\le (p,q)}) = \rank( \dot{y}^{\le (p,q)}) =
\text{$(p,q)^{th}$ entry in corner rank matrix of $\dot{y}$}
\]
for all $p,q$. It is not difficult to see that $g \in ByB$ if and only if
these conditions are met.

 In fact, $\overline{ByB/B}$ is cut out by the equations:
 \begin{equation} \label{eq:Schubeq}
\rank(g^{\le (p,q)}) \le \rank( \dot{y}^{\le(p,q)}) \quad \text{ for all $p,q$}.
\end{equation}
Intersecting with the slices $S_x, S_x^H$ and $S_x^I$ defines
subvarieties
\begin{gather*}
  S_{x,y} := S_x \cap \overline{ByB/B}, \\
  S_{x,y}^H := S^H_x \cap \overline{ByB/B}, \\
  S^I_{x,y} := S^I_x \cap \overline{ByB/B}
\end{gather*}
which are cut out of the respective affine spaces by rank conditions.

\begin{ex} \label{ex:running3}
  We continue our running example (see Examples \ref{ex:running} and
  Example \ref{ex:running2}). The corner rank matrix $y = (2,3,0,1)$
  is
\[
\dot{y} = \left ( \begin{matrix} 0 & 0 & 1 & 0 \\ 0 & 0 & 0 & 1 \\ 1 &
    0 & 0 & 0 \\ 0 & 1 & 0 & 0 \end{matrix} \right ) \quad \text{and}
\quad \text{corner rank matrix} = \left ( \begin{matrix} 1 & 2 & 3 & 4 \\ 1 & 2 & 2 & 3 \\ 1 &
    2 & 2 & 2 \\ 0 & 1 & 1 & 1 \end{matrix} \right ) 
\]
Recall that $x = (0,2,1,3)$ and hence the slice $S_x$ is:
\[
  S_x = \left ( \begin{matrix}
      1 & 0 & 0 & 0 \\ a & 0 & 1 & 0 \\ b & 1 & 0 & 0 \\ c & d & e & 1
   \end{matrix} \right )
\]
In this case, the equations cutting out $S_{x,y} \subset S_x$ reduce
  to $\rank g^{\le(3,0)} \le 0$, i.e. $c = 0$, and $\rank g^{\le(1,2)}
  = 1$, i.e. $ae + bd = 0$. This recovers the description claimed in
  Example \ref{ex:running}.
\end{ex}

 \subsection{Decomposition of the projectivized slice} \label{sec:decomp}
 In this section we make the outline in \S \ref{sec:sketch} precise. In particular, we define the
 subvarieties $U$, $Z$ and $Z'$. As in 
 \S \ref{sec:sketch} we denote
 \[
\ic := \ic( \PM_\lambda \dot{S}_{x,y}, \QM).
   \]

Given a $T$-stable closed subvariety $X
 \subset \dot{S}_{x,y}$ we denote by $\PM_\lambda \dot{X}$ the closed
 subvariety obtained as the image  of $\dot{X} := X - \{ x \}$ in
 $\PM_\lambda \dot{S}_{x,y}$. It is a closed subvariety. Set
 \begin{align*}
   Z := \PM_\lambda \dot{S}_{x,y}^H, \quad 
   U := \PM_\lambda \dot{S}_{x,y}- \PM_\lambda \dot{S}_{x,y}^H, \quad
   Z' := \PM_\lambda \dot{S}^I_{x,y}
 \end{align*}
and denote by $i, j, i'$ the inclusions of  $Z, U, Z'$ into $\PM_\lambda \dot{S}_{x,y}$.

 Let
 \[
L = \{ z \in [x,y] \;| \;z^{-1}(0) = x^{-1}(0) \} \subset [x,y].
   \]
In order to carry out
 the program outlined in \S \ref{sec:sketch} we need to check the
 following statements:

 \begin{thm} \label{thm:inductive}
   $U$ retracts $T$-equivariantly onto $Z'$, $(i')^*\ic$ is pure, and
   the Poincar\'e polynomial of $H^\bullet(Z', (i')^*\ic)$ is given by $I_{x,y,L}$.
 \end{thm}

 We prove this in \S\ref{sec:inductive_geometry} and \S \ref{sec:inductive_sheaves} below.

 \begin{thm} \label{thm:hypercube}
   $i^! \ic$ is pure, and the Poincar\'e polynomial of $H^\bullet(Z, i^!\ic)$
   is given by $Q_{x,y,L}$.
 \end{thm}

 We prove this in \S \ref{sec:hypercube_sheaves} below, using an
 alternative (geometric) definition of the hypercube map. In \S
 \ref{sec:hypercube combinatorics} we check that these two definitions
 agree.

 \begin{ex}
  In the setting of our running example (see Examples
  \ref{ex:running}, \ref{ex:running2} and \ref{ex:running3}) we have
  \[
  S^H_x = \left ( \begin{matrix}
      1 & 0 & 0 & 0 \\ * & 0 & 1 & 0 \\ * & 1 & 0 & 0 \\ 0 & 0 & 0 & 1
    \end{matrix} \right ) \quad \text{and} \quad
    S^I_x = \left ( \begin{matrix}
      1 & 0 & 0 & 0 \\ 0 & 0 & 1 & 0 \\ 0 & 1 & 0 & 0 \\ 0 & * & * & 1
   \end{matrix} \right ) 
\]
which produce two $T$-stable ``opposite $\PM^1$'s'' in $\PM^1
\times \PM^1$ discussed in Example \ref{ex:running2}.
\end{ex}

 \subsection{The inductive piece:
   geometry} \label{sec:inductive_geometry}
 The results for the inductive piece follow essentially because
everything is nicely compatible with restriction to the fixed points
of a particular choice of $\gamma : \CM^* \to T$ whose centralizer is
$\GL_1 \times \GL_n$. This is a common
theme in Geometric Representation Theory: much information can be
gained from Levi subgroups via torus localization. Now we give the details.

Recall from above that we have fixed a one-parameter subgroup $\lambda
: \CM^* \to T$ which acts attractively on $S_{x,y}$.   It is this $\lambda$
that we used to form the quotient 
  \[
\PM_\lambda \dot{S}_{x,y} = \dot{S}_{x,y}/_\lambda \CM^*.
    \]

    \begin{remark}
      It will not be important to fix a particular choice of $\lambda$. However, the reader might like to check that we could
      take
      \[
\lambda : z \mapsto \diag (1, z, z^2, \dots, z^n).
  \]
    \end{remark}

    Below another one-parameter subgroup will be
    important.  Consider
\begin{align*}
  \gamma : \CM^* &\to \GL_{n+1} \\
   z &\mapsto \diag (1, z, z, \dots, z).
       \end{align*}
Because we have two $\CM^*$'s at play, we will use $\lambda$ and
$\gamma$ to distinguish them.  For example, if we refer to the $\CM^*_\lambda$-action we mean
the $\CM^*$-action provided by $\lambda$, and $S_{x,y}^\gamma$ denotes
the $\CM^*$-fixed points on $S_{x,y}$, with $\CM^*$-action given by
$\gamma$.

Firstly, note that $\CM^*_\gamma$ acts naturally on
$(n+1)\times (n+1)$ matrices via conjugation. The action is via
scaling block matrices as follows
\[
  \left ( \begin{array}{c} \begin{tikzpicture}[scale=.5]
      \node at (0,0) {$1$};
      \node at (2,0) {$z^{-1}$};
      \node at (0,-2) {$z$};
      \node at (2,-2) {$1$};
      \draw (-.5,-.5) -- (3.5, -.5);
      \draw (.5,.5) -- (.5, -3.5);
    \end{tikzpicture} \end{array} \right )
\]
In particular,
\[
G^\gamma = \GL_1(\CM) \times \GL_{n}(\CM).
\]
This subgroup will play an important role below. We give it its own
notation:
\[
G_{n} = \GL_1(\CM) \times \GL_{n}(\CM) \subset G.
\]
This is a Levi subgroup of $G$. We
denote by $S_n$ its Weyl group. It consists of the subgroup of
$S_{n+1}$ consisting of those permutations which fix $0$. We let $B_{n} = G_n \cap B$ denote the Borel subgroup of upper
triangular matrices in $G_n$. 

\begin{remark}
  We caution the reader that the inclusion of $G_n \into G$ is via
  block $1 \times n$-matrices, not $n \times 1$ as it more
  common. Similarly, our inclusion $S_n \into S_{n+1}$ via
  permutations fixing $0$ is slightly non-standard.
\end{remark}

\begin{prop}
  $(G/B)^\gamma$ is a disjoint union of flag varieties for $G_n$. More
  precisely,
  \[
(G/B)^\gamma = \bigsqcup_{x \in S_n \setminus S_{n+1}} G_n \cdot x
\]
and each $G_n \cdot x$ orbit is isomorphic to $G_n/B_n$.
\end{prop}

\begin{proof}
  This is a standard result, so we will be brief. Firstly, $G_n = G^\gamma$ certainly acts on the
  $(G/B)^\gamma$. If $z$ denotes a minimal coset
  representative of $S_n \setminus S_n$ (i.e. $z(l) = 0$ for some
  $0 \le l \le n$ and $z$ maps the remainder $0, 1, \dots, l-1, l+1,
  \dots, n$ monotonically to $1, \dots, n$) then the stabilizer of
  $G_n$ acting on $zB/B$ is
  \[
    z^{-1} G_n z \cap B = B_n
  \]
  as follows from an easy calculation. Thus action on the points given
  by minimal coset representatives gives an inclusion
    \begin{equation} \label{eq:Gninclustion}
\bigsqcup_{x \in S_n \setminus S_{n+1}} G_n/B_n \subset (G/B)^\gamma.
\end{equation}

A similar computation to that establishing $G^\gamma = G_n$ above
shows that on each of the $xU_-$ for $G/B$, the $\gamma$-action is via scaling the free
variables in $0^{th}$-row by $z^{-1}$, and scaling the entries in the
$x^{-1}(0)^{th}$-column by $z$. Thus the fixed points in this chart
have zeroes in the $0^{th}$-row and $x^{-1}(0)^{th}$-column, except
the $1$ in position $(0, x^{-1}(0))$. It follows that
\[
(xU)^{\gamma} \subset G_n  \cdot xB/B
\]
As the charts $xU_-$ cover $G/B$, we deduce that the inclusion in
\eqref{eq:Gninclustion} is an equality. \end{proof}

Consider a $T$-fixed point $xB/B$. We can consider the slice $S_x
\subset G/B$ to the $B$-orbit through $xB/B$. On the other hand, if $z
\in S_nx$ denotes the minimal coset representative then we have seen
that $g \mapsto g \cdot zB/B$ yields an embedding 
\[
 \phi : G_n/B_n \into G/B
\]
of flag varieties. If we set $x' := xz^{-1}$ then $x' \in S_n$, and
$x'B_n/B_n$ maps to $xB/B$ under this inclusion. We denote by
$S_{x,n}$ the image under this inclusion of the slice $S_{x'}$ to the
$B_n$-orbit through $x'$ in $G_n/B_n$.

\begin{prop} \label{prop:slice_equal}
  We have $S_x^\gamma = S_{x,n}$ as subvarieties of $G/B$.
\end{prop}

\begin{proof}
  Recall the subvariety $S_x^I$ defined in \S
  \ref{sec:slices}. a simple computation shows
  \[
S_x^I = S_x^\gamma.
\]
On the other hand, the slice to $x'$ in $G_n/B_n$ has the form:
\[
  \left ( \begin{array}{c} \begin{tikzpicture}[scale=.5]
      \node at (0,0) {$1$};
      \node at (2,0) {$0$};
      \node at (0,-2) {$0$};
      \node at (2,-2) {$S_{x'}$};
      \draw (-.5,-.5) -- (3.5, -.5);
      \draw (.5,.5) -- (.5, -3.5);
    \end{tikzpicture} \end{array} \right )
\]
Right multiplication by $z$ produces exactly $S_x^I$, and the result follows.
\end{proof}

Finally, one checks easily that $\gamma$ has positive weights on the
slice $S_x$. The following is immediate:

       \begin{lem} \label{lem:retract}
For any $p \in S_{x} - S_{x}^H$ we have
\[
         \lim_{z \to 0} \gamma(z) \cdot
         p \in S_x^I. \]
     \end{lem}

      \subsection{The inductive piece:
        sheaves} \label{sec:inductive_sheaves}
      In this section we prove Theorem \ref{thm:inductive}. We keep
      the notation of previous sections, in particular \S
      \ref{sec:decomp}. Note first that Lemma \ref{lem:retract}
      implies that $U = \PM_\lambda \dot{S}_{x,y} - \PM_\lambda
      \dot{S}^H_{x,y}$ retracts equivariantly onto $\PM_\lambda
      \dot{S}_{x,y}^I$. Hence, by the Attractive Proposition and the
      Attractive Weight Argument, we conclude:
      \begin{itemize}
      \item $(i')^*\ic$ is pure;
      \item $H^\bullet(U, j^*\ic) = H^\bullet(Z', (i')^*\ic)$,
        and both are pure.
      \end{itemize}
      All that remains is to check that the Poincar\'e polynomial of
      $H^\bullet(Z', (i')^*\ic)$, is given by $Q_{x,y,L}$. This will
      take a little more work, and uses in a more substantial way the
      results of the previous section.

      Consider the embedding
      \[
\phi : G_n/B_n \into G/B
\]
of flag varieties considered in the previous section, whose image is
the $G_n$-orbit through $xB/B \in G/B$.

\begin{prop}
  $\phi^* \ic(\overline{ByB/B}, \QM)$ is pure.
\end{prop}

\begin{proof}
  Consider the Bia{\l}ynicki-Birula decomposition with respect to
  $-\gamma$.\footnote{i.e. $z
  \mapsto \diag(1,z^{-1},\dots, z^{-1})$.} Because $-\gamma$ has
  positive weights on $B$, the attracting sets are unions of Bruhat
  cells. In particular, the attracting set $Z^+ \subset G/B$ for the
  component $Z = \phi( G_n/B_n)$ of $(G/B)^\gamma$ is given by
  \[
Z^+ = \bigsqcup_{u \in S_nx} BuB/B
\]
and the attracting map $\pi$ realizes $Z^+$ as an affine bundle over $Z$. We
have a diagram
\[
  \begin{tikzpicture}[scale=0.7]
    \node (z) at (0,0) {$Z^+$};
    \node (fbig) at (2,0) {$G/B$};
    \node (flittle) at (0,-2) {$G_n/B_n$};
    \draw[->] (z) to node[above] {$\phi'$} (fbig);
    \draw[->] (z) to[out=-45,in=45] node[right] {$\pi$} (flittle);
    \draw[<-] (z) to[out=-135,in=135] node[left] {$\iota$} (flittle);
        \draw[->] (flittle) to[out=0,in=-90] node[below, right] {$\phi$} (fbig);
  \end{tikzpicture}
\]
By Braden's theorem on hyperbolic localization \cite[Theorems 1 and
2]{Braden} we know that
\[
\pi_! (\phi')^* \ic(\overline{ByB/B}, \QM)
\]
is pure. (It is a summand of the hyperbolic localization of
$\ic(\overline{ByB/B}, \QM) $ to the fixed points of $-\gamma$.) On the other hand, we know that this is equal to
\[
\iota^! (\phi')^* \ic(\overline{ByB/B}, \QM)
\]
by the Attractive Proposition. Finally, $Z^+$ is stratified by Bruhat
cells, and $\iota$ is a normally non-singular inclusion for this
stratification. Hence $\iota^! = \iota^*[-2m]$, where $m$ is the complex
dimension of the fibres. We deduce that the above complexes are equal
(up to a shift) to
\[
\iota^* (\phi')^* \ic(\overline{ByB/B}, \QM) \cong \phi^*\ic(\overline{ByB/B}, \QM).
\]
The latter complex is the one for which we wished to establish purity.
\end{proof}

As $\phi^* \ic(\overline{ByB/B}, \QM)$ is pure, it decomposes as a direct sum
of shifts of intersection cohomology complexes. We now describe how to
compute some of these multiplicities. First some notation: given a polynomial
$p = \sum a_i q^i$ with positive 
coefficients and an object $\FS$ in a derived category, write
\[
p \cdot \FS := \bigoplus \FS^{\bigoplus a_i}[-2i].
\]
Because $\phi^* \ic(\overline{ByB/B}, \QM)$ is pure and its stalks
have no cohomology in odd or negative degree, we can certainly write
\begin{equation} \label{eq:res_mult}
\phi^* \ic(\overline{ByB/B}, \QM) = \bigoplus_{v \in S_nx}  \gamma'_v \cdot \ic(\overline{BvB/B}). 
\end{equation}
for certain $\gamma_v' \in \ZM_{\ge 0}[q]$.

Recall that in \S \ref{sec:inductive_combinatorial} we defined certain
polynomials $\gamma_v$ as part of the definition of the inductive
piece:

\begin{prop} \label{prop:gammas}
  For $x \ne v \in L$ we have $\gamma_v = \gamma_v'$.
\end{prop}

\begin{proof}
  Taking Poincar\'e polynomials of stalks on both sides of
  \eqref{eq:res_mult} for any $x \ne v \in L$. We deduce the
  relations:
  \[
P_{v,y} = \sum_{x \ne u \in L} \gamma_v' P_{u,v}.
\]
These are (a repackaging of) the relations that determine the
$\gamma_v$ in \S \ref{sec:inductive_combinatorial}. The
proposition follows.
\end{proof}

We now reach the goal of this section:

\begin{thm}
  The Poincar\'e polynomial of $H^\bullet(Z', (i')^*\ic)$ is given by
  $I_{x,y,L}$. 
\end{thm}

\begin{proof}
  Consider the commutative diagram
  \[
  \begin{tikzpicture}[yscale=.6,xscale=.9]
    \node (lf) at (0,4) {$G_n/B_n$};
    \node (ls) at (0,2) {$\dot{S}_{x,n}$};
    \node (sp) at (0,0) {$\PM_\lambda \dot{S}_{x,n}$};
    \draw[->] (ls) to node[right] {$k_n$} (lf);  \draw[->] (ls) to (sp);
        \node (bf) at (2,4) {$G/B$};
    \node (bs) at (2,2) {$\dot{S}_x$};
    \node (bp) at (2,0) {$\PM_\lambda \dot{S}_x$};
    \draw[->] (bs) to node[right] {$k$} (bf);  \draw[->] (bs) to (bp);
    \draw[->] (lf) to node[above] {$\phi$} (bf);
    \draw[->] (ls) to node[above] {$i''$} (bs); 
    \draw[->] (sp) to node[above] {$i'$} (bp);
  \end{tikzpicture}
\]
where we have used Proposition \ref{prop:slice_equal} to identify
$S_x^I$ with $S_{x,n}$. Because $\CM^*$ acts (via $\lambda$) on
$\dot{S}_x$ with finite stabilizers we can apply the Finite Stabilizer
Argument. We have
\[
H^\bullet(Z', (i')\ic) = H^\bullet_{\CM^*}(\dot{S}_{x,n}, (i'')^* \ic(
\dot{S}_{x,y},\QM))
= H^\bullet_{\CM^*}(\dot{S}_{x,n}, (i'')^* k^* \ic( \overline{ByB/B},\QM))
\]
where for the second equality we have used that $k$ is a normally
non-singular inclusion. Using the above commutative diagram, we can
rewrite this as:
\begin{gather*}
H^\bullet_{\CM^*}(\dot{S}_{x,n}, k_n^* \phi^*\ic(
\overline{ByB/B},\QM)) \cong \bigoplus_{x \ne v \in L}H^\bullet_{\CM^*} (\dot{S}_{x,n}, \gamma_v
\cdot k_n^* \ic(
\overline{BvB/B},\QM) \\
\cong \bigoplus_{x \ne v \in L}H^\bullet (\PM_\lambda \dot{S}_{x,n}, \gamma_v
\cdot \ic(\PM_\lambda \dot{S}_{x,v},\QM))
\end{gather*}
We used Proposition \ref{prop:gammas} for the first step, and the
normal nonsingularity of $S_{x,n} \into G_n/B_n$ and the Finite
Stabilizer Argument again
for the second step.
Taking Poincar\'e polynomials (and using Proposition \ref{prop:qkl})
we deduce that its Poincar\'e polynomial is
\[
\sum_{x \ne v \in L} \gamma_v \cdot \qKL_{x,v} = I_{x,y,L}. \qedhere
  \]
\end{proof}

\subsection{Geometry of weighted projective space} \label{sec:weightedproj}
We now turn to the
hypercube piece, which comes from a sheaf on $S_x^H/_\lambda
\CM^*$. We will see that this space is isomorphic to a weighted
projective space. In this section we discuss some preliminaries on
weighted projective spaces with torus action.

Suppose that a torus $T$ acts diagonally on $V = \bigoplus_{i=1}^p 
\CM e_i$. Assume moreover:
\begin{enumerate}
\item the weights of $T$ on $V$ are linearly independent;
\item we have fixed a cocharacter $\lambda : \CM^* \to T$, all of whose
  weights on $V$ (with respect to the action induced by $T$) are positive.
\end{enumerate}
Note that (1) implies:
\begin{enumerate}
\item[(3)] $T$ has finitely many orbits on $V$. Moreover, $T$-orbits on $V$
  are classified by subsets of $\{1, \dots, p\}$; given such a subset
  $I$ the corresponding $T$-orbit is
  \[
V_I = \left \{ \sum \lambda_{i\in I} e_i \; \middle | \; \lambda_i \in
  \CM^* \right \}.
\]
The stratification by the $\{ V_I \}$ will be called the
\emph{toric stratification} of $V$.
\item[(4)] For any subset $I \subset \{ 1, \dots, n \}$ we can find a
  cocharacter $\lambda_I : \CM^* \to T$ such that the induced action
  fixes each $e_i$ for $i \in I$ and has positive weights on the
  rest. In particular
  \[
\lim_{z \to 0} \lambda_I(z) \cdot v \in \bigoplus_{i \in I} \CM e_i
\]
for all $v \in V$.
\end{enumerate}

Denote the weighted projective space obtained as the
quotient by
\[
\PM_\lambda \dot{V} := ( V - \{ 0 \}) /_{\lambda} \CM^*.
\]
Statements (3) and (4) for $V$ imply the following for $\PM_\lambda \dot{V}$:
\begin{enumerate}
\item[(5)] $T$ has finitely many orbits on $\PM_\lambda
  \dot{V}$. Moreover, these orbits are classified by non-empty subsets
  of $\{1, \dots, p \}$. Given such a subset $I$ the corresponding
  $T$-orbit is $\OC_I$, given by the image of $V_I$ in $\PM_\lambda
  \dot{V}$. We will refer to this stratification as the \emph{toric stratification} of $\PM_\lambda \dot{V}$.
\item[(6)] For any $\emptyset \ne I \subset \{1, \dots, p \}$ we set
  \[
U_I = \bigcup_{I \subset J} \OC_J = \{ [\lambda] \; | \; \lambda =
\sum \lambda_i e_i, \; \lambda_i \ne 0 \text{ if } i \in I \}.
\]
\end{enumerate}
Note that $U_{\{1\}}, \dots, U_{\{p\}}$ form an open covering of $\PM_\lambda
\dot{V}$ by finite quotients of $\CM^{p-1}$. Also, (4) above implies
that, for all $\emptyset \ne I \subset \{1,\dots, p \}$,
\begin{equation} \label{eq:Uretract}
  \text{$U_I$ retracts equivariantly onto $\OC_I$.}
\end{equation}

\subsection{Equivariant sheaves on vector spaces} \label{sec:shvec}
Let $T$ and $V =
\bigoplus \CM e_i$ be as
above. (In applications, $V$ will be a slice to a stratum in weighted
projective space. Hence, it will not be the ``same $V$'' as in the
previous section. We hope that this does not lead to too much confusion.) In particular, the weights of $T$ on $V$ are
linearly independent and we have the stratification $V_I$ by subsets
of $\{ 1, \dots, p \}$ (so $V_{\emptyset} = \{ 0 \}$ and
$V_{\{1,\dots, p\}}$ consists of vector all of whose coordinates are
  non-zero.)

  The aim of this section is to prove the following:

  \begin{prop} \label{prop:pure}
    Consider $\FS \in D^b_T(V,\QM)$. Suppose that:
    \begin{enumerate}
    \item The stalks of $\FS$ are pure of weight zero.
    \item The restriction of $\FS$ to each stratum is isomorphic to a
      direct sum of even shifts of constant sheaves.
    \item For any $I$ the restriction map
      \[
H^\bullet_T(V,\FS) \to H^\bullet_T(V_I, \FS_{|V_I})
\]
is surjective.
\end{enumerate}
Then $\FS$ is pure of weight zero.
  \end{prop}

  \begin{remark}
    This is close to several results in the literature. As the reader
    will see, it uses in a crucial way in the proof that all
    subvarieties $\overline{V_I}$ are smooth. The analogue of this
    theorem for a general attractive fixed point is false.
  \end{remark}

  \begin{proof}
    Consider the trunction functors $\tau_{\le k}$ and $\tau_{>k}$
    associated to the standard (i.e. not the perverse!) $t$-structure
    on $D^b_T(V,\QM)$. We first argue that each sheaf in the
    distingushed triangle
    \begin{equation} \label{eq:triangle}
\tau_{\le k} \FS \to \FS \to \tau_{>k} \FS \triright
\end{equation}
saisfies the conditions (1), (2) and (3) of the theorem. (1) and (2)
are clear, as $\tau_{\le k}$ and $\tau_{>k}$ have a predictable
effect on stalks. For (3), note that from the existence of an
attractive cocharacter we know by the Attractive Proposition that the
restriction map
\[
H^\bullet_T(V, \FS) \to H^\bullet_T(\{0\},\FS_0)
\]
is an isomorphism. Hence we have a commutative diagram:
\[
  \begin{tikzpicture}[xscale=1.5,yscale=.8]
\node (ull) at (-4,0) {$0$} ;
\node (ul) at (-2,0) {$H^\bullet_T(\{0\}, (\tau_{\le k}\FS)_0)$} ;
\node (u) at (0,0) {$H^\bullet_T(\{0\},\FS_0)$};
\node (ur) at (2,0) {$H^\bullet_T(\{0\}, (\tau_{> k}\FS)_0)$} ;
\node (urr) at (4,0) {$0$} ;
\node (lll) at (-4,-2) {$0$} ;
\node (ll) at (-2,-2) {$H^\bullet_T(\{0\}, (\tau_{\le k}\FS)_{V_I})$} ;
\node (l) at (0,-2) {$H^\bullet_T(\{0\},\FS _{V_I})$};
\node (lr) at (2,-2) {$H^\bullet_T(\{0\}, (\tau_{> k}\FS) _{V_I})$} ;
\node (lrr) at (4,-2) {$0$} ;
\draw[->] (ull) to (ul); \draw[->] (ul) to (u); \draw[->] (u) to (ur);
\draw[->] (ur) to (urr);
\draw[->] (lll) to (ll); \draw[->] (ll) to (l); \draw[->] (l) to (lr);
\draw[->] (lr) to (lrr);
\draw[->] (ul) to (ll); \draw[->] (u) to (l); \draw[->] (ur) to (lr);
\end{tikzpicture}
\]
(The rows are short exact by parity vanishing of $T$-equivariant
cohomology of $T$-orbits). The middle arrow is a
surjection, hence so are the left and right arrows by the 5-lemma\footnote{or
  more precisely the 4-lemma\dots}. Thus (1), (2) and (3) are true for
each term in \eqref{eq:triangle}.

If we know the theorem for the left and right terms in
\eqref{eq:triangle}, then each term is isomorphic to a direct sum of
even shifts of constant sheaves on the closure of the strata $V_I$, by
purity. In
particular, $\Hom(\tau_{>k} \FS , \tau_{\le k} \FS[1]) = 0$ and hence
\eqref{eq:triangle} splits, and the theorem is true for $\FS$ too. In particular, it is enough to show
the proposition for the left and right terms. By induction, we may
assume that $\FS$ is concentrated in degree zero.

Fix a stratum $\OC := V_I$ which is maximal (i.e. open) in the support of
$\FS$. (Of course $\OC$ many not be unique, but this won't worry us.) We will argue
  that a choice of summand $\QM_{\OC} \subset \FS_{\OC}$ leads to a
  summand $\QM_{\overline{\OC}}\subset \FS$.
  
Let $j$ denote the open inclusion of all strata above $V_I$ into
$V$. (This subvariety was called $U_I$ in \S\ref{sec:weightedproj}.)
By the Attractive Proposition we have
\[ \Hom(\QM_{\overline{\OC}}, \FS) = H^0_T({\overline{\OC}}, j^*\FS) =
    H^0_T(\{0\},\FS_0) =
H^0_T(V,\FS). \]
Because we have a surjection
  \[
H^0_T(V,\FS) = \Hom(\QM_{\overline{\OC}}, \FS) \onto \Hom(\QM_{\OC}, \FS_{\OC}) = H^0_T(\OC, \FS_{\OC})
\]
we may find a map $\beta : \QM_{\overline{\OC}} \to \FS$ whose restriction to $\OC$ is
the inclusion of our summand. On the other hand, our chosen
projection $\FS_\OC \to \QM_\OC$ corresponds (via adjunction) to a morphism
\[
\FS \to j_*\QM_{\OC}
\]
which restricts to our chosen map on the open stratum. Applying
$\tau_{\le 0}$ we get a morphism
\[
\alpha: \tau_{\le 0} \FS = \FS \to \tau_{\le 0}  j_*\QM_{\OC} = \QM_{\overline{\OC}}
\]
which still agrees with our chosen projection on the open
stratum. Composing we get $\alpha \circ \beta : \QM_{\overline{\OC}}
\to \QM_{\overline{\OC}}$ which is the identity on the open
stratum. As
\[
\Hom(\QM_{\overline{\OC}}, \QM_{\overline{\OC}}) \to
\Hom(\QM_{\OC}, \QM_{\OC})
\]
is an isomorphism, we deduce that $\alpha \circ \beta$ is an
isomorphism everywhere. In other words, we have found
$\QM_{\overline{\OC}}$ as a summand of $\FS$. Continuing in this way,
we deduce that we can find an isomorphism:
\[
\FS = \bigoplus_{I \subset \{ 1, \dots, p\}}
\QM_{\overline{V_I}}^{\oplus m_I}.
\]
As each ${\overline{V_I}}$ is smooth, $\QM_{\overline{V_I}}$ is pure
of weight zero, hence so is $\FS$ and we are done.
  \end{proof}

Recall the notation of the previous subsection on weighted projective
spaces. The previous statement has the following corollary, which will
be important below:
  
  \begin{cor} \label{cor:pure}
    Suppose that $\FS$ is a sheaf on weighted projective space
    $\PM_\lambda \dot{V}$. Suppose that:
    \begin{enumerate}
    \item The stalks of $\FS$ are pure.
    \item The restriction of $\FS$ to each toric stratum $\OC_I$ is
      isomorphic to a direct sum of even shifts of constant sheaves.
    \item For $i \in \{1, \dots, p\}$ and any $J \subset \{1, \dots,
      p\}$ containing $i$ the restriction map
      \[
H^\bullet_T(U_{\{ i \}},\FS_{U_{\{i\}}}) \to H^\bullet_T(\OC_J, \FS_{\OC_J})
\]
is surjective.
\end{enumerate}
Then $\FS$ is pure.
\end{cor}

\begin{proof}
  It is enough to prove that the restriction of $\FS$ to each $U_{\{i
    \}}$ is pure. This then follows (using the Finite Stabilizer
  Argument) from Proposition
  \ref{prop:pure}.
\end{proof}

The following allows us to compute Poincar\'e polynomials of global
sections from Poincar\'e polynomials of stalks:

\begin{prop} \label{prop:PP}
  Suppose that $\FS \in D^b(\PM_\lambda \dot{V}, \QM)$ is a pure sheaf
  whose restriction to each toric stratum is a direct sum of constant
  sheaves in even degrees. Then the Poincar\'e polynomial of $H^\bullet(\PM_\lambda
  \dot{V},\FS)$  is given by
\[
\sum_{\emptyset \ne I \subset \{1, \dots, p \}} (q-1)^{|I|-1} P(\FS_{x_I})
\]
where $x_I$ denotes a point in the toric stratum $\OC_I$, and $P(\FS_{x_I})$
denotes the Poincar\'e polynomial of the stalk.
\end{prop}

\begin{proof}
  If the formula is true for $\FS_1$ and $\FS_2$, then it is true for
  their sum. Hence (using our purity assumption on $\FS$ to deduce
  that $\FS$ is a direct sum of $IC$-sheaves, and then our assumptions
  on the restriction to each stratum to deduce that only constant
  sheaves show up) it is enough to
  check the formula when
\[
  \FS = \ic(\overline{\OC_I},\QM) =
  \QM_{\overline{\OC_I}}. \]
In this case, $\overline{\OC_I}$ is a weighted projective space of
dimension $|I|-1$ and its Poincar\'e polynomial is
\[
1 + q + \dots + q^{|I|-1}.
\]
Our formula predicts that this should equal
\[
\sum_{\emptyset \ne J \subset I} (q-1)^{|J|-1} = \sum_{0 \ne k \le |I|}
{ |I| \choose k }(q-1)^{k-1}
\]
This is easily checked.\footnote{For example, by a) multiplying both terms by
$(q-1)$, then adding $1$ and using the binomial theorem, or b) considering the
toric stratification of projective space and counting points.}
\end{proof}
 
\subsection{The hypercube piece:
  geometry} \label{sec:hypercube_geometry} In \S
\ref{sec:hypercube_sheaves} and \S \ref{sec:hypercube combinatorics}
we establish Theorem \ref{thm:hypercube}, which describes the Poincar\'e
polynomial for hypercube piece. Here we begin by describing the
(rather simple) geometry involved in the hypercube piece.

\begin{remark}
  The simple geometry of the hypercube piece reminds one of the
  importance of the ``miraboic subgroup'' in type $A$. This suggests
  that generalization these results to other types might need new ideas.
\end{remark}

Recall that $S_x^H$ denotes subvariety of $S_x$ where all coordinates are zero, except
for those below the 1 in the $0^{th}$ column. For example, if $x =210345$ then
\begin{equation*}
  S_x^H = \left ( \begin{matrix}
    0 & 0 & 1 & 0 & 0 & 0\\ 0 & 1 & 0 & 0 & 0 & 0 \\ 1 & 0 & 0 & 0 & 0 & 0\\ 0 &
    0 & * & 1 & 0 & 0\\ 0 & 0 & * & 0 & 1 & 0\\ 0 & 0 & * & 0 & 0 & 1\end{matrix} \right )
\end{equation*}
It will be useful for the reader to keep the explicit form of $S_x^H$
in mind below.

Recall that in \S \ref{sec:equations} we gave equations cutting out
$S_{x,y}$ inside the affine space $S_x$. In general these equations
are complicated. However in simple
situations one can be lucky:

\begin{prop}
  $S_{x,y}^H \cong \CM^p$ for some $p$.
\end{prop}

\begin{proof}
  Consider an $(l \times l)$-matrix of the form
\[   \left ( \begin{matrix} & & & x_1 \\ & P & & \vdots \\ & & &
      x_l  \end{matrix} \right ) \]
where $P$ is an $(l-1) \times l$-matrix of 0's and 1's having at most one 1 in
each row and column. Then its determinant is either 0 or $\pm x_i$ for
some $1 \le i \le l$, as one sees easily by expanding the determinant
from left to right. Applying this observation to the above
equations implies that $S_{x,y}^H \subset S_x^H$ is cut out by the
vanishing of certain coordinates in the $m^{th}$-column. The
proposition follows.
\end{proof}

Denote by $\e_i$ the
character of the torus given by $\diag (\lambda_0, \lambda_1, \dots,
\lambda_n) \mapsto \lambda_i$. The following is an easy direct
calculation:

\begin{lem}
  The $T$-weights on $S_x^H$ are all distinct, and belong to $\{ e_i - \e_0\}_{i = 1}^n$. In
  particular, $T$-has finitely many orbits on $S_x^H$. 
\end{lem}

Recall that
\[
Z = \PM_\lambda \dot{S}_{x,y}
\]
where $\lambda$ is a cocharacter acting attractively on the slice
$S_{x,y}$. We deduce from the above that $Z$ is a weighted projective
space, and hence results of \S\ref{sec:weightedproj}
and S\ref{sec:shvec} apply.

\begin{prop} \label{prop:intersections}
  $Z = \PM_\lambda S_{x,y}^H$ is isomorphic to a weighted projective
  space, and $T$ has finitely many orbits on $Z$. Moreover, each toric
  stratum in $Z$ is the intersection with 
  a unique Schubert stratum $\PM_\lambda \dot{S}^u_{x,y}$ for $u \in   (x,y]$. 
\end{prop}

\begin{proof}
  The only statement not clear from the above is the one starting
  ``Moreover, \dots''. However, note that the intersection of $Z$ with
  the image of any Schubert stratum is $T$-stable, and hence a union
  of $T$-orbits. The statement follows because $T$ has finitely many
  orbits on $Z$.
\end{proof}

Let $\Chi$ denote the $T$-weights occurring in $S_{x,y}^H$. (Note that
the weights in $\Chi$ are linearly independent and the $S_{x,y}^H
\cong \CM^{|\Chi|}$.) It follows that we obtain a map
\[
\theta_{\geom} : \text{subsets of $\Chi$} \to [x,y].
\]
uniquely characterised by the fact that for any subset $I$
of  $\Chi$, $S_{x,y} \cap S_{x,y}^{\theta(I)}$ is the toric stratum determined by the
non-vanishing of the coordinates in $I$. This is the \emph{(geometric) hypercube map} and will play an
important role below. We will eventually see that it agrees with the
combinatorial hypercube map in (defined in
\S \ref{sec:hypercube piece}). Until we know their equality, we keep
the $\geom$ subscript to distinguish the two.

\subsection{The hypercube piece:
  sheaves} \label{sec:hypercube_sheaves} In this section, we prove
that $i^*\ic$ is pure, and that its Poincar\'e polynomial is given by
the hypercube piece in our formula, where we use the geometric
hypercube map in place of its combinatorial version. All that then remains is to check
that the geometric $\theta_{\geom}$ (defined above) agrees with the
combinatorial hypercube map $\theta$ (defined in
\S \ref{sec:hypercube piece}). This we do in the subsequent section.

\begin{prop} \label{prop:ipure}
  $i^*\ic$ is pure.
\end{prop}

\begin{proof}
  We will check the conditions of Corollary \ref{cor:pure}. As $i^*\ic$
  is the restriction of a pointwise pure sheaf, its stalks are pure, and so (1)
  is satisfied. Moreover, the restriction of
  $\ic(\overline{ByB/B},\QM)$ to each Bruhat stratum is a direct sum
  of shifts of constant sheaves in even degrees. As the stratification
  of $Z$ is via intersections with Bruhat cells, it follows that (2)
  is satisfied. We now turn to (3) which is the most subtle. Denote by
  $T' = T/\lambda(\CM^*)$ the torus which acts on $\PM_\lambda
  \dot{S}_{x,y}$.   We want
  to check that for each $i$, the restriction map
  \[
H_{T'}^\bullet(U_{\{i\}}, (i^*\ic)_{U_{\{i\}}}) \to H_{T'}^\bullet( Z_J, (i^*\ic)_J)
    \]
is surjective, for all subsets $J$ containing $i$. Now, via the
attractive proposition we can rewrite this as the restriction map
  \[
H_{T'}^\bullet(Z_{\{i\}}, (i^*\ic)_{Z_{\{i\}}}) \to H_{T'}^\bullet( Z_J, (i^*\ic)_J)
    \]
We already know that both are isomorphic to direct sums of (even) shifts of
$H^\bullet_{T'}(Z_{\{i\}}\QM)$ and $H^\bullet_{T'}(Z_J,\QM)$
respectively. In particular, when we reduce modulo the maximal ideal
of $H_{T'}^\bullet(\pt,\QM)$ (i.e. elements of positive degree) we obtain a map
\[
  H^\bullet( \{p\}, (i^*\ic)_p) \to H^\bullet(\{p'\} , (i^*\ic)_{p'})\]
where $p$ (resp. $p'$) denotes the unique (resp. a choice of) point in
the stratum $Z_{\{i\}}$ (resp. $Z_J$). 
By Nakayama's lemma we are done if we know that this map is
surjective. However, by definition of the geometric hypercube map,
this map may be identified with the map
\[
H^\bullet( \{u\}, \ic(\overline{ByB/B})_u)  \to H^\bullet( \{v\}, \ic(\overline{ByB/B})_v) 
\]
considered in \cite[\S 3.5]{BMP}. This map is surjective by
\cite[Theorem 3.6]{BMP}.
\end{proof}

\begin{remark}
We briefly recall why \cite[Theorem
  3.6]{BMP} holds. The key point is that in the case of flag varieties
  the natural restriction map
  \begin{equation} \label{eq:BMsurj}
IH^\bullet( \overline{ByB/B},\QM) \to H^\bullet(\{u\},\ic(\overline{ByB/B},\QM)_u)
\end{equation}
is surjective, for any point $u \in G/B$. Indeed, by equivariance it
is enough to check this for $T$-fixed points, and then this follows
because all fixed points on $G/B$ are attractive. (Once one has the
surjectivity in \eqref{eq:BMsurj} it is easy to deduce that one has a
surjection $H^\bullet(\{u\},\ic(\overline{ByB/B},\QM)_{u}) \to
H^\bullet(\{v\},\ic(\overline{ByB/B},\QM)_v)$ whenever $v$'s stratum
contains $u$ in its closure.) Braden and
MacPherson use \cite[Theorem
  3.6]{BMP} to deduce the unimodality of Kazhdan-Lusztig polynomials.
\end{remark}

By definition of the hypercube map, the Poincar\'e polynomial of
$i^*\ic$ on the toric stratum $\OC_I \subset Z$ agrees with the
Kazhdan-Lusztig polynomial $P_{\theta_{\geom}(I),y}$. Hence we can use
Propositions \ref{prop:ipure} and \ref{prop:PP} to deduce the
following:

\begin{prop} \label{prop:PPIC}
  The Poincar\'e polynomial of $H^\bullet(Z, i^*\ic)$ is
  given by
  \[
\widetilde{Q}_{\geom} := \sum_{\emptyset \ne I \subset \{1, \dots, p \}} (q-1)^{|I|-1} P_{\theta_{\geom}(I),z}.
\]
\end{prop}

\begin{cor} \label{cor:PPIC}
  The Poincar\'e polynomial of $H^\bullet(Z, i^!\ic)$ is given by
  \[
Q_{\geom} = q^{\ell(y)-\ell(x) - 1}\widetilde{Q}_{\geom})(q^{-1}).
    \]
  \end{cor}

  \begin{proof}
    The Poincar\'e polynomial of the dual of $H^\bullet(Z, i^*\ic)$ is given by
    $\widetilde{Q}_{\geom} (q^{-1})$. Because Verdier duality
    interchanges $i^!$ and $i^*$ we deduce that $\widetilde{Q}_{\geom}
    (q^{-1})$ agrees with the Poincar\'e polynomial of
    \[
H^\bullet(Z, i^! \DM \ic) = H^\bullet( Z, i^!
\ic[2(\ell(y)-\ell(x)-1)]).
\]
We ue that $\DM \ic(X,\QM) = \ic(X, \QM)[2d]$, where $d = \dim_{\CM}
X$, and the fact that $\PM_\lambda \dot{S}_{x,y}$ has dimension
$\ell(y)-\ell(x)-1$. Hence the Poincar\'e polynomial of $H^\bullet( Z, i^!
\ic)$ is $q^{\ell(y)-\ell(x) - 1}\widetilde{Q}_{\geom}(q^{-1})$ as claimed.
\end{proof}

\begin{remark}
  The formula for the Poincar\'e polynomial of $H^\bullet(Z,
  i^!\ic)$ agrees with the hypercube piece, with the only difference
  being that the former involves the geometric hypercube map, whereas
  the later involves the combinatorial hypercube map. Thus to prove
  our formula it is enough to show that the two hypercube maps agree.
\end{remark}

\section{Combinatorics of the hypercube map}
\label{sec:hypercube combinatorics} 

In the previous section we almost finished the proof of our
formula. The only missing piece is the following theorem, that
compares the ``geometric'' and ``combinatorial'' hypercube maps. Here
the arguments are of a different flavour, hence the new section. The goal of this section is the following:

\begin{thm} \label{thm:geom hypercube}
  The set
\[ L = \{ z \in [x,y] \; | \; z^{-1}(0) = x^{-1}(0) \} \]
provides a hypercube decomposition of the Bruhat interval
$[x,y]$. Moreover, we have $\theta = \theta_\geom$, thus the geometric
and combinatorial hypercube maps agree.
\end{thm}

\subsection{Explicit description of the hypercube
  map} \label{sec:hypercube explicit}
Recall that $S_x^H$ denotes subvariety of $S_x$ where all coordinates are zero, except
for those below the 1 in the $0^{th}$ column. For example, if $x =210345$ then
\begin{equation*}
  S_x^H = \left ( \begin{matrix}
    0 & 0 & 1 & 0 & 0 & 0\\ 0 & 1 & 0 & 0 & 0 & 0 \\ 1 & 0 & 0 & 0 & 0 & 0\\ 0 &
    0 & * & 1 & 0 & 0\\ 0 & 0 & * & 0 & 1 & 0\\ 0 & 0 & * & 0 & 0 & 1\end{matrix} \right )
\end{equation*}
We can compute the geometric hypercube map explicitly by taking any
subset of the $*$-variables, setting the remaining $*$-variables to
zero, and then describing in which Schubert cell we lie. (It is a
consequence of Proposition \ref{prop:intersections} that the resulting
Schubert cell only depends on which $*$-variables are non-zero.) In
any specific example this is easily decided by performing row and
column operations.

\begin{ex} \label{ex:hypercube calcs}
  Here we give two examples when $n = 3$:
  \begin{enumerate}
  \item Suppose that $x = (0,1,2,3)$ is the identity. Consider a
    general element of $S_x^H$:
\[ \left ( \begin{matrix}
      1 & 0 & 0 & 0 \\ a & 1 & 0 & 0 \\ b & 0 & 1 & 0 \\ c & 0 & 0 &
      1 \end{matrix} \right )\]
    If $c \ne 0$, then by rescaling the
    last row and performing
    row operations upwards followed by column operations rightwards we see that
    \[
\left ( \begin{matrix}
      1 & 0 & 0 & 0 \\ a & 1 & 0 & 0 \\ b & 0 & 1 & 0 \\ c & 0 & 0 &
      1 \end{matrix} \right )  \sim
  \left ( \begin{matrix}
      0 & 0 & 0 & -c^{-1} \\ 0 & 1 & 0 & -ac^{-1} \\ 0 & 0 & 1 & -bc^{-1} \\ 1 & 0 & 0 &
      c^{-1} \end{matrix} \right ) \sim
    \left ( \begin{matrix}
      0 & 0 & 0 & 1 \\ 0 & 1 & 0 & 0 \\ 0 & 0 & 1 & 0 \\ 1 & 0 & 0 &
      0 \end{matrix} \right ) 
      \]
all lie in the same Bruhat cell. In particular $\theta(I) = t_{0,3}$
if $3 \in I$ (i.e. if $c \ne 0$). Similar arguments show that when $x
= 0$ the hypercube map only
depends on the last non-zero element in the first column. In other
words
\[
 \theta(I) = t_{(0,\max I)} \quad \text{for all $I$.} \]
\item Consider the case $x = (0,3,2,1)$. Again, by upwards row and
  rightwards column operations we see:
\[  \left ( \begin{matrix}
      1 & 0 & 0 & 0 \\ 0 & 0 & 0 & 1 \\ 0 & 0 & 1 & 0 \\ c & 1 & 0 &
      0 \end{matrix} \right )  \sim
  \left ( \begin{matrix}
      0 & -c^{-1} & 0 & 0 \\ 0 & 0 & 0 & 1 \\ 0 & 0 & 1 & 0 \\ c & 1 & 0 &
      0 \end{matrix} \right )  \sim
  \left ( \begin{matrix}
      0 & 1 & 0 & 0 \\ 0 & 0 & 0 & 1 \\ 0 & 0 & 1 & 0 \\ 1 & 0 & 0 &
      0 \end{matrix} \right )
\]
Thus $\theta(\{3\}) = (3,0,2,1) = t_{(0,3)}x$. On the other and if
$a, b$ and $c$ are all non-zero then
\[
  \left ( \begin{matrix}
      1 & 0 & 0 & 0 \\ a & 0 & 0 & 1 \\ b & 0 & 1 & 0 \\ c & 1 & 0 &
      0 \end{matrix} \right )  \sim
    \left ( \begin{matrix}
      0 & 0 & 0 & 1 \\ 0 & 0 & 1 & 0 \\ 0 & 1 & 0 & 0 \\ 1 & 0 & 0 &
      0 \end{matrix} \right )  
\]
One can check this by hand. A more efficient way is to notice that the
corner rank matrices of both matrices are given by
\[
  \left ( \begin{matrix}
      1 & 2 & 3 & 4 \\ 1 & 2 & 3 & 3 \\ 1 & 2 & 2 & 2 \\ 1 & 1 & 1 &
      1 \end{matrix} \right )  \]
and hence both belong to the same Schubert cell by the results of
\S\ref{sec:equations}. Thus $\theta(\{1,2,3\}) = w_0$.
\end{enumerate}

\end{ex}

We now give an explicit general description of the hypercube map. We will make
the following simplifications:
\begin{enumerate}
\item The subvariety $S_{x,y}^H$ can be viewed as a coordinate
  subspace inside $S_{x}^H$. In particular, the geometric
  hypercube map for $S_{x,y}^H$ is the restriction of that for
  $S_x^H$. Thus it is enough to describe the hypercube map for
  $S_x^H$.
\item Recall that $m = x^{-1}(0)$. The free variables in $S_x^H$ can
  be identified with the set $x(m+1), x(m+1), \dots, x(n)$. (In this
  indexing, the variable $i$ corresponds to the unique free variable
  in $S_x^H$ which occurs in the $i^{th}$-row. These are exactly the
  free variables which have a $1$ to their right in $S_x^H$.) 
\end{enumerate}

Thus,
  from now on we regard the hypercube map as a map
  \[
\theta_{\geom} : \left \{ \begin{array}{c}\text{subsequences of }\\\text{$x(m+1), x(m+2), \dots,
  x(n)$}\end{array} \right \} \to [x, w_0]
\]
which maps the empty subsequence to $x$. Note as the the elements of
the set $x(m+1), x(m+2), \dots,  x(n)$ are distinct, subsequences and
subsets are the same thing. However, in the following the order in
which entries occur in the subsequence will play an important role in
determining the map.

Consider a non-empty subsequence $I$ of $x(m+1), x(m+2), \dots,
x(n)$. Define a sequence $i_1', i_2', \dots$ inductively as follows:
\begin{align*}
  i_1 &= \text{maximum of the entries in $I$,} \\
  i_{k+1} &= \text{maximum of the entries right of $i_k'$ in $I$.}
\end{align*}
For some $l$, the set of entries right of $i_l$ is
empty, in which case $i_{l+1}$ is undefined. We let
\[
I_{\decr} = (i_1, i_2, \dots, i_l)
\]
denote the resulting subsequence.

\begin{remark}
  One may consider $I \mapsto I_{\decr}$ as the extraction of a ``greedy
  decreasing subsequence'': first choose the maximal element of $I$,
  then choose a maximal element right of it, etc.
\end{remark}

Now consider the $k+1$-cycle $\sigma_I$
which sends:
\[
i_l \mapsto i_{l-1} \mapsto \dots i_2 \mapsto i_1 \mapsto 0 \mapsto i_l
\]
We define
\[
\theta_{\geom}'(I) := \sigma_I x.
\]
If $I$ is empty, we set $\theta_{\geom}'(I) := x.$ We obtain in this way a map
  \[
\theta_{\geom}' : \left \{ \begin{array}{c}\text{subsequences of }\\\text{$x(m+1), x(m+2), \dots,
  x(n)$}\end{array} \right \} \to S_{n+1}.
\]
The following shows that this explicit description does indeed describe the geometric
hypercube map.

\begin{thm}
  $\theta_\geom = \theta_\geom'$.
\end{thm}

\begin{proof}
In the discussion below we make use of the description of
Schubert cells in terms of the corner rank matrix recalled in \S
\ref{sec:equations}. In particular, we will use the notation $g^{\le
  (p,q)}$ introduced there.\footnote{We are grateful to A.~Henderson
  for suggesting the proof below. All inaccuracies are ours!}

Consider the permutation matrix $x$. Its corner rank matrix is given by
\[
\text{rank }x^{\le (p,q)} = \# \{ \text{elements in $x(0), x(1), \dots,
  x(q)$ which are $\ge p$} \}.
  \]
On the other hand, if $g$ denotes an element of the stratum $(S_x^H)_I$ (so
that $I$ is a subset of $\{ x(m+1), \dots, x(n) \}$), its corner rank
matrix is computed as follows:
\begin{equation} \label{eq:g cornerrank}
  \text{rank }g^{\le (p,q)} = \begin{cases}\text{rank } x^{\le (p,q)}
    & \text{if $q < m$,} \\
    \text{rank } x^{\le (p,q)} &
  \text{if $I \cap \{ \ge p \}  \subset \{ x(0), x(1), \dots,
    x(q) \}$,} \\
  \text{rank } x^{\le (p,q)}  + 1 & \text{otherwise}.
\end{cases}
\end{equation}
(Indeed, $x$ and $g$ differ only at the $m^{th}$-column. This column
is not contained in $g^{\le (p,q)}$ if $q < m$. Otherwise, the
condition $I \cap \{ \ge p \}  \subset \dots $ expresses whether this new column is in the span of the
    others in $g^{\le (p,q)}$.)

    Now consider the permutations $x$ and $\sigma_Ix$ in string
    notation (i.e. as the sequences $(x(0), x(1), \dots, x(n))$ and
    $(\sigma_Ix(0), \sigma_Ix(1), \dots, \sigma_Ix(n))$). In the
    passage from $x$ to $\sigma_Ix$ the elements $I_{\decr}$ move left, and
    $0$ moves all the way to the right, to the position occupied by the final
    element of $I_{\decr}$. In particular, if we consider the elements which
    are larger than $p$, then those belonging to $I_{\decr}$ move to the
    left. It follows that:
\begin{align} \label{eq:sigmax cornerrank}
  \text{rank }(\sigma_Ix)^{\le (p,q)} &=
   \begin{cases} \text{rank } x^{\le (p,q)}
    & \text{if $q < m$,} \\
    \text{rank } x^{\le (p,q)} &
  \text{if $I_{\decr} \cap \{ \ge p \}  \subset \{ x(0), x(1), \dots,
    x(q) \}$,} \\
  \text{rank } x^{\le (p,q)}  + 1 & \text{otherwise}.
\end{cases}
\end{align}
We claim that if $q \ge m$, then
\begin{equation} \label{eq:Iiff}
  I \cap \{ \ge p \}  \subset \{ x(0), x(1), \dots,
    x(q) \} \Leftrightarrow
  I_{\decr} \cap \{ \ge p \}  \subset \{ x(0), x(1), \dots,
  x(q) \}. 
  \end{equation}
  The implication $\Rightarrow$ is obvious, as $I_{\decr}$ is a subset of
  $I$. For the other direction, assume that  $I \cap \{ \ge p \}  \not
  \subset \{ x(0), x(1), \dots, x(q) \}$ and let $j$ belong to $I \cap
  \{ \ge p \}$ but not to $\{ x(0), x(1), \dots, x(q) \}$. Then, by
  definition of $I_{\decr}$, there exists some element $j'$ in $I_{\decr}$ which
  occurs to the right of $j$ in $x$, and is larger than $j$. Thus $j'$
  belongs to $I_{\decr} \cap \{ \ge p \}$ but not to $\{ x(0), x(1), \dots,
  x(q) \}$. Thus \eqref{eq:Iiff} holds.
  
Thus the corner rank matrices described by \eqref{eq:g cornerrank} and
\eqref{eq:sigmax cornerrank} coincide. We deduce that $g$ and
$\sigma_Ix$ belong to the same Schubert cell. Thus
\[
\theta_{\geom}(I) = \theta'_\geom(I)
\]
which proves the theorem.
\end{proof}

\begin{ex} \label{ex:extreme}
  The extreme examples of the hypercube map are the following:
  \begin{enumerate}
  \item If $x = (0,1, \dots ,n)$ is the identity then $m = 0$ and for any
    sequence of $(1, \dots, n)$, $I = (\max I)$ because $x$
  contains no descending sequences whatsoever. Hence
  \[
\theta_\geom(I) = t_{0, \max I}.
\]
In this case the image of $\theta_\geom$ is isomorphic to the totally
ordered set
\[
  t_{(0,1)} \le t_{(0,2)} \le \dots \le t_{(0,n)}.
\]
The image of $\theta_\geom$ is illustrated (for $n = 3$) in Figure
\ref{fig:0123}. This example matches the explicit calculation in Example
\ref{ex:hypercube calcs}(1).
  \item If $x = (0,n,\dots, 2, 1)$ then $m = 0$ and for any
    subsequence $I$ of $(n,\dots, 2, 1)$, we have $I_{\decr}= I$ in descending
  order. Hence each $\sigma_I$ is distinct and
  \[
\theta : \text{subsequences  of $(n, \dots, 1,0)$} \to S_n
\]
is injective. It identifies the lattice of subsets $\{ 1, \dots, n
\}$ with the lattice between $x$ and $w_0$ in $S_n$. The image of
$\theta$ is illustrated (for $n = 3$) in Figure \ref{fig:0321}. (This
example generalizes the explicit calculations in Example 
\ref{ex:hypercube calcs}(2).) 

That such a nice
lattice is present in Bruhat order might come as a surprise to the reader. It is less surprising when
one realizes that $c = w_0x = (n, 0, 1, \dots, n-1)$ is a Coxeter
element, and hence $w_0$ gives an order reversing bijection between
the interval $[x,w_0]$ and the interval $[\id,c]$. The interval
$[\id,c]$ is easily seen to be isomorphic to a $\{ 0, \dots, n\}$-hypercube.
  \end{enumerate}
\end{ex}

\begin{figure} \label{fig:0123}
  \centering
  \includegraphics[scale=0.4]{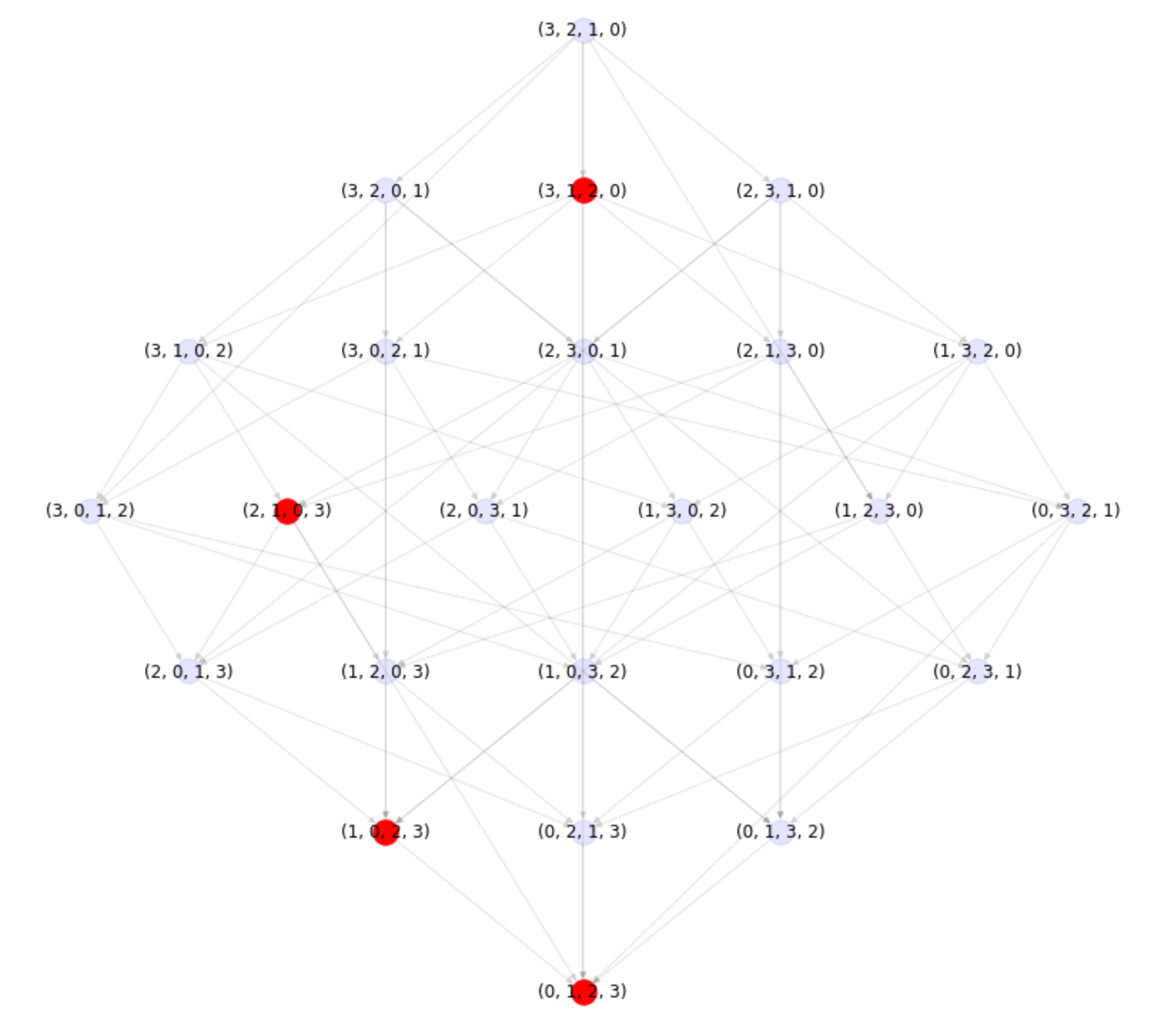}
  \caption{The image of the hypercube map based at $x = 0123$. The
    reader can check that the red dots form a subposet isomorphic to a chain $3 \to 2
  \to 1 \to 0$.}
\end{figure}

\begin{figure}\label{fig:0321}
  \centering
  \includegraphics[scale=0.4]{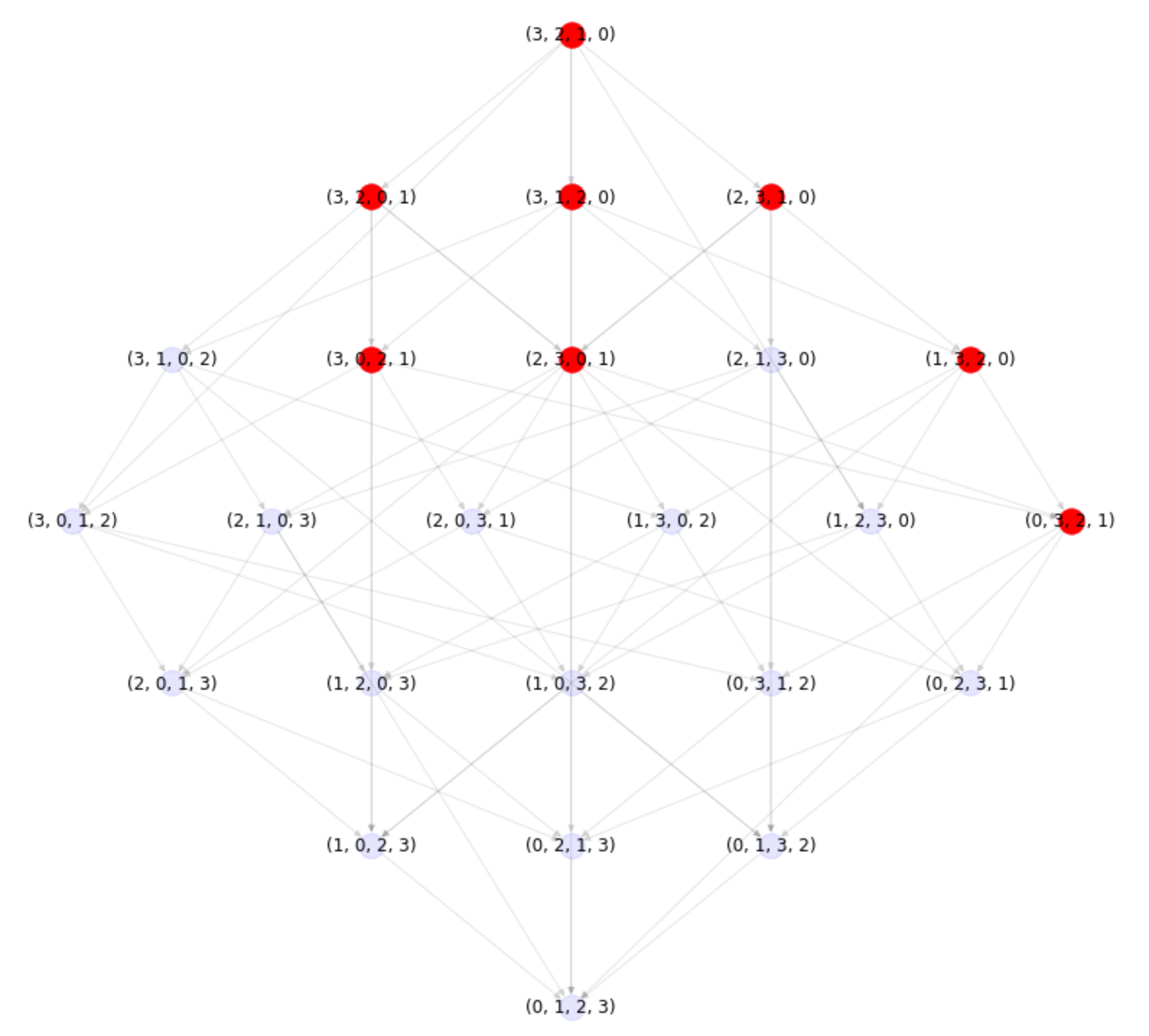}
  \caption{The image of the hypercube map based at $x = 0321$. The reader
  can check that the red dots form subposet isomorphic to the
  1-skeleton of a $\{0,1,2,3\}$-hypercube.}
\end{figure}

\subsection{The hypercube map and Bruhat graph} In this section we
prove that the geometric hypercube map is determined by what it does to
subsequences of $(x(m+1), \dots, x(n))$ of length $0$ and $1$. In
particular, we will see that the edges
\[
  \{ \theta (\{ i \}) \to x \; | \; i \in (x(m+1), \dots, x(n))\}
\]
in the Bruhat graph of the interval $[x,y]$ span a hypercube cluster,
and that the geometrical hypercube map and combinatorial hypercube maps
(defined in \S \ref{sec:hypercube piece}) agree.

In the following, the following standard fact on Bruhat order for the
symmetric group will be indespensible (see e.g. \cite[Lemma
2.1]{Billey-pattern}):

\begin{prop}[``Subword Condition''] Suppose that $u$ and $v$ are
  permutations in $S_{n+1}$ such that $u(i) = v(i)$ except in
  positions $i_1 < i_2 < \dots < i_m$. Consider the permutations $u'$
  (resp. $v'$) of the set $\{ u(i_j) \; | \; 1 \le j \le m\}$ (with
  its natural increasing ordering) given in
  string notation by $(u(i_1),u(i_2), \dots,
  u(i_m))$ (resp. $(v(i_1),v(i_2), \dots, v(i_m))$. Then $u \le v$ if and only if $u' \le v'$.
\end{prop}

Recall that $x$ and $y$ are fixed permutations in $S_{n+1}$, and that
  \[
\theta_{\geom} : \left \{ \begin{array}{c}\text{subsequences of }\\\text{$x(m+1), x(m+2), \dots,
  x(n)$}\end{array} \right \} \to [x, w_0]
\]
denotes the geometric hypercube map. The combinatorial characterisation of the hypercube map discussed
above will be a consequence of two simple propositions:

\begin{prop} \label{prop:swapsies}
Suppose that $i$ occurs to
  the left of $j$ in $I$.
  \begin{enumerate}
  \item If $i > j$ then $\theta_\geom(\{i\})$ and $\theta_\geom(\{j\})$ are
    incomparable.
  \item If $i < j$ then $\theta_\geom(\{i\}) < \theta_\geom(\{j\})$ 
  \end{enumerate}
\end{prop}

\begin{proof}
The hypercube map
  in question only disturbs $0, i$ and $j$. Hence, by the Subword Condition, 
  we can assume $x = (0, i, j)$ viewed as a
  permutation of $\{ 0, i, j\}$. In both cases we have
  \[ \theta_\geom(\{i\}) = (i,0,j) \quad \text{and} \quad \theta_\geom(\{j\}) = (j,i,0) \]
If $i < j$ then $(j,i,0) <  (i,0,j)$ establishing (2). If $i > j$ then
$(i,0,j)$ and $(j,i,0)$ are distinct and both of length 2, and hence incomparable,
establishing (1).
\end{proof}

\begin{remark}
  We could have also deduced this proposition from the $n = 2$ case of
  Example \ref{ex:extreme}.
\end{remark}

\begin{prop} \label{prop:map hypercube}
  Suppose that $I$ is a subsequence of $(x(m+1), \dots,
  x(n))$ such that $I_{\decr} = I$. (That is, $I_{\decr}$ is decreasing.) Then
  $\theta_\geom$ gives the unique embedding of directed graphs
  \[
H_I \to [x,w_0]
\]
such that $\emptyset \mapsto x$ and $\{ i \} \mapsto
\theta_\geom(\{i\})$ for all $i \in I$.
\end{prop}

\begin{proof}
Let $I = (i_1, \dots, i_l)$. By the Subword Condition we can work
with permutations of the (ordered) set $0, i_l, \dots, i_0$, and
assume that
\[
x = (0, i_0, \dots, i_l).
\]
However we have seen in Example \ref{ex:extreme}(2) that the interval
between $\theta_\geom(\emptyset) = x$ and $\theta(I) = (i_0, \dots,
i_l,0)$ is isomorphic to an $I$-hypercube. The proposition follows, as
any morphism between $I$-hypercubes is uniquely
determined by what it does on the base vertex and vertices of height $1$.
\end{proof}

We are now ready to give:

\begin{proof}[Proof of Theorem \ref{thm:geom hypercube}] We first need
  to show that the edges $\{ z \to x \; | \; z \notin L \}$ span a
  hyperube, and that the combinatorial and geometric hypercube maps
  agree. First note that $tx \in L$ if and only if $t(0) \ne 0$, in
  other words that $t = (0, l)$ for some $l$. In order for $tx > x$ we
  must have $x^{-1}(l) > x^{-1}(0)$. Thus the edges incident to $x$ in
  both the combinatorial and geometric hypercube may be identified
  with the set
  \[
E = \{ tx \to x \; | \; t = (0, l)\text{ for some $l \in \{ x(m+1),
  \dots, x(n)\}$, such that $x < tx \le y$.}\} 
\]
It will be convenient to consider
\[ E' = \{ l \in \{ x(m+1),
    \dots, x(n)\} \; | \; x < t_{(0,l)}x \le y\}.
\]
Of course, $E$ and $E'$ are obviously isomorphic, and we let $\phi : E
\to E'$ denote the obvious isomorphism.

The reason to distinguish these two sets is to compare two partial
orders on them. Consider a partial order on $E$ given by $\alpha \le
\beta$ of their are $\le$ in Bruhat
order. On the other hand, say $i \prec j$ in $E'$ if $i$ occurs to the
left of $j$ and $i < j$ (in the usual sense, i.e. as numbers!). Now
Proposition \ref{prop:swapsies} tells us that
\[
\alpha < \beta \Leftrightarrow \phi(\alpha) \prec \phi(\beta).
\]
In other words, $E$ and $E'$ are isomorphic as posets.

Now consider a subset $F$ of $E$. To compute the image of $F$ under
the combinatorial hypercube map we first form the subset $F_{\max}
\subset F$ consisting of maximal elements, and then define
\[
\theta(F) = \text{crown of hypercube spanned by $F_{\max}$.}
\]

On the other hand, consider a subset $I \subset E'$. To compute the
image of $I$ under the geometric hypercube map we first form the
subset $I_{\decr} \subset I$ consisting of a greedy decreasing
subsequence. Thanks to Proposition \ref{prop:map hypercube} we may
describe the image of $I$ as
\[
\theta_{\geom}(I) = \text{crown of hypercube spanned by $\{
  \theta_{\geom}(\{i\}) \to x \; | \; i \in I'\}$} \]
Via the isomorphism $\phi$ the formation of
$F_{\max} \subset F$ and $I_{\decr} \subset I$ correspond to each
other. We conclude that the combinatorial and geometric hypercube maps
agree. It also follows from Proposition \ref{prop:map hypercube}  that
the edges $E$ span a hypercube cluster.

Finally, note that there was nothing special about $x$ in the above
argument. In particular, for any $z \in L$ the set
\[
\{ u \to z \; | \; u \notin L \}
\]
spans a hypercube cluster. In particular, $L \subset [x,y]$ is a
hypercube decomposition, and the theorem is proved.
\end{proof}

\section{Geometry of the conjecture}  \label{sec:conjecture}
Despite considerable effort we are unable to prove Conjecture
\ref{conj:main}. Our conjecture is motivated by Theorem
\ref{thm:main}, and hence it is tempting to imitate its proof. In this
section we discuss the obstacles remaining in doing so.

\subsection{Subvarieties of slices} In this section, we explain the natural
generalizations (for any hypercube decomposition) of the subvarieties $S_{x,y}^H$ and $S_{x,y}^I$ of the
slice $S_{x,y}$ (see \S\ref{sec:slices}).

Recall the coordinate characters $\e_i$ of our torus defined in
\S\ref{sec:hypercube_geometry} and let $R^+ = \{ \e_i - \e_j \; | \; i
< j \}$ denote positive roots with respect to $B$.
Recall that we have a bijection between reflections $t_{(i,j)}$ in
$S_{n+1}$ and \emph{negative} roots $\e_j - \e_i$. We use this to
relabel the edges of our Bruhat graph with negative roots.

As above, let $S_{x,y}$ denote the standard slice through $xB/B$
inside the Schubert variety $\overline{ByB/B.}$ Because we are in type
$A$, we have a $T$-equivariant embedding
\begin{equation}
  \label{eq:embed}
\iota:  S_{x,y} \into T_x (S_{x,y}) = \bigoplus_{u \stackrel{\alpha}{\to} x \in E} \CM_\alpha
\end{equation}
where $T_x$ denotes the tangent space at $x$, $E$ denotes the edges with target $x$ in the Bruhat graph
$[x,y]$, and $\CM_\alpha$ denotes the one-dimensional $T$-module on
which $T$ acts via the character $\alpha$.\footnote{Indeed, any attractive
fixed point $z$ has an affine neighbourhood 
that may be embedded inside the tangent space at $x$, see \cite[Proof
of Theorem 17]{BrionECEIT}. In 
type $A$ it is known \cite{LS-Schubert} that this
tangent space is spanned by its $T$-invariant curves. Hence the
statement.}

Suppose fixed a choice of hypercube decomposition $J
\subset [x,y]$. This breaks the set $E$ into two ``inductive'' and ``hypercube'' pieces $E = E_{\ind}
\sqcup E_{\hyp}$ where
\[
E_{\ind}= \{ u \to x \in E \; | \; u \in J \} \quad \text{and} \quad 
E_{\hyp}= \{ v \to x \in E \; | \; v \notin J\}.
\]
To simplify notation we denote by
\[
V = \bigoplus_{u \stackrel{\alpha}{\to} x \in E} \CM_\alpha
\]
the vector space with $T$-action into which $S_{x,y}$ embeds.  Our decomposition of $E$ leads
to a direct sum decomposition
\[
V = V_\ind \oplus V_\hyp
\]
and it is tempting to consider the subvarieties
\[
S_{x,y}^I := S_{x,y} \cap V_\ind \quad \text{and} \quad S_{x,y}^H :=
S_{x,y} \cap V_\hyp.
\]

Recall that $J = \{ \le z \} \cap [x,y]$ for some $ \in [x,y]$. The
following is the analogue of Proposition \ref{prop:slice_equal}:

\begin{prop} \label{prop:inductive slice general}
  We have $S_{x,y}^I =  S_{x,z}$.
\end{prop}

The proof is satisfying, and neatly explains the ``diamond complete''
condition in the definition of a hypercube decomposition. We
give it in \S\ref{sec:schnitt}.

The following should be the analogue of Proposition \ref{prop:intersections}
\begin{conj} \label{conj:hyp}
  We have $S_{x,y}^H = V_{\hyp}$. Moreover, each toric stratum
  (corresponding to a subset $I  \subset E_{\hyp}$) is the
  intersection of $S_{x,y}^H$ with a unique Schubert cell $BuB/B$ and
  we have
  \[
\theta(I) = u
\]
where $\theta$ is the hypercube map associated to our
hypercube decomposition $J$.
\end{conj}

We believe we have a proof of this conjecture. However the argument is
unsatisfying and long. We are willing to try to write it down if the
obstacle to be discussed in \S \ref{sec:obstacle} can be overcome.

\begin{remark} Two remarks of a hygienic nature:
  \begin{enumerate}
  \item The ($T$-equivariant) embedding $\iota$ inside the tangent space involves a
    choice of ($T$-equivariant) sections
of the quotient map
\[
\mg_x \onto \mg_x/\mg_x^2
\]
where $\mg_x$ denotes the maximal ideal at $x$. Thus, although the
right-hand side of \eqref{eq:embed} is canonical, the embedding
isn't in general. Thus we need to be careful with constructions based
on $\iota$. For example, it is not clear a priori that $S_{x,y}^H$ and
$S_{x,y}^I$ are well-defined. It turns out that they are, as follows
from some considerations in equivariant geometry which we
omit.
\item The embedding $\iota$ is not the same as the embedding of
  $S_{x,y}$ considered in \S\ref{sec:slices}. One can get a choice of
  embedding $\iota$ by composing the embedding considered in
  \S\ref{sec:slices}, with a projection to those coordinates which
  span one-dimensional orbits in $S_{x,y}$. This gives a canonical
  $\iota$ which is useful in examples.
  \end{enumerate}
\end{remark}

\subsection{The main obstacle} \label{sec:obstacle} In this section we
explain the main difficulty in generalizing our proof. We assume that
Conjecture \ref{conj:hyp} is known. To simplify notation, let
\[
X = \PM_\lambda \dot{S}_{x,y}
\]
and let $Z$ and $Z'$ denote closed subvarieties of $X$ given by the
images of $\dot{S}_{x,y}^H$ and $\dot{S}_{x,y}^I$ respectively. We
have the following diagram of inclusions:
\[
  \begin{tikzpicture}[xscale=1.4,yscale=.55]
    \node (ul) at (-2,2) {$X - Z'$};
    \node (ur) at (2,2) {$X - Z$};
    \node (m) at (0,0) {$X$};
    \node (ll) at (-2,-2) {$Z$};
    \node (lr) at (2,-2) {$Z'$};
    \draw[->] (ll) to node[above] {$i$} (m);
    \draw[->] (lr) to node[above] {$i'$} (m);
    \draw[->] (ul) to node[above] {$j$} (m);
    \draw[->] (ur) to node[above] {$j'$} (m);
    \draw[->] (ll) to (ul);
    \draw[->] (lr) to (ur);
  \end{tikzpicture}
\]
As above, we denote by $\ic$ the intersection cohomology
complex on $X$, and by $T' = T/\lambda(\CM^*)$ the torus acting on
$X$.

We have a commutative diagram of cohomology groups:
\begin{equation} \label{eq:big diagram}
\begin{array}{c}  \begin{tikzpicture}[xscale=1.4,yscale=.55]
    \node (ul) at (-2,2) {$H^\bullet_{T',c}(X - Z, j^!\ic)$};
    \node (ur) at (2,2) {$H^\bullet_{T'}(X-Z, j^*\ic)$};
    \node (m) at (0,0) {$H^\bullet_{T'}(X,\ic)$};
    \node (ll) at (-2,-2) {$H^\bullet_{T'}(Z, i^!\ic)$};
    \node (lr) at (2,-2) {$H^\bullet_{T'}(Z', (i')^*\ic)$};
    \draw[->] (ll) to node[above] {$\alpha$} (m);
    \draw[<-] (lr) to node[above] {$\beta$} (m);
    \draw[->] (ul) to (m); 
    \draw[<-] (ur) to (m); 
    \draw[->] (ll) to node[left] {$\xi$}  (ul);
    \draw[<-] (lr) to node[right] {$\zeta$}  (ur);
  \end{tikzpicture} \end{array}
\end{equation}
(Here $H^\bullet_{T',c}(X - Z, j^!\ic) = H^\bullet_{T'}(X, j_!j^!\ic)$
denotes compactly supported cohomology.) We know:
\begin{enumerate}
\item The diagonal
  \[
H^\bullet_{T'}(Z, i^*\ic) \stackrel{\alpha}{\longto} H^\bullet_{T'}(X,\ic) \longto H^\bullet_{T'}(X-Z, j^*\ic)
\]
and anti-diagonal
\[
H^\bullet_{T',c}(X - Z, j^!\ic) \longto H^\bullet_{T'}(X,\ic)  \stackrel{\beta}{\longto} H^\bullet_{T'}(Z', (i')^*\ic)
  \]
 are part of the long exact sequence for the pair $(X, X-Z)$
 (resp. $(X, Z')$).
\item The groups $H^\bullet_{T'}(Z, i^*\ic)$ and $H^\bullet_{T'}(Z',
  (i')^*\ic)$ are pure. (For the first group, this follows from
  Conjecture \ref{conj:hyp}, and for the second this may be assumed by
  induction.)
\end{enumerate}

  It will be useful to recall the magic that happened where our theorem
  holds. In that case $X - Z$ retracts equivariantly onto $Z'$ (using the character
  $\gamma$ in \S\ref{sec:inductive_geometry}) and hence: $\zeta$ is an
  isomorphism, $H^\bullet_{T'}(X-Z, i^*\ic)$ is pure, and 
  \[
H^\bullet_{T'}(Z, i^*\ic) \stackrel{\alpha}{\to} H^\bullet_{T'}(X,\ic) \stackrel{\beta}{\to} H^\bullet_{T'}(Z', (i')^*\ic)
\]
is a short exact sequence. Thus, both $\xi$ and $\zeta$ are
isomorphisms in this case and \eqref{eq:big diagram} collapses to a
single short exact sequence.

We do not know whether $\zeta$ and $\xi$ are isomorphisms in
general. However, if one is, then so is the other one, and our
conjecture holds. Recall that we say that $H_{T'}^\bullet(X, \FS)$ is
equivariantly formal if it is free as a
$H_{T'}^\bullet(\pt,\QM)$-module. The following are easily seen to be
equivalent:
\begin{enumerate}
\item $\beta$ is surjective,
\item $H^\bullet_{T',c}(X - Z, j^!\ic)$ is pure and equivariantly formal.
\end{enumerate}
Similarly, we have equivalences between the following statements:
\begin{enumerate}
\item[(3)] $\alpha$ is split injective (or equivalently, its
  non-equivariant analogue is injective),
\item[(4)] $H^\bullet_{T',c}(X - Z', j^*\ic)$ is pure and equivariantly formal.
\end{enumerate}
It seems that one could use equivariant localization to reduce to rank
2 and deduce equivalences between:
\begin{enumerate}
\item[(5)] $H^\bullet_{T',c}(X - Z, j^!\ic)$ is pure and equivariantly
  formal,
\item[(6)] $\xi$ is an isomorphism.
\end{enumerate}
Similarly, localization techniques should show equivalences between:
\begin{enumerate}
  \item[(7)] $H^\bullet_{T',c}(X - Z', j^*\ic)$ is pure and equivariantly
    formal,
  \item[(8)] $\zeta$ is an isomorphism.
  \end{enumerate}
  Thus it seems likely that (1)--(8) are all equivalent, and any one
  of them implies our conjecture.
  
\subsection{Bruhat polytopes and the inductive piece} \label{sec:schnitt}
In this final section we prove Proposition \ref{prop:inductive slice
  general}. We include this proof because we think it is enlightening,
and explains neatly why the ``diamond complete'' condition in the
definition of a hypercube decomposition arises. Recall that $z$
denotes the crown of the inductive piece, i.e. $J = \{ \le z \}$. We
want to show:

\begin{prop} \label{prop:inductive} $S_{x,y} \cap V_{\ind}= S_{x,z}$.
\end{prop}

We will need some results from the beautiful paper \cite{GS} before turning to the proof.
Let us fix throughout a regular dominant weight $\lambda$. In the
following we will identify $W$ with its image in the $\lambda$ orbit,
i.e. $x$ is identified with $x(\lambda) \in \hg^*$, where $\hg^*$
denotes the realification of the character lattice of $T$. We also consider
$G/B$ as embedded by the Pl\"ucker embedding:
\[
\Plucker : G/B \into \PM(V_\lambda).
  \]

Take a
point $p \in G/B$. Consider its $T$-orbit closure
$\overline{T \cdot p}$. Gelfand and Seganova prove the following:
\begin{enumerate}
\item The set $\{ u \; | \; u \in
  (\overline{T \cdot p})^{T} \} \subset \hg^*$ are the vertices of a convex
  polytope $\Poly(p) \subset \hg^*$.
\item $\Poly(p)$ agrees with the moment map image of $\overline{T
    \cdot p}$;
\item $x \in \Poly(p)$ if and only if the $x$-coordinate of the
  Pl\"ucker coordinate is non-zero. (The $x(\lambda)$-weight
  space in $V_\lambda$ is one-dimensional, so this statement makes
  sense.)
\end{enumerate}
Let us call a polytope that occurs in this way (with $\lambda$ fixed
as always) a \emph{flag
  polytope}. The theory of toric varieties also produces the important
fact (also proved in \cite{GS}) that:
\begin{equation}
  \label{eq:faces}
  \text{Any face of a flag polytope is a flag polytope.}
\end{equation}
From \eqref{eq:faces} we deduce (by considering one-dimensional faces) that:
\begin{equation}
  \label{eq:edges}
  \text{Any edge of $\Poly(p)$ is parallel to a root.}
\end{equation}
We also need the classification of 2-dimensional
flag polytopes. They are the following, together with their horizontal
flips:
\begin{align*}
\begin{array}{c}
  \begin{tikzpicture}[scale=0.45]
    \node (t) at (90:2) {$\bullet$};
    \node (l) at (180:2) {$\bullet$};
    \node (r) at (0:2) {$\bullet$};
    \node (b) at (-90:2) {$\bullet$};
  \draw[->] (t) to (l); 
  \draw[->] (t) to (r); 
  \draw[->] (l) to (b);
  \draw[->] (r) to (b);
\end{tikzpicture}
\end{array}, \;
\begin{array}{c}
  \begin{tikzpicture}[scale=0.5]
    \node (t) at (90:2) {$\bullet$};
    \node (lt) at (150:2) {$\bullet$};
    \node (lb) at (-150:2) {$\bullet$};
        \node (b) at (-90:2) {$\bullet$};
  \draw[->] (t) to (lt); 
  \draw[->] (lt) to (lb); 
  \draw[->] (lb) to (b);
  \draw[->] (t) to (b);
\end{tikzpicture}
\end{array}, \;
    \begin{array}{c}
  \begin{tikzpicture}[scale=0.5]
    \node (t) at (90:2) {$\bullet$};
    \node (lt) at (150:2) {$\bullet$};
    \node (lb) at (-150:2) {$\bullet$};
    \node (rt) at (30:2) {$\bullet$};
  \draw[->] (t) to (lt); \draw[->] (t) to (rt);
  \draw[->] (lt) to (lb); 
  \draw[->] (rt) to (lb);
\end{tikzpicture}
    \end{array}, \;
    \begin{array}{c}
  \begin{tikzpicture}[scale=0.5]
    \node (lb) at (-150:2) {$\bullet$};
    \node (rt) at (30:2) {$\bullet$};
    \node (rb) at (-30:2) {$\bullet$};
        \node (b) at (-90:2) {$\bullet$};
  \draw[->] (rt) to (lb); \draw[->] (rt) to (rb);
  \draw[->] (lb) to (b); \draw[->] (rb) to (b);
\end{tikzpicture}
                \end{array}, \;
  \begin{array}{c}
  \begin{tikzpicture}[scale=0.5]
    \node (t) at (90:2) {$\bullet$};
    \node (lt) at (150:2) {$\bullet$};
    \node (lb) at (-150:2) {$\bullet$};
    \node (rt) at (30:2) {$\bullet$};
    \node (rb) at (-30:2) {$\bullet$};
        \node (b) at (-90:2) {$\bullet$};
  \draw[->] (t) to (lt); \draw[->] (t) to (rt);
  \draw[->] (lt) to (lb); \draw[->,gray!50!white] (lt) to (rb);
  \draw[->,gray!50!white] (rt) to (lb); \draw[->] (rt) to (rb);
  \draw[->] (lb) to (b); \draw[->] (rb) to (b);
  \draw[->,gray!50!white] (t) to (b);
\end{tikzpicture}
                \end{array}.
\end{align*}
(The gray edges in the last diagram indicated edges in the Bruhat
graph which are not edges of the polytope.) Note that these pictures
represent conformal equivalence classes of polytopes: we can stretch
edge lengths, but angles must be preserved. For example, there is no
need for the hexagon above to be regular, but all angles
between edges (including gray edges) must be as displayed.

\begin{proof}[Proof of Proposition \ref{prop:inductive}.]
  The inclusion $S_{x,y} \cap I \supset S_{x,z}$ is immediate. The tricky
bit is to establish that $S_{x,y} \cap I \subset S_{x,z}$.

  Choose $p \in S_{x,y} \cap I$. It is enough to show that if $u \in
  \Poly(p)$ then $u \le z$. (Note $\lim_{z \to \infty} \lambda(z)
  \cdot p = u'$ for some $u' \in \Poly(p)$, where $\lambda$ is our
  anti-dominant cocharacter as always. Thus $p$ belongs to
  $Bu'B/B \cap U^{-}xB/B$ which is contained in $S_{x,z}$ if $u' \le
  z$.)

  Consider the vertices of $\Poly(p)$. We can divide these into those
  edges contained in inductive piece (which we will call \emph{blue}
  vertices) and those which are not (which we will call \emph{red}
  vertices). We will show the contrapositive to the statement of the
  last paragraph: namely we will show that if the set of red vertices
  in $\Poly(p)$ is non-zero, then $p \notin S_{x,y} \cap I$.

  So let us assume that red vertices exist.  We can certainly find a
  red vertex $u$ which is connected to a blue  vertex $v$ by an edge
  in $\Poly(p)$. Note that the edges in $\Poly(p)$ inherit a direction
  from the Bruhat graph (equivalently, from $\lambda$). Because $v
  \to v'$ implies that $v'$ is also blue, we know that the orientation
  of the edge joining $u$ to $v$ has to be $u \to v$.

  If $v = x$, then the $u$-coordinate of $\Plucker(p)$ is non-zero,
  and hence $p \notin S_{x,y} \cap I$ and we are done. Otherwise,
  there exists an edge $v \to w$ in the polytope. Now consider the
  2-dimensional face $F$ of $\Poly(p)$ containing $u, v$ and
  $v'$. From our knowledge of the 2-dimensional flag polytopes, and
  forgetting conformal structure (i.e. just remembering isomorphism
  type of directed graph) we see that there are 5 possibilities:
  \begin{align*}
\begin{array}{c}
  \begin{tikzpicture}[scale=0.45]
    \node[red] (t) at (90:2) {$u$};
    \node[blue] (l) at (180:2) {$v$};
    \node (r) at (0:2) {$\bullet$};
    \node[blue] (b) at (-90:2) {$w$};
  \draw[->] (t) to (l); 
  \draw[->] (t) to (r); 
  \draw[->] (l) to (b);
  \draw[->] (r) to (b);
\end{tikzpicture}
\end{array}, \;
\begin{array}{c}
  \begin{tikzpicture}[scale=0.5]
    \node[red] (t) at (90:2) {$u$};
    \node[blue] (lt) at (150:2) {$v$};
    \node[blue] (lb) at (-150:2) {$w$};
        \node (b) at (-90:2) {$\bullet$};
  \draw[->] (t) to (lt); 
  \draw[->] (lt) to (lb); 
  \draw[->] (lb) to (b);
  \draw[->] (t) to (b);
\end{tikzpicture}
\end{array}, \;
\begin{array}{c}
  \begin{tikzpicture}[scale=0.5]
    \node (t) at (90:2) {$\bullet$};
    \node[red] (lt) at (150:2) {$u$};
    \node[blue] (lb) at (-150:2) {$v$};
        \node[blue] (b) at (-90:2) {$w$};
  \draw[->] (t) to (lt); 
  \draw[->] (lt) to (lb); 
  \draw[->] (lb) to (b);
  \draw[->] (t) to (b);
\end{tikzpicture}
\end{array}, \;
    \begin{array}{c}
  \begin{tikzpicture}[scale=0.5]
    \node (t) at (90:2) {$\bullet$};
    \node[red] (lt) at (150:2) {$u$};
    \node[blue] (lb) at (-150:2) {$v$};
    \node (rt) at (30:2) {$\bullet$};
    \node (rb) at (-30:2) {$\bullet$};
        \node[blue] (b) at (-90:2) {$w$};
  \draw[->] (t) to (lt); \draw[->] (t) to (rt);
  \draw[->] (lt) to (lb); \draw[->,gray!50!white] (lt) to (rb);
  \draw[->,gray!50!white] (rt) to (lb); \draw[->] (rt) to (rb);
  \draw[->] (lb) to (b); \draw[->] (rb) to (b);
  \draw[->,gray!50!white] (t) to (b);
\end{tikzpicture}
                \end{array}, \;
    \begin{array}{c}
  \begin{tikzpicture}[scale=0.5]
    \node[red] (t) at (90:2) {$u$};
    \node[blue] (lt) at (150:2) {$v$};
    \node[blue] (lb) at (-150:2) {$w$};
    \node (rt) at (30:2) {$\bullet$};
    \node (rb) at (-30:2) {$\bullet$};
        \node (b) at (-90:2) {$\bullet$};
  \draw[->] (t) to (lt); \draw[->] (t) to (rt);
  \draw[->] (lt) to (lb); \draw[->,gray!50!white] (lt) to (rb);
  \draw[->,gray!50!white] (rt) to (lb); \draw[->] (rt) to (rb);
  \draw[->] (lb) to (b); \draw[->] (rb) to (b);
  \draw[->,gray!50!white] (t) to (b);
\end{tikzpicture}
                \end{array}.
\end{align*}
We claim the last possibility is impossible. Indeed, because the blue
nodes are closed under the Bruhat order, the following nodes must be blue:
\begin{align*}
    \begin{array}{c}
  \begin{tikzpicture}[scale=0.5]
    \node[red] (t) at (90:2) {$u$};
    \node[blue] (lt) at (150:2) {$v$};
    \node[blue] (lb) at (-150:2) {$w$};
    \node (rt) at (30:2) {$\bullet$};
    \node[blue] (rb) at (-30:2) {$\bullet$};
        \node[blue] (b) at (-90:2) {$\bullet$};
  \draw[->] (t) to (lt); \draw[->] (t) to (rt);
  \draw[->] (lt) to (lb); \draw[->,gray!50!white] (lt) to (rb);
  \draw[->,gray!50!white] (rt) to (lb); \draw[->] (rt) to (rb);
  \draw[->] (lb) to (b); \draw[->] (rb) to (b);
  \draw[->,gray!50!white] (t) to (b);
\end{tikzpicture}
                \end{array}.
\end{align*}
  Now diamond completeness forces all six nodes to be blue, which is a
  contradiction.

  In the remaining four cases, consider the green edge below:
  \begin{align*}
\begin{array}{c}
  \begin{tikzpicture}[scale=0.45]
    \node[red] (t) at (90:2) {$u$};
    \node[blue] (l) at (180:2) {$v$};
    \node (r) at (0:2) {$\bullet$};
    \node[blue] (b) at (-90:2) {$w$};
  \draw[->] (t) to (l); 
  \draw[->] (t) to (r); 
  \draw[->] (l) to (b);
  \draw[green,->] (r) to (b);
\end{tikzpicture}
\end{array}, \;
\begin{array}{c}
  \begin{tikzpicture}[scale=0.5]
    \node[red] (t) at (90:2) {$u$};
    \node[blue] (lt) at (150:2) {$v$};
    \node[blue] (lb) at (-150:2) {$w$};
        \node (b) at (-90:2) {$\bullet$};
  \draw[->] (t) to (lt); 
  \draw[->] (lt) to (lb); 
  \draw[->] (lb) to (b);
  \draw[green,->] (t) to (b);
\end{tikzpicture}
\end{array}, \;
\begin{array}{c}
  \begin{tikzpicture}[scale=0.5]
    \node (t) at (90:2) {$\bullet$};
    \node[red] (lt) at (150:2) {$u$};
    \node[blue] (lb) at (-150:2) {$v$};
        \node[blue] (b) at (-90:2) {$w$};
  \draw[->] (t) to (lt); 
  \draw[->] (lt) to (lb); 
  \draw[->] (lb) to (b);
  \draw[green,->] (t) to (b);
\end{tikzpicture}
\end{array}, \;
    \begin{array}{c}
  \begin{tikzpicture}[scale=0.5]
    \node (t) at (90:2) {$\bullet$};
    \node[red] (lt) at (150:2) {$u$};
    \node[blue] (lb) at (-150:2) {$v$};
    \node (rt) at (30:2) {$\bullet$};
    \node (rb) at (-30:2) {$\bullet$};
        \node[blue] (b) at (-90:2) {$w$};
  \draw[->] (t) to (lt); \draw[->] (t) to (rt);
  \draw[->] (lt) to (lb); \draw[->,gray!50!white] (lt) to (rb);
  \draw[->,gray!50!white] (rt) to (lb); \draw[->] (rt) to (rb);
  \draw[->] (lb) to (b); \draw[green,->] (rb) to (b);
  \draw[->,gray!50!white] (t) to (b);
\end{tikzpicture}
                \end{array}.
  \end{align*}
  We claim that the source of each green arrow is red: in the first
  and last cases this follows from diamond completeness of the blue
  vertices; in the second case this is clear; in the third case this
  follows because the blue nodes are closed under Bruhat order.

  Now the green arrow gives us a new edge $u' \to v'$ of our polytope,
  with $u'$ blue, $v'$ red, and $v' < v$. Continuing in this way we
  eventually produce an edge $u'' \to x$ with $u''$ red, and we deduce
  that $p \notin S_{x,y} \cap I$ as above (i.e. because then the $u''$-Pl\"ucker
  coordinate is non-zero), and we are done.
\end{proof}


\newcommand{\etalchar}[1]{$^{#1}$}
\def\cprime{$'$} \def\cprime{$'$} \def\cprime{$'$}

\end{document}